\documentclass[12pt]{amsart}
\usepackage{amsmath}
\usepackage{amssymb}
\usepackage[all]{xy}
\usepackage{epsfig}
\usepackage{color}
\usepackage{mathabx}
\usepackage{tikz}
\usepackage[hidelinks]{hyperref}
\usetikzlibrary{arrows,calc,matrix}
\usepackage[english]{babel}
\usepackage{mathtools}

\newcommand{\pdim}{M}
\newcommand{\qdim}{N}
\newcommand{\dimsum}{D}
\newcommand{\block}{I}
\newcommand{\UR}{\mathrm{UR}}
\newcommand{\bfalpha}{{\boldsymbol{\alpha}}}
\newcommand{\bfbeta}{{\boldsymbol{\beta}}}
\newcommand{\bfgamma}{{\boldsymbol{\gamma}}}
\newcommand{\bfdelta}{{\boldsymbol{\delta}}}
\renewcommand{\emptyset}{\diameter}
\newcommand{\Hmu}{H}
\newcommand{\para}{P}
\newcommand{\HBQ}{\Hmu}
\newcommand{\Tr}{\mathrm{Tr}}
\newcommand{\subsp}{L}
\newcommand{\Matrix}{\bfdelta}
\renewcommand{\subset}{\subseteq}
\renewcommand{\supset}{\supseteq}
\newcommand{\bfa}{\mathbf a}
\newcommand{\Prob}{\mathrm{Prob}}
\newcommand{\id}{1}

\newcommand{\pp}{\mathbf{p}}
\newcommand{\qq}{\mathbf{q}}
\newcommand{\KK}{\mathcal{K}}
\newcommand{\altb}{d}

\font\sn = cmssi8 scaled \magstep0
\font\si = cmti8 scaled \magstep0
\long\def\comdavid#1{\ifdraft{{\color{blue}\si #1 }}\else\ignorespaces\fi}
\long\def\combarak#1{\ifdraft{\color{red}\sn #1 }\else\ignorespaces\fi}

\newcommand{\tendsto}[1]{\xrightarrow[#1]{}}
\newcommand\dee{\mathrm{d}}
\newif\ifdraft\drafttrue
\draftfalse
\newcommand\name[1]{\label{#1}{\ifdraft{\sn [#1]}\else\ignorespaces\fi}}

\newcommand\eq[2]{{\ifdraft{\ \tt [#1]}\else\ignorespaces\fi}\begin{equation}\label{#1}{#2}\end{equation}}
\newcommand {\equ}[1]{\eqref{#1}}

\newcommand{\MM}{{\mathcal{M}}}
\newcommand{\HH}{{\mathcal{H}}}
\newcommand{\HHH}{{\mathbb{H}}}

\newcommand{\Q}{{\mathbb {Q}}}

\newcommand{\R}{{\mathbb{R}}}
\newcommand{\C}{{\mathbb{C}}}

\newcommand{\Z}{{\mathbb{Z}}}

\newcommand{\BB}{{\mathcal{B}}}
\newcommand{\Gr}{{\mathrm{Gr}}}
\newcommand{\B}{{B}}

\newcommand{\E}{{\mathbb{E}}}

\newcommand{\N}{{\mathbb{N}}}
\newcommand{\Hh}{{\mathbb{H}}}

\newcommand{\given}{{|}}

\newcommand{\cl}{\overline}

\newcommand{\Ad}{{\operatorname{Ad}}}

\newcommand{\GL}{\operatorname{GL}}

\newcommand{\Stab}{\operatorname{Stab}}

\newcommand{\wbar}{\overline}

\newcommand{\SL}{\operatorname{SL}}

\newcommand{\SO}{\operatorname{SO}}

\newcommand{\PGL}{\operatorname{PGL}}

\newcommand{\GO}{\operatorname{O}}

\newcommand{\PO}{\operatorname{PO}}

\newcommand{\zero}{\mathrm{o}}

\newcommand{\Lie}{\operatorname{Lie}}

\newcommand {\ignore}[1] {}

\newcommand{\spa}{{\rm span}}

\newcommand{\dist}{{\rm dist}}

\newcommand{\LL}{{\mathcal L}}

\newcommand{\df}{{\, \stackrel{\mathrm{def}}{=}\, }}

\newcommand{\FF}{{\mathcal{F}}}

\newcommand{\supp}{{\rm supp}}

\newcommand{\sm}{\smallsetminus}

\newcommand{\vre}{\varepsilon}

\newcommand{\Hom}{\operatorname{\Hom}}

\newcommand{\BA}{{\mathrm{BA}}}
\newcommand{\DI}{{\mathrm{DI}}}
\newcommand{\VWA}{{\mathrm{VWA}}}
\newcommand{\Isom}{{\mathrm{Isom}}}

\newtheorem{thm}{Theorem}[section]

\newtheorem{lem}[thm]{Lemma}

\newtheorem{prop}[thm]{Proposition}

\newtheorem{cor}[thm]{Corollary}

\theoremstyle{definition}

\newtheorem{remark}[thm]{Remark}

\newtheorem{example}[thm]{Example}

\newtheorem{dfn}[thm]{Definition}

\begin{document}
\title[Random walks and Diophantine approximation]{Random walks on homogeneous spaces and Diophantine approximation on fractals}

\author{David Simmons}
\address{University of York, Department of Mathematics, Heslington, York YO10 5DD, UK}
\email{David.Simmons@york.ac.uk}
\urladdr{\url{https://sites.google.com/site/davidsimmonsmath/}}

\author{Barak Weiss}
\address{Tel Aviv University, Tel Aviv Israel}
\email{barakw@post.tau.ac.il}


\begin{abstract}
We extend results of Y. Benoist and J.-F. Quint concerning random
walks on homogeneous spaces of simple Lie groups to the case where the
measure defining the random walk generates a semigroup which is not
necessarily Zariski dense, but satisfies some expansion properties for
the adjoint action. Using these dynamical results,
we study Diophantine properties of typical points on some self-similar
fractals in $\R^d$. As examples, we show that for any self-similar
fractal $\KK \subset \R^d$ satisfying the open set condition (for
instance any translate or dilate of Cantor's middle thirds set or of a
Koch snowflake), almost every point with respect to the natural
measure on $\KK$ is not badly approximable. Furthermore, almost every
point on the fractal is of generic type, which means (in the
one-dimensional case) that its continued fraction expansion contains
all finite words with the frequencies predicted by the Gauss
measure. We prove analogous results for matrix approximation, and for
the case of fractals defined by M\"obius transformations. 
\end{abstract}
\maketitle

\section{Overview}
\label{sectionintro}
The purpose of this paper is twofold: to prove new results about
random walks on homogeneous spaces, and to apply these results, as
well as previously known results, to questions about the Diophantine
properties of typical points on various fractals. In this section we
state and discuss illustrative special cases of our results,
postponing the most general statements, and postponing as well the
definitions of the terms appearing in the theorems. 

\begin{thm}\name{thm: illustrative measures}
Let $t \geq 2$ and $d \geq 1$ be integers, let $G = \SL_{d+1}(\R), \,
\Lambda = \SL_{d+1}(\Z)$, and $X = G/\Lambda$, and let $m$ be the
$G$-invariant probability measure on $X$ derived from Haar measure on
$G$. For each $i=1,\ldots,t$, fix $c_i >1$, $\mathbf{y}_i \in \R^d$,
and $O_i \in \SO_d(\R)$, and let 
\[
h_i
= \left[\begin{matrix}
c_i O_i & \mathbf{y}_i \\
0 & c_i^{-d}
\end{matrix} \right]
\in G
\ \ \ \ \
(i=1, \ldots, t).
\]
Assume that $\mathbf y_1 = 0$ and that the vectors $\mathbf
y_2,\ldots,\mathbf y_t$ span $\R^d$. Fix $p_1, \ldots, p_t > 0$ with
$p_1 + \cdots + p_t = 1$, and let $\mu = \sum_{i=1}^t p_i \delta_i$
(where $\delta_i$ is the Dirac mass on $E \df \{1,\ldots,t\}$ centered
at $i$). Then for any $x\in X$ and for $\mu^{\otimes \N}$-a.e. $(i_1,
i_2, \ldots) \in E^\N$, the sequence 
\[
\{h_{i_n}\cdots h_{i_1}x : n \in \N\}
\]
is equidistributed in $X$ with respect to $m$; i.e. the sampling
measures $\frac{1}{N} \sum_{n=0}^{N-1}\delta_{h_{i_n} \cdots
  h_{i_1}x}$ converge to $m$ as $N \to \infty$ in the weak-*
topology. 
\end{thm}

Theorem \ref{thm: illustrative measures} is modeled on groundbreaking
work of Yves Benoist and Jean-Fran\c cois Quint. In
\cite{BenoistQuint7}, they obtained the same conclusion under the
assumption that the Zariski closure $\Hmu$ of the group generated by
$\supp(\mu)$ coincides with $G$, whereas in Theorem \ref{thm:
  illustrative measures} $\Hmu$ is not semisimple and could be
solvable Following their strategy,
and using many of their results, we first show that $m$ is the unique
behavior of almost every random path, starting at an arbitrary initial
point $x$. Theorem \ref{thm: illustrative measures} is a special case
of one of our main results on random walks on homogeneous spaces,
namely Theorem \ref{thm: new main}. In contrast to the work of
Benoist--Quint as well as earlier work in this domain, the hypotheses
of these theorems involve expansion properties for the adjoint action
of elements of $\supp(\mu)$. These properties cannot be detected
solely from algebraic properties of the group $\Hmu$. 

We use these results to study a question which has attracted
considerable attention recently: understanding the Diophantine
properties of a typical point on a fractal. Regarding this, we have
the following: 

\begin{thm}\name{thm: illustrative application}
Let $\KK \subset \R^d$ be the limit set of an irreducible finite
system of contracting similarity maps satisfying the open set
condition, let $s = \dim_H(\KK)$, and let $\mu_{\KK}$ denote the
restriction to $\KK$ of $s$-dimensional Hausdorff measure. Then
$\mu_{\KK}$-a.e. $\alpha\in\KK$ is not badly approximable, and is
moreover of generic type. 
\end{thm}

The class of fractals appearing in Theorem \ref{thm: illustrative
  application} contains such standard examples of self-similar sets as
Cantor's middle thirds set (or any image of it under an affine map),
the Koch snowflake, the Sierpi\'nski triangle, etc. Regarding these
and more general fractals, and natural measures supported on them, it
was previously established that they give zero measure to the set of
very well approximable numbers/vectors but contain many (in the sense
of Hausdorff dimension) badly approximable points. The measure of the
set of badly approximable points in such sets was considered by
Einsiedler, Fishman, and Shapira \cite{EFS}. They showed among other
things that in case $\KK$ is Cantor's middle thirds set,
$\mu_{\KK}$-a.e. $\alpha \in \KK$ is not badly approximable. They used
the invariance of $\KK$ under the $\times 3$ map and their proof
relied on deep dynamical results of Lindenstrauss
\cite{Lindenstrauss2}. Our proof relies on the self-similar structure
of $\KK$, and improves on \cite{EFS} in several respects: by
establishing that $\alpha$ is typically of generic type, and by
extending the result to a general class of fractals in every
dimension. 

The fractals in Theorem \ref{thm: illustrative application} are limit
sets of iterated function systems (IFSes) consisting of {\em
  similarities} $\R^d \to \R^d$. By employing directly results of
Benoist and Quint we are also able to prove similar results for
fractals which are limit sets of IFSes of {\em M\"obius
  transformations}, with the difference that the usual notions of
Diophantine approximation are replaced by analogous notions for
Diophantine approximation with respect to a Kleinian group. We are
also able to treat measures supported on fractals other than the
Hausdorff measures, and to discuss additional Diophantine properties,
including the setup of matrix Diophantine approximation, Dirichlet
improvability, intrinsic approximation on spheres, and more.

The paper is divided into two parts. In the first we establish our
results for random walks on homogeneous spaces, and in the second we
apply these results to Diophantine approximation. The first part is
completely independent of the second part but relies heavily on work
of many authors, and in particular on the work of Benoist and
Quint. The second part can be read independently of the first,
provided one is willing to accept three dynamical results: Theorems
\ref{theorempart1} and \ref{theoremequibootstrap}, which are proven in
Part \ref{part1}, and prior results of Benoist and Quint, summarized
as Theorem \ref{theoremBQ}.

\medskip

{\bf Acknowledgements.} The first-named author was supported in part by the EPSRC Programme Grant EP/J018260/1. The second-named author was supported by ERC starter grant DLGAPS 279893. The authors are grateful to Yves Benoist and Jean-Fran\c cois Quint for useful discussions, and to Alex Eskin for useful comments and encouraging remarks.

\tableofcontents

\part{Random walks on homogeneous spaces}
\label{part1}
\section{Main results -- Stationary measures and random walks}
Let $\mu$ be a probability measure on a group $G$. A measure $\nu$ on
a $G$-space $X$ is called {\em $\mu$-stationary} if $\int_G g_*\nu\,
\dee \mu(g) = \nu$. Clearly, every $G$-invariant measure $\nu$ is
$\mu$-stationary for every probability measure $\mu$ on $G$. For a
general action of a group on a compact space, invariant measures need
not exist, but $\mu$-stationary measures always exist. An
understanding of all the stationary measures for an action leads to a
very detailed understanding of the action (see
e.g. \cite{Furstenberg_random_products, Furstenberg_stiffness,
  FurstenbergKifer}). This is most easily seen when there is a unique
stationary probability measure. Our main result identifies some measures $\mu$ on
$G$ for which there is a unique stationary probability measure on a homogeneous
space $X = G/\Lambda$, and describes the random paths starting from an
arbitrary point. 

We need some notation, which will be used throughout the paper. Let
$G$ be a unimodular noncompact Lie group with finitely many connected
components, let $E \subset G$ be compact, and let $\mu$ be a compactly
supported probability measure on $G$ such that $\supp(\mu) = E$. We will sometimes
think of $E$ as an abstract indexing set for elements of $G$, in which
case we will think of $\mu$ as a measure on $E$ and write $e \mapsto
g_e$ for the inclusion map from $E$ to $G$. Let $\Gamma$ and
$\Gamma^+$ denote respectively the subgroup and subsemigroup of $G$
generated by $E$. If $\Gamma_1, \Gamma_2$ are two subgroups of $G$, we
say that $\Gamma_1$ is \emph{virtually contained} in $\Gamma_2$ if $\Gamma_1
\cap \Gamma_2$ is of finite index in $\Gamma_1$. Let $\B$ denote the
infinite Cartesian power $E^\N$, 
and let $\beta$ denote the Bernoulli measure $\mu^{\otimes \N}$. For
each $b = (b_1, \ldots) \in \B$, let $b_1^n$ denote the finite word
$(b_1, \ldots, b_n)$ and write 
\begin{equation}
\label{randomwalkorder}
g_{b_1^n} = g_{b_n} \cdots g_{b_1}.
\end{equation}
Let $\mu^{*n}$ denote the measure on $G$ obtained as the
pushforward of the measure $\mu^{\otimes n} = \mu \otimes \cdots \otimes \mu$ on $E^n$
under the map $b_1^n \mapsto g_{b_1^n}$. Let
$V = \Lie(G)$ be the Lie algebra of $G$, let $\SL^\pm(V)$ be the group
of linear automorphisms of $V$ with determinant $\pm 1$, let $\Ad: G \to
\SL^\pm(V)$ be the adjoint representation, and for each $d = 1,
\ldots, \dim G -1$ let $V^{\wedge d} = \bigwedge^d V$ and let $\rho_d:
G \to \SL^\pm(V^{\wedge d})$ be the $d$-th exterior power of $\Ad$. We
say that two subspaces $V_1, V_2$ of $V^{\wedge d}$ are {\em
  complementary} if $V^{\wedge d} = V_1 + V_2$ and $V_1\cap V_2 =
\{0\}$. In \S \ref{sec: positivity}, following Oseledec, we will
define a {\em subspace of non-maximal expansion}, to be denoted by
$V_b^{<\max}$, and a {\em subspace of subexponential expansion}, to be
denoted by $V_b^{\leq 0}$. These are subspaces of $V$ and of
$V^{\wedge d}$ respectively, defined for $\beta$-a.e. $b\in\B$, and
depending measurably on $b$. 

\begin{thm}\name{thm: new main}
Let $G$, $\mu$, and $\rho_d: G \to \SL^\pm(V^{\wedge d})$ be as above,
and suppose that the identity component of $G$ is simple. Let
$\Lambda$ be a lattice in $G$, let $X = G/\Lambda$, and let $m_X$ be
the $G$-invariant probability measure on $X$ induced by Haar measure
on $G$. Suppose that $\Gamma$
acts transitively on the connected components of $X$, and that
$\Gamma$ is not virtually contained in a conjugate of
$\Lambda$.
Assume that for each $d = 1, \ldots, \dim G-1$, there is a nontrivial
proper subspace $W^{\wedge d} \subsetneqq V^{\wedge d}$ such that the following hold: 
\begin{itemize}
\item[\textup{(I)}]
For every $g \in \supp(\mu)$, $W^{\wedge d}$ is
$\rho_d(g)$-invariant. For $\beta$-a.e. $b\in\B$, if $d=1$ then
$W^{\wedge d}$ is complementary to $V_b^{<\max}$ and if $d>1$, then
$V_b^{\leq 0} \cap W^{\wedge d} = \{0\}$. 

\item[\textup{(II)}] For every $g \in \supp(\mu)$, $\Ad(g)$ acts on
  $W=W^{\wedge 1}$ as a similarity map (with respect to some fixed
  inner product on $W$), and 
\[
\int_G \log \|\Ad(g)|_W \|\, \dee \mu (g) >0.
\]

\item[\textup{(III)}] For any $d$, if a linear subspace $\subsp
  \subset V^{\wedge d}$ has a 
  finite orbit under the semigroup generated by $\supp(\mu)$, then $\subsp
  \cap W^{\wedge d} \neq \{0\}$. 
\end{itemize}
Then:
\begin{itemize}
\item[(i)]
The only $\mu$-stationary probability measure on $X$ is $m_X$.
\item[(ii)]
For any $x \in X$, for $\beta$-almost every $b \in \B$, the sequence
$
(g_{b_1^n} x)_{n\in\N}
$
is equidistributed with respect to $m_X$.
\end{itemize}
\end{thm}

Theorem \ref{thm: new main} is modeled on results of Benoist and
Quint. Namely, conclusion (i) is obtained in
\cite[Theorem~1.1]{BenoistQuint6} and conclusion (ii) is obtained in
\cite[Theorem~1.3]{BenoistQuint7} under the assumption that the
Zariski closure $\Hmu$ of $\Gamma$ is semisimple with no compact
factors. Our proof of Theorem \ref{thm: new main} relies heavily on
arguments introduced by Benoist and Quint. 

Despite the very similar approaches, we do not assume that $\Hmu$ is
semisimple, but instead introduce assumptions \textup{(I)--(III)}. As
we will see in \S \ref{sec: positivity}, these assumptions imply that
for any $v \in V$, for almost any $b \in \B$, the random sequence of
vectors $(\Ad(g_{b_1^n})v)_{n \ge 1}$ become longer and longer (at a
rate independent of $v$) and are attracted projectively to $W \df
W^{\wedge 1}$ as $n \to \infty$. In other words, $W$ plays the role of
a ``subspace of maximal expansion'' to which all trajectories get
attracted. This crucial observation makes it possible to employ the
``exponential drift'' argument of Benoist and Quint and conclude that
any stationary measure $\nu$ is invariant under a subgroup of $W$. We
note that in our work $W$ is a {\em deterministic} subspace, whereas
the subspace which plays a similar role in the arguments of Benoist
and Quint (which they denote by $V_b$) is a random subspace depending
on $b$. 

In the main application of interest in Part \ref{part2}, the group $\Hmu$ which
will appear will not be semisimple, and assumptions
\textup{(I)--(III)} will be satisfied. In fact, \textup{(I)--(III)}
can never be satisfied when $\Hmu$ is semisimple. On the other hand,
conditions \textup{(I)--(III)} do not depend only on $\Hmu$, but also
on the decomposition of $V$ into expanding and contracting spaces for
the transformations $\Ad(g) \;  (g \in \supp(\mu))$. It is possible
(e.g. by adapting \cite[\S 3.5]{BenoistQuint5}) to construct examples
of measures $\mu$ for which the group $\Hmu$ is solvable and for which
both conclusions of Theorem \ref{thm: new main} fail. 

Alex Eskin and Elon Lindenstrauss have recently announced a
far-reaching extension of the work of Benoist and Quint, which implies
Theorem \ref{thm: new main}(i). 

We will also need a result which extends the second conclusion of
Theorem \ref{thm: new main} to certain fiber bundles over $X$. In the following theorem $\bar{\B} = E^\Z$, $\bar{\beta}$ is the Bernoulli measure
$\mu^{\otimes \Z}$ on $\bar{\B}$, and $T: \bar{\B} \to \bar{\B}$ is
the shift map.

\begin{thm}\name{thm: fiber bundle extension}
Let $G$ be a unimodular connected Lie group, let $\Lambda$ be a
lattice in $G$, let $X = G/\Lambda$, and let $m_X$ be the unique
$G$-invariant probability measure on $X$. Let $\mu$ be a compactly supported
probability measure on $G$, let $E = \supp(\mu)$, and let $\B, \beta,\, \Gamma$ be
as above. Fix $x \in X$ and suppose that for $\beta$-a.e. $b\in\B$,
the sequence $(g_{b_1^n}x)_{n \in \N}$ is equidistributed with respect
to $m_X$. Let $K$ be a compact group, let $m_K$ be Haar measure on
$K$, and let $\kappa: \Gamma \to K$ be a homomorphism 
such that the $\Gamma$-action $\gamma(x, k) = (\gamma x ,
\kappa(\gamma)k)$ on $X \times K$ is ergodic with respect to $m_X
\otimes m_K$. Let $Y$ be a locally compact metric space, $f : \bar{\B}
\to Y$ a measurable map, and $m_Y = f_* \bar{\beta}$. 

Then for any $x \in X$, for $\bar{\beta}$-a.e. $b\in \bar{\B}$, the sequence
\[
\big(g_{b_1^n} x, \kappa(g_{b_1^n}), f(T^n b) \big)_{n\in\N}
\]
is equidistributed with respect to the measure $m_X \otimes m_K \otimes m_Y $ on $X \times K \times Y$.
\end{thm}

\section{Random matrix products for semigroups, and positivity}
\name{sec: positivity}
Throughout this section we keep the notation and assumptions of Theorem \ref{thm: new main}. Our goal will be to describe some consequences of hypotheses \textup{(I)--(III)}.
%
We will need more notation. For each $d = 1,\ldots, \dim G - 1$, we fix an inner product on the vector space $V^{\wedge d}$ and use it to define a metric on $V^{\wedge d}$ and an operator norm on $\GL(V^{\wedge d})$. We denote the projective space of lines in $V^{\wedge d}$ by $\mathbb{P}(V^{\wedge d})$, and the Grassmannian space of $k$-dimensional subspaces by $\Gr_k(V^{\wedge d})$. The element of $\mathbb{P}(V^{\wedge d})$ corresponding to a point $x \in V^{\wedge d} \sm \{0\}$ will be denoted by $[x]$, and the image of a nonzero subspace $W \subset V^{\wedge d}$ in $\mathbb{P}(V^{\wedge d})$ will be denoted by $[W]$. We will denote the distance between a vector $v\in V^{\wedge d}$ and a subspace $W\subset V^{\wedge d}$ by $\dist(v,W)$, and the distance between their projectivizations by $\dist([v],[W])$. In the latter case the distance can be measured with respect to any metric on $\mathbb P(V^{\wedge d})$ which induces the standard topology. This should cause at worst mild confusion.

The main results of this section are the following three
statements. The first should be compared to
\cite[Corollary~5.5]{BenoistQuint3}, the second to
\cite[Lemma~6.8]{BenoistQuint3}, and the third to
\cite[Lemma~4.1]{EskinMargulis}, where the same conclusions are
obtained under different hypotheses.

\begin{prop} \name{prop: cor 5.5}
Under assumptions \textup{(I)--(III)}, we have:
\begin{itemize}
\item[a)]
For every $\alpha>0$, there exist $c_0>0$, $q_0 \geq 1$ such that for
any $v \in V\sm \{0\}$, we have 
\[
\beta \left( \left\{b \in \B: \forall q \geq q_0, \;
\|\Ad(g_{b_1^q})v\| \geq
c_0\|\Ad(g_{b_1^q})\| \; \|v\| \right \} \right)\geq 1-\alpha.
\]
\item[b)]
For every $\alpha > 0$ and $\eta>0$, there exists $q_0 \geq 1$ such
that for any $v \in V\sm \{0\}$, we have 
\[
\beta \left(\left\{b \in \B: \forall q \geq q_0, \; \dist \left([
\Ad(g_{b_1^q}) v], 
[W] \right ) \leq \eta
\right \} \right) \geq 1-\alpha.
\]
\end{itemize}
\end{prop}

\begin{prop}\name{prop: stationary algebraic}
Under assumptions \textup{(I)} and \textup{(III)}, for each $d =
1,\ldots,\dim(G) - 1$, the only $\mu$-stationary probability measure on $V^{\wedge
  d}$ is the Dirac measure $\delta_0$ centered at $0$. 
\end{prop}

\begin{prop}\name{prop: positivity}
Under assumptions \textup{(I)--(III)}, there exist $n_0\in\N$ and
$\vre > 0$ such that for all $d$, $v\in V^{\wedge d} \sm \{0\}$, and
$n \geq n_0$, we have 
\eq{eq: loginequality2}{
\frac{1}{n} \int_G \log\frac{\|\rho_d(g)v\|}{\|v\|}\, \dee \mu^{*n}(g) > \vre.
}
\end{prop}


We recall the following: 

\begin{thm}[Oseledec, \cite{Oseledec}]
\label{theoremoseledets}
Let $G,\mu$ be as above, let $V$ be a vector space, and let $\rho:G\to
\GL(V)$ be an action. Then there exist $k \in \N$, numbers $\chi_1 >
\cdots > \chi_k$ (called \underline{Lyapunov exponents}), and a
measurable map which assigns to $\beta$-a.e. $b \in \B$ a descending
chain of subspaces (called \underline{Oseledec subspaces}) 
\[
V = V_0 \supsetneqq V_1(b)
\supsetneqq \cdots \supsetneqq V_{k-1}(b) \supsetneqq V_k=
\{ 0\},
\]
such that for all $i = 1,\ldots,k$ and $v\in V_{i-1}(b)\sm V_i(b)$,
\eq{eq: lyapunov def}{
\lim_{n\to\infty} \frac{\log\|\rho(g_{b_1^n} )v\|}{n} = \chi_i.
}
The convergence in \eqref{eq: lyapunov def} is uniform as $v$ ranges over any compact subset of $V_{i-1}(b)\sm V_i(b)$. Furthermore,
\eq{eq: sum zero}{
\sum_{i=1}^k d_i\chi_i =\int_G \log |\det (\rho(g))| \; \dee \mu,
}
where $d_i = \dim V_{i-1} -\dim V_{i}$, and for $\beta$-a.e. $b\in\B$, for all $i$, we have
\eq{eq: equivariance}{
V_i(T(b)) =
\rho(g_{b_1})V_i(b).
}
\end{thm}

In the sequel, we will denote the subspace $V_1(b)$ from Theorem
\ref{theoremoseledets} by $V^{< \max}_b$. We will call it the {\em
  Oseledec space of non-maximal expansion}. Similarly, if $j_0 =
\max\{j = 0,\ldots,k : \chi_j > 0\}$, then we will denote the
Oseledec subspace $V_{j_0}(b)$ by $V^{\leq 0}_b$, and we will call it
the {\em Oseledec space of subexponential expansion}.

Fix $d = 1,\ldots,\dim(G) - 1$, and consider the special case of
Theorem \ref{theoremoseledets} occuring when $V = V^{\wedge d}$ and
$\rho = \rho_d$. Note that since $\rho_d(G) \subset \SL^\pm(V)$,
\eqref{eq: sum zero} implies that $\sum_{i = 1}^k d_i \chi_i = 0$. On
the other hand, since the space $W^{\wedge d}$ is proper and
invariant, assumption (I) guarantees that $\chi_1 > 0$, from
which it follows that $\chi_k < 0$ and $k \geq 2$. In particular we have $\{0\} \neq
V^{\leq 0}_b \subset V^{< \max}_b \subsetneqq V^{\wedge d}$. 

\begin{prop}\name{prop: under assumptions}
Under assumptions \textup{(I)} and \textup{(II)}, for $d=1$, for
$\beta$-a.e. $b\in\B$, for any compact set $C \subset V \sm V^{<
  \max}_b$ there exists $c>0$ such that for all $v \in C$ and all
$n\in\N$, we have 
\[
\|\Ad(g_{b_1^n}) v\|\geq c \|\Ad(g_{b_1^n})\| .
\]
\end{prop}

\begin{proof}
We will write $A \asymp_\times B$ if $A, B$ are two quantities satisfying $c^{-1} \leq \frac{A}{B} \leq  c$ for some constant $c>1$ depending only on $G$ and $\mu$. If $c$ (the {\em implicit constant}) depends on an additional parameter $p$ we will write $A \asymp_{\times,p} B$.

Fix $v\in V \sm V_b^{<\max}$. By assumption \textup{(I)}, we can write $v = \pi_1(v) + \pi_W(v)$, where $\pi_1(v) \in V^{< \max}_b$ and $\pi_W(v) \in W \sm \{0\}$. Then by Theorem \ref{theoremoseledets}, we have
\[
\frac{\|\Ad(g_{b_1^n})
\pi_1(v)\|}{\|\Ad(g_{b_1^n}) \pi_W(v)\|} \tendsto{n \to \infty} 0
\]
and \, thus \, $\|\Ad(g_{b_1^n}) v\| \asymp_{\times, b, v} \|\Ad(g_{b_1^n})
\pi_W(v)\|$. \, Moreover, \, by \, assumption \, \textup{(II)} \, we \, have
$\|\Ad(g_{b_1^n}) \pi_W(v)\| \asymp_{\times, v}
\|\Ad(g_{b_1^n})|_W\|.$ In both cases the implicit constant can be
taken to be uniform for $v$ in a compact subset of $V \sm V^{<
  \max}_b$. Choose a basis $\{e_i\}_{i=1}^{\dim V}$ of $V$ consisting
of elements which do not belong to $V^{< \max}_b$. By the same logic,
we have $\|\Ad(g_{b_1^n}) e_i\| \asymp_{\times, b}
\|\Ad(g_{b_1^n})|_W\|$ for each $i$. Thus $\|\Ad(g_{b_1^n})v\|
\asymp_{\times,b,v} \|\Ad(g_{b_1^n})|_W\| \asymp_{\times, b}
\|\Ad(g_{b_1^n})\|$, where for each fixed $b$, the implicit constant
is uniform on compact subsets of $V \sm V^{< \max}_b$. 
\end{proof}

\begin{prop}\name{prop: tending to W}
Under assumptions \textup{(I)} and \textup{(II)}, for $d=1$, for $\beta$-a.e. $b \in \B$, for all $v \in V \sm V^{< \max}_b$, we have
\begin{equation}\label{eq: first assertion}
\frac{\dist \left(\Ad(g_{b_1^n})v, W \right) }{\|\Ad(g_{b_1^n})\|} \tendsto{n \to \infty} 0
\end{equation}
and hence
\begin{equation}\label{eq: second assertion}
\dist \left([\Ad(g_{b_1^n})v], [W] \right) \tendsto{n \to \infty} 0.
\end{equation}
For fixed $b$, the convergence is uniform for $v$ in a compact subset
of $V \sm V^{< \max}_b$. 
\end{prop}

\begin{proof}
By assumption \textup{(I)}, we can choose $w\in W$ such that $v-w\in V^{<
  \max}_b$. Again by \textup{(I)}, we have $\Ad(g_{b_1^n}v)w \in W$
for all $n$. Thus for any $0<\vre < \chi_1-\chi_2$, we have 
\[
\begin{split}
\dist(\Ad(g_{b_1^n})v, W) & \leq
\|\Ad(g_{b_1^n}) v - \Ad(g_{b_1^n} )w\| = \|\Ad(g_{b_1^n})(v-w)\| \\
& =
O\left(e^{n(\chi_2 + \vre)} \right) = 
o\left(\|\Ad(g_{b_1^n})\| \right).
\end{split}
\]
This establishes \eqref{eq: first assertion}. Equation \eqref{eq:
  second assertion} and the final assertion follow from combining with
Proposition \ref{prop: under assumptions}. 
\end{proof}

\begin{prop}
\name{propositionV1empty}
Assume that \textup{(I)} and \textup{(III)} hold, and fix $d = 1,\ldots,\dim(G) - 1$ and $v\in V^{\wedge d} \sm \{0\}$. Then we have $v\notin V^{\leq 0}_b$ for $\beta$-a.e. $b \in \B$, and if $d=1$ then $v \notin V^{< \max}_b$ for $\beta$-a.e. $b \in \B$.
\end{prop}

\begin{proof}
The proofs for $d=1$ and $d>1$ are identical, exchanging everywhere
$V^{\leq 0}$ for $V^{<\max}$ and $\rho_d$ for $\Ad$. For concreteness
we prove the assertion for $d=1$. Fix $v\in V \sm \{0\}$, and let
$\mu^{\ast i} * \delta_{[v]}$ denote the pushforward of $\mu^{\otimes
  i}$ under the map $b_1^i \mapsto [\Ad(g_{b_1^i})v]$, or equivalently
the pushforward of $\mu^{\ast i}\otimes \delta_{[v]}$ under the map
$(g,[v]) \mapsto [\Ad(g)v]$. For each $N\geq 1$, let 
\[
\nu_N = \frac{1}{N} \sum_{i = 0}^{N - 1} \mu^{\ast i}\ast \delta_{[v]},
\]
which is a probability measure on the compact space
$\mathbb{P}(V^{})$. By the equivariance property \equ{eq:
  equivariance}, for all $n$ and $b^n_1 \in E^n$, for
$\beta$-a.e. $b' \in \B$ we have 
\[
\Ad(g_{b^n_1} ) v \in V^{< \max}_{b'}\ \Longleftrightarrow \ v \in
V^{< \max}_{b_1^n b'}.
\]
A straightforward induction and Fubini's theorem imply that for all $i \geq 0$, we have
\[
\int_\B \delta_{[v]}([V^{< \max}_b]) \, \dee \beta(b) = \int_\B
\mu^{\ast i}*\delta_{[v]}([V^{< \max}_b]) \, \dee \beta(b),
\]
and hence, for all $N \geq 1$, we have
\eq{eq: clearly}{
\begin{split}
\beta\left( \left\{b \in \B : v \in V^{< \max}_{b} \right\}
\right) & = \int_{\B}
\delta_{[v]} ([V^{< \max}_{b}]) \, \dee \beta(b) \\
& = \int_\B \nu_N ([V^{< \max}_{b}]) \, \dee
\beta(b).
\end{split}
}

We need to show that \equ{eq: clearly} is zero. Applying the Lebesgue dominated convergence theorem to the functions $b \mapsto \nu_N([V^{< \max}_{b}]) \leq 1,$ it suffices to show that for $\beta$-a.e. $b\in\B$, we have $\nu_N([V^{< \max}_b]) \to_{N \to \infty} 0$. Suppose the contrary. Then there exist $\vre>0$ and a set $\B_0 \subset \B$ with $\beta(\B_0)>0$, such that for each $b \in \B_0$, there is a subsequence $N_k \to \infty$ with $\nu_{N_k}([V^{< \max}_b]) \geq \vre$. We can further assume that $\B_0$ is contained in the set of full $\beta$-measure which appears in assumption \textup{(I)}. Let $V' = V^{< \max}_{b_0}$ for some $b_0 \in \B_0$, let $(N_k)_{k\in\N}$ be the corresponding subsequence, and let $\nu_\infty$ be a weak-* limit point of the sequence $(\nu_{N_k})_{k\in\N}$. Then $\nu_\infty$ is $\mu$-stationary and satisfies $\nu_\infty([V'])\geq \vre$. According to the ergodic decomposition theorem for stationary measures (see e.g. \cite[\S 3]{FurstenbergKifer}), there is an ergodic component $\nu'_\infty$ of $\nu_\infty$ satisfying 
$\nu'_\infty([V']) > 0$. Let $k\leq \dim V^{}$ be the smallest number such that some $k$-dimensional subspace of $V^{}$ is given positive measure by $\nu'_\infty$. Then any two distinct $k$-dimensional subspaces of $V^{}$ intersect in a measure zero set, so $\nu'_\infty$ acts as an additive atomic measure on the set of all such subspaces. Since finite atomic stationary ergodic measures are supported on finite sets invariant under the semigroup, there exists a finite $\supp(\mu)$-invariant collection of subspaces $\{\subsp_1,\ldots,\subsp_r\}$ whose union contains the support of $\nu'_\infty$. Now by assumption \textup{(III)}, each of the subspaces $\subsp_i$ intersects $W^{}$ nontrivially. So by assumption \textup{(I)}, $\subsp_i\cap V' \subsetneqq \subsp_i$ is of dimension strictly less than $k$, and thus $\nu'_\infty([\subsp_i\cap V']) = 0$. So $\nu'_\infty([V']) = 0$, a contradiction.
\end{proof}

\begin{proof}[Proof of Proposition \ref{prop: positivity}]
Fix $\alpha > 0$ to be specified below. By Proposition \ref{propositionV1empty}, for each $v' \in V^{\wedge d} \sm \{0\}$ there exist $\vre_0 = \vre_0(v')$ and $B_0 = B_0(v') \subset B$ such that $\beta(B_0) \geq 1-\alpha$ and for all $b \in B_0$, $\dist([v'], [V^{\leq 0}_b]) \geq \vre$. Choose $\vre_1(v') \in (0, \vre_0(v'))$. Then there is a neighborhood $\mathcal{U} = \mathcal{U}_{v'}$ of $[v']$ in $\mathbb{P}(V^{\wedge d})$ such that for all $b \in B_0(v')$ and $v \in V^{\wedge d} \sm \{0\}$ with $[v]\in \mathcal{U}$, we have $\dist([v], [V^{\leq 0}_b]) \geq \vre_1(v').$ Since the projective space $\mathbb{P}(V^{\wedge d})$ is compact, there exist a finite cover $\{\mathcal{U}_1, \ldots, \mathcal{U}_k\}$ of $\mathbb{P}(V^{\wedge d})$, a finite collection $\{B_1, \ldots, B_k\}$ of subsets of $B$ such that $\beta(B_j) \geq 1-\alpha$ for all $j$, and $\vre_1>0$ such that for all $j = 1,\ldots,k$, $b \in B_j$, and $v \in V^{\wedge d} \sm \{0\}$ with $[v] \in \mathcal{U}_j$, we have $\dist([v], [V^{\leq 0}_b]) \geq \vre_1$.

Choose $\chi > 0$ strictly less than the smallest positive Lyapunov exponent of $V^{\wedge d}$. By the uniformity in Theorem \ref{theoremoseledets}, for each $j$ there exists $n_j$ such that for all $n \geq n_j$, $v \in V^{\wedge d} \sm \{0\}$ with $[v]\in\mathcal{U}_j$, and $b \in B_j$, we have
\[
\|\rho_d(g_{b_1^n}) v \| \geq e^{n\chi} \|v\|.
\]
Let $N = \max_j n_j$. For each $v \in V^{\wedge d} \sm \{0\}$ and $n \geq N$ let
\[
S = S_{n,v} = \{b_1^n \in E^n : \|\rho_d(g_{b_1^n}) v\| \geq e^{n\chi} \|v\|\}.
\]
Note that if $[v] \in \mathcal{U}_j$ and $b \in B_j$ then $b_1^n \in S_{n,v}$ for all $n\geq N$. Since $\beta(B_0(v_j)) \geq 1-\alpha$ we obtain that $\mu^{\otimes n}(S) \geq 1-\alpha.$
\ignore{
By Theorem \ref{theoremoseledets}, for $\beta$-a.e. $b\in\B$ we have $\lim_{n \to \infty} \frac{\log \|\rho_d(g_{b_1^n}) \|}{n} > \chi$, and hence there exists $q_1 \geq 1$ such that
\[
\beta(\{b\in B : \forall n\geq q_1, \; \|\rho_d(g_{b_1^n})\| \geq
e^{n\chi}\}) \geq 1 - \alpha.
\]
For $v\in V\sm \{0\}$, $n \in \N$ and $c>0$, let
\[
S_{n,v, c} = \{(g_1,\ldots,g_n)\in G^n : \|\rho_d(g_1\cdots g_n) v\| \geq
c e^{n\chi} \|v\|\}.
\]
It follows from Theorem \ref{theoremoseledets} and Proposition \ref{propositionV1empty} that there exist $c_0$ and $q_0 \geq q_1$ such that if $n \geq q_0$ then for all $v \in V\sm \{0\}$ we have $\mu^{\otimes n}(S_{n,v, c_0}) \geq 1 - 2\alpha$. 
}
Thus we find:
\begin{align*}
& \frac 1n\int_G \log\frac{\|\rho_d(g) v\|}{\|v\|} \;\dee \mu^{*n}(g)
= \frac 1n\int_{E^n} \log\frac{\|\rho_d(g_{b_1^n}) v\|}{\|v\|}
\;\dee \mu^{\otimes n}(b_1^n)\\
\geq & \frac 1n \int_S \log(e^{n\chi})\;\dee \mu^{\otimes n} + \frac{1}{n} \int_{E^n\sm S} \log\|\rho_d(g_{b_1^n})^{-1}\|^{-1}
\;\dee \mu^{\otimes n} (b_1^n)\\
\geq & \frac 1n \big[(1 - \alpha) n \chi-
\alpha \, n\log\max_{g\in \supp(\mu)} \|\rho_d(g)^{-1}\|\big]\\
= & (1 - \alpha)\chi - \alpha \log\max_{g\in \supp(\mu)}
\|\rho_d(g)^{-1}\|.
\end{align*}
To finish the proof, choose $\alpha$ small enough so that the last expression is a positive number independent of $v$.
\end{proof}

\begin{proof}[Proof of Proposition \ref{prop: cor 5.5}]
Fix $\alpha,\eta > 0$. By Proposition \ref{propositionV1empty} and a compactness argument similar to the one used in the proof of Proposition \ref{prop: positivity}, there exists $\vre > 0$ such that for all $v\in V\sm\{0\}$,
\[
\beta(\{b\in B : \dist([v],[V^{< \max}_b]) \geq \vre \}) \geq 1 - \alpha/2.
\]
Now for each $b\in B$, let $N(b)$ be the smallest integer with the following property: for all $v\in V$ such that $\dist([v],[V^{< \max}_b]) \geq \vre$ and for all $n\geq N(b)$, we have $\|\Ad(g_{b_1^n}) v\| \geq \frac{1}{N(b)} \|\Ad(g_{b_1^n})\| \; \|v\|$ and $\dist([\Ad(g_{b_1^n})v],[W]) \leq \eta$. Then by Propositions \ref{prop: under assumptions} and \ref{prop: tending to W}, $N(b) < \infty$ for $\beta$-a.e. $b\in B$. Therefore there exists $N_0$ such that
\[
\beta(\{b\in B : N(b) \leq N_0\}) \geq 1 - \alpha/2.
\]
Now fix $v\in V\sm\{0\}$. For all $b\in B$ such that $\dist([v],[V^{< \max}_b]) \geq \vre$ and $N(b) \leq N_0$, and for all $n \geq N_0$, we have $\|\Ad(g_{b_1^n}) v\| \geq \frac{1}{N_0} \|\Ad(g_{b_1^n})\| \; \|v\|$ and $\dist([\Ad(g_{b_1^n})v],[W]) \leq \eta$. These facts demonstrate (a) and (b) respectively.
\ignore{

Follows from Propositions \ref{prop: under assumptions}, \ref{prop: tending to W} and \ref{propositionV1empty} by Egorov's theorem and a compactness argument. See \cite[\S 5]{BenoistQuint3} for details.
}
\end{proof}

\begin{proof}[Proof of Proposition \ref{prop: stationary algebraic}]
Let $\nu$ be a $\mu$-stationary probability measure on $V^{\wedge d}$ which is not
equal to the Dirac measure $\delta_0$, let $Z = B \times V^{\wedge
  d}$, let $\lambda = \beta \otimes \nu$, and let 
\[
Y = \{(b, v) \in Z: v \notin V^{\leq 0}_b\}.
\]
According to Proposition \ref{propositionV1empty},
$\lambda(Y) = 1$. Define $\hat{T}: Z \to Z$ by $\hat{T}(b,v) = (Tb,
\rho_d(g_{b_1})v)$. Since $\nu$ is $\mu$-stationary, $\lambda$ is
$\hat{T}$-invariant. By the definition of $Y$, for every $(b,v) \in Y$
we have $\|\rho_d(g_{b_1^n})v \| \to \infty$. Let $t > 0$ be large
enough so that $\lambda(Y_0)>0$, where 
\[Y_0 = \{(b,v) \in Y: \|v\| \leq t\}.\]
Then for all $(b,v) \in Y_0$, for all $n$ large enough we have
$\hat{T}^n(b,v) \notin Y_0$, and we get a contradiction to the
Poincar\'e recurrence theorem. 
\end{proof}

The following observation will also be useful. 
\begin{prop}\name{prop: subalgebra}
Under assumptions \textup{(I)} and \textup{(II)}, the subspace $W = W^{\wedge 1}
\subseteq V = \Lie(G)$ is abelian, and in particular is a
subalgebra. 
\end{prop}

\begin{proof}
Let $b\in\B$ belong to the subset of full $\beta$-measure for which
the conclusion of Theorem \ref{theoremoseledets} holds. Denote by
$\bar{g}_{b_1^n}$ the induced action of $g_{b_1^n}$ on the quotient
space $V/W$. Then for all
large enough $n$, by assumption (I) we have 
\[
\|\bar{g}_{b_1^n}\| < \|g_{b_1^n}|_W\|
\]
and by assumption (II) we have
\[
\|g_{b_1^n}|_W\| > 1.
\]
It follows that the eigenvalues of $g_{b_1^n}$ all have modulus $\leq
\|g_{b_1^n}|_W\|$, and by assumption (II), $g_{b_1^n}|_W$ is normal
and its eigenvalues all have modulus equal to $\|g_{b_1^n}|_W\|$. Now
if $w_1,w_2\in W\otimes\C$ are eigenvectors corresponding to
eigenvalues $\lambda_1,\lambda_2$, then $[w_1,w_2]$ is either 0 or an
eigenvector with corresponding eigenvalue $\lambda_1\lambda_2$. But
since $|\lambda_1\lambda_2| = \|g_{b_1^n}|_W\|^2 > \|g_{b_1^n}|_W\|$,
the latter case is impossible, so $[w_1,w_2] = 0$. 
\end{proof}

\section{Modifying the arguments of Benoist--Quint}
In this section we will outline how to prove Theorem \ref{thm: new
  main} by adapting the arguments of Benoist and Quint. A crucial input
to the work of Benoist and Quint was some information on the action of
random matrices. We have already proved the analogous results required
in our setup in \S
\ref{sec: positivity}. The other arguments appearing in \cite{BenoistQuint3} can be easily
adapted to our new setup. There are many modifications but all of them
are minor. A self-contained treatment would have required many pages,
consisting largely of arguments due to Benoist and Quint, and hence we
will simply refer to \cite{BenoistQuint3} and take note of which parts
of \cite{BenoistQuint3} need to be modified to deal with our
setup. This will show that the conclusion of
\cite[Theorem~1.1]{BenoistQuint3} is valid in our setup, which, as we
will see, implies part (i) of our theorem. It will also show that
\cite[Lemma~6.3]{BenoistQuint3} is valid in our setup, a fact which we
will use in the proof of part (ii) of our theorem.

\begin{proof}[Proof of Theorem \ref{thm: new main}(i)]
We begin by comparing Theorem \ref{thm: new main}(i) with
\cite[Theorem~1.1]{BenoistQuint3}. The differences in the statements
of the theorems can be summarized as follows: 
\begin{itemize}
\item[1.] In \cite[Theorem~1.1]{BenoistQuint3}, it is assumed that the
  Zariski closure $\Hmu$ of $\Gamma$ is semisimple with no compact
  factors, while in Theorem \ref{thm: new main}(i), for each $d =
  1,\ldots,\dim(G) - 1$ we assume the existence of a subspace 
  $W^{\wedge d} \subset V^{\wedge d}$ satisfying (I)-(III). 
\item[2.] In \cite[Theorem~1.1]{BenoistQuint3}, it is assumed that $G$
  is connected and simple, while in Theorem \ref{thm: new main}(i), we
  assume only that the identity component of $G$ is simple and that
  $\Gamma$ acts transitively on the connected components of $X =
  G/\Lambda$. 
\item[3.] The conclusion of \cite[Theorem~1.1]{BenoistQuint3} states
  only that the only \emph{nonatomic} $\mu$-stationary probability measure is
  $m_X$, while the conclusion of Theorem \ref{thm: new main}(i) states
  that $m_X$ is the only $\mu$-stationary probability measure, meaning that there
  are no atomic $\mu$-stationary measures. However, in Theorem
  \ref{thm: new main}(i) we also assumed that $\Gamma$ is
not virtually contained in any lattice conjugate to $\Lambda$.
\end{itemize}

Regarding (3), in the context of Theorem \ref{thm: new main}(i), the assumption on $\Gamma$ implies that for all $x\in X$, the orbit $\Gamma x$ is infinite. This in turn implies that $X$ does not admit any atomic $\mu$-stationary measure. 

Regarding (2), the only place where the connectedness assumption is
used in \cite{BenoistQuint3} is in the proof of
\cite[Lemma~8.2]{BenoistQuint3}. There, it is claimed that
\cite[Proposition~6.7]{BenoistQuint3} implies (a) that $G_\alpha = G$,
but as stated, the conclusion of this proposition gives only (b) that
the Lie algebra of $G_\alpha$ is a (nontrivial) ideal in the Lie
algebra of $G$. However, under Benoist--Quint's assumption that $G$ is
connected and simple, (b) implies (a). 

Now suppose that the identity component of $G$ is simple, that
$\Gamma$ acts transitively on the connected components of $X =
G/\Lambda$, and that (b) holds. Then $G_\alpha$ contains $G_0$, the
identity component of $G$, and thus since $\alpha$ is fixed by
$G_\alpha$, it follows that $\alpha$ is a linear combination of the
$G_0$-invariant probability measures on the connected components of $X$. Now let
$\alpha'$ be the projection of $\alpha$ onto the set of connected
components of $X$. Then $\alpha'$ is $\mu$-stationary, so since a
stationary measure on a finite set is invariant, $\alpha'$ is
$\Gamma$-invariant. Since $\Gamma$ acts transitively on the connected
components of $X$, it follows that $\alpha'$ is the uniform measure
and thus that $\alpha = m_X$ and $G = G_\alpha$. Thus, the inference
from (b) to (a) is valid in our setting as well and we do not need to
assume that $G$ is connected. 

Regarding (1), the assumption that $\Hmu$ is semisimple with no
compact factors is used only in three places in \cite{BenoistQuint3}: 
\begin{itemize}
\item[1a.] Benoist and Quint refer to Furstenberg and Kesten
  \cite{FurstenbergKesten} for the proof of
  \cite[Proposition~5.2]{BenoistQuint3}. The reference
  \cite{FurstenbergKesten} assumes that $\Hmu$ is semisimple with no
  compact factors. 
\item[1b.] Benoist and Quint refer to Eskin and Margulis
  \cite{EskinMargulis} in two places in \cite[\S
  6]{BenoistQuint3}. The reference \cite{EskinMargulis} uses the
  Furstenberg--Kesten theorem on the positivity of the first Lyapunov
  exponent \cite[Lemma~4.1]{EskinMargulis}, which assumes that $\Hmu$
  is semisimple with no compact factors. \cite{EskinMargulis} also
  uses the assumption of semisimplicity directly in the proof of
  \cite[Proposition~2.7]{EskinMargulis}. 
\item[1c.] The proof of \cite[Lemma~6.8]{BenoistQuint3} refers to \cite{FurstenbergKesten} as well as using the assumption that $\Hmu$ is semisimple directly.
\end{itemize}

Regarding (1c), the only place where \cite[Lemma~6.8]{BenoistQuint3}
is needed is in the proof of \cite[Proposition~6.7]{BenoistQuint3},
where only the cases $V = V^{\wedge d}$ ($d = 1,\ldots,\dim(G)$) are
needed. So it suffices to show that the conclusion of
\cite[Lemma~6.8]{BenoistQuint3} holds for these spaces. Since $G$ is
unimodular, it is obvious that \cite[Lemma~6.8]{BenoistQuint3} holds
for the top-level space $V = V^{\wedge \dim G} \cong \R$, and for $d =
1,\ldots,\dim(G) - 1$, it is immediate from Proposition \ref{prop:
  stationary algebraic} that \cite[Lemma~6.8]{BenoistQuint3} holds for
the space $V = V^{\wedge d}$. 

Regarding (1b), we begin by observing that Proposition \ref{prop:
  positivity} implies that \cite[Lemma~4.1]{EskinMargulis} is valid in
our setting for the representations $(V,\rho) = (V^{\wedge d},\rho_d)$
($d = 1,\ldots,\dim(G) - 1$). Thus the same is true for
\cite[Lemma~4.2]{EskinMargulis}, which is proven directly from
\cite[Lemma~4.1]{EskinMargulis}. Note that in our context we have $H
\subset \SL^\pm(V)$ automatically, so there is no need to derive it from
semisimplicity as is done in the proof of
\cite[Lemma~4.2]{EskinMargulis}. 

Now, \cite[Lemma~4.2]{EskinMargulis} is used in two places in
\cite{BenoistQuint3}. First of all, it is used in the proof of
\cite[Proposition~6.1]{BenoistQuint3} as
\cite[Lemma~6.2]{BenoistQuint3}. There, the only case that is needed
is the case of the representation $(V,\rho) = (\Lie(G),\Ad) =
(V^{\wedge 1},\rho_1)$ (cf. \cite[\S 6.1]{BenoistQuint3}), which is
valid in our context as noted above. 

Secondly, \cite[Lemma~4.2]{EskinMargulis} is also used indirectly in
the proof of \cite[Lemma~6.3]{BenoistQuint3}, which refers to a
construction in \cite[\S 3.2]{EskinMargulis}, 
\ignore{
\footnote{We take this
  opportunity to record some mathematically relevant typos in \cite[\S
  6.2]{EskinMargulis}: 
\begin{itemize}
\item in the second paragraph $\alpha_1,\ldots,\alpha_r$ should be the basis of simple positive roots
\item in \cite[(26)]{EskinMargulis}, $\omega_k$ should have a coefficient of $c_k$
\item in the sentence after \cite[(30)]{EskinMargulis}, $\sum_j q_j \omega_j$ should be $\sum_j q_j \alpha_j$
\item in \cite[(36)]{EskinMargulis}, $q_k$ should be $c_k q_k$
\end{itemize}
We also mention that the paper \cite{BenoistQuint4} generalizes
\cite{EskinMargulis} and provides another proof of its main
result. \comdavid{Not sure if it is worth it having such a footnote,
  maybe we should just email Eskin--Margulis, since the paper still
  seems to be a preprint.}} }
which in turn depends on
\cite[Condition~A]{EskinMargulis} being satisfied. Now
\cite[Condition~A]{EskinMargulis} can be paraphrased as saying that
the conclusion of \cite[Lemma~4.2]{EskinMargulis} is valid for certain
representations denoted by \cite{EskinMargulis} as $(V_i,\rho_i)$ (not
to be confused with our representations $(V^{\wedge d},\rho_d)$),
whose defining property is that for each $i$ there exists $w_i\in V_i$
such that $\Stab(\R w_i) = P_i$, where $P_i$ is a predetermined
``standard'' parabolic subgroup. But in fact, if we let $d_i$ be the
dimension of the unipotent radical of $P_i$, then our representation
$(V^{\wedge d_i},\rho_{d_i})$ has this same property (taking $w_i$ to
be a volume form for the unipotent radical), and thus we may take
$(V_i,\rho_i) = (V^{\wedge d_i},\rho_{d_i})$. Thus, by Proposition
\ref{prop: positivity} we know that \cite[Lemma~4.2]{EskinMargulis} is
valid for these representations, i.e. that
\cite[Condition~A]{EskinMargulis} is satisfied in our setup. Note that
this proof circumvents the implicit use of semisimplicity in the proof
of \cite[Proposition~2.7]{EskinMargulis}, where it is assumed that any
$H$-invariant subspace of a representation has a complementary
invariant subspace. This argument was needed
in the original proof because of 
the hypothesis of \cite[Lemma~4.1]{EskinMargulis} that $V$ does not
have any $H$-invariant vectors, but since Proposition \ref{prop:
  positivity} does not have such a hypothesis, it is not necessary to
argue that we can reduce to this case as is done in the proof of
\cite[Proposition~2.7]{EskinMargulis}. 

Regarding (1a), we do not claim that
\cite[Proposition~5.2]{BenoistQuint3} is true in our setting, but we
claim instead that after redefining some notation appropriately,
Equation (5.3), Lemma~5.4, and Corollary~5.5 of 
\cite{BenoistQuint3} are all true
in our setting in the case $V = \Lie(G)$. Since these results are the
only results of \cite[\S 5]{BenoistQuint3} which are needed in
subsequent sections, this shows how to circumvent the use of
semisimplicity occurring in (1a). 

The notational changes we want to make to \cite[\S 5]{BenoistQuint3} are as follows:

\begin{itemize}
\item Instead of choosing $P$ to be a minimal parabolic subgroup of
  $G$, we let $P$ be the (not necessarily parabolic) group of $g\in G$
  such that $\Ad(g)$ preserves 
  $W$ and $\Ad(g)|_W$ is a similarity. Note that by assumptions (I)
  and (II), we have $\supp(\mu) \subset P$. 
\item Instead of letting $V$ be an arbitrary representation of $G$, we
  require $V = \Lie(G)$. 

\item Instead of letting $V_0$ be the weight space of the largest
  weight $\chi$, we simply let $V_0 = W$, and instead of letting the
  family $(V_b)_{b\in\B}$ be defined by 
  \cite[Proposition~5.2]{BenoistQuint3}, we let $V_b = W$ for
  all $b\in\B$. Note that by the $\supp(\mu)$-invariance of $W$, we
  have $V_b = b_0 V_{Tb}$ for all $b\in \supp(\beta)$. Also note that
by Proposition \ref{prop: subalgebra}, 
  $V_0=W$ is a Lie subalgebra, and this is necessary in
  order for the concept of a flow indexed by $V_0$ to make sense
  (cf. \cite[\S 6.5]{BenoistQuint3}) and in particular to guarantee
  the existence of conditional measures with respect to this flow
  (cf. \cite[\S 6.6]{BenoistQuint3}). In Benoist--Quint's setup, the
  fact that $V_0$ is a subalgebra follows immediately from the 
  definition of $V_0$. 
\item Since \cite[Proposition~5.2(a)]{BenoistQuint3} is not valid for
  arbitrary representations in our setting, the existence of a map
  $\xi:B \to G/P$ satisfying $\xi(b) = b_0 \xi(Tb)$ is not \emph{a
    priori} clear. In fact, if we had chosen $P$ to be a minimal
  parabolic subgroup of $G$, then it seems unlikely that such a $\xi$
  would exist in general. However, our choice of $P$ guarantees that
  $\supp(\mu) \subset P$ and thus that the constant function $\xi(b) =
  [P]$, where $[P]$ is the
 identity coset in $G/P$, satisfies $\xi(b) = b_0 \xi(Tb)$ for all
 $b\in 
  \supp(\beta)$. So we let $\xi \equiv [P]$. 
\item For convenience we choose the section $s:G/P \to G/U$ so that
  $s([P]) = [U]$, where $[U]$ is the identity coset in $G/U$,
  so that $s(\xi(b)) = [U]$ for all $b\in\B$. This 
  choice implies that $\sigma(zu,\xi(b)) = z$ for all $zu \in P = ZU$
  and $b\in\B$. In particular, we have $\theta(b) = \pi_Z(b_0)$ and
  thus $\theta_\R(b) = \log\|\Ad(b_1)|_W\|$ for all $b\in\B$. (Note
  that in \cite[(5.2)]{BenoistQuint3}, $\chi$ should be understood as
  a homomorphism from $Z$ to $\R$ defined by the formula $\chi(ma) =
  \chi(a)$, where $m\in M = K\cap Z$ and $a\in A$.) 
\end{itemize}
Using this notation, assumption (II) guarantees that
\cite[Lemma~5.4]{BenoistQuint3} holds in our setup. Combining
assumptions (I) and (II) guarantees that the formula
\cite[(5.3)]{BenoistQuint3} holds. Finally, Proposition \ref{prop: cor
  5.5} guarantees that \cite[Corollary~5.5]{BenoistQuint3} holds. 

To summarize, we have shown that the conclusion of 
\cite[Theorem~1.1]{BenoistQuint3} is valid in our setup, and have
shown that it implies part (i) of Theorem \ref{thm: new main}.
\end{proof}

\ignore{

Following \cite[\S 6.7]{BenoistQuint3} and \cite[\S
3.4]{BenoistQuint3}, for each $b \in \B$, we set 
\[
W_b = \{v \in V : \sup_n \left(e^{\theta(b)+\cdots +
\theta(T^{n-1} b)} \|\Ad((b_1\cdots b_n)^{-1}) v\| \right) < \infty\},
\]
\[
W_b(x) = \Big\{y\in X: \dist\left((b_1\cdots b_n)^{-1} x, (b_1\cdots b_n)^{-1} y \right ) \tendsto{n \to \infty} 0\Big\}.
\]
Then for all $b \in \B$, $W_b$ is a vector space such that $W \subset W_b$. Moreover, by assumption (II) we have $\int_G \theta(g) \;\dee\mu(g) > 0$, and thus for $\beta^X$-a.e. $(b,x)\in Z$ we have
\begin{equation}
\label{Wbcontainment}
\exp\big(W_b\big)x \subset W_b(x).
\end{equation}

\medskip

{\bf Step 1 (Limit measures are non-atomic).}
For $\beta^X$-a.e. $(b,x)\in Z$, we claim that $\nu_b(\exp(W_b)x)=0$.

\medskip

}

\ignore{

For $g \in \supp(\mu)$, let $\theta(g)$ such that $\Ad(g)|_W$ is expansion by $ e^{\theta(g)}$. By assumption \textup{(I)}, $\Ad(g_n \cdots g_1)|_W$ is a similarity with expansion $e^{\theta(g_n) + \cdots + \theta(g_1)}$. For any $b \in \B$, we set
\[
W_b = \left\{v \in V : \sup_n \left(e^{\theta(b_1)+\cdots + \theta(b_n)} \|b_n^{-1} \cdots b_1^{-1} v\| \right) < \infty \right\}.
\]

Using hypotheses \textup{(I)} and (V) \combarak{check this}, one can see that for any metric $d$ on $X$ which induces the topology, for $\beta$-a.e. $b\in\B$,
\[
x' \in \exp(W)x \ \implies \
d\left(b_n^{-1} \cdots
b_1^{-1}x, b_n^{-1} \cdots
b_1^{-1}x' \right ) \tendsto{n \to \infty} 0.
\]

\[
W_b(x) = \{x' \in X: \lim_{n \to \infty} d(b_n^{-1} \cdots
b_1^{-1}x, b_n^{-1} \cdots
b_1^{-1}x') \tendsto{n \to \infty} 0\}.
\]
Here $d$ is some metric on $X$ inducing the topology. Note that by \textup{(I)} and (V), $W \subset W_{b}(x)$ for $\beta^X$-a.e. $(b,x)\in Z$. \combarak{Explain this if important.}

\medskip
}


\begin{proof}[Proof of Theorem \ref{thm: new main}(ii)]
Suppose first that $X$ is compact. According to Theorem \ref{thm: new main}(i), the only $\mu$-stationary probability measure on $X$ is the $G$-invariant probability measure $m_X$ induced by Haar measure. According to the so-called ``Breiman law of large numbers'' (see e.g. \cite[Chapter~2.2]{BenoistQuint_book}), for all $x \in X$, for $\beta$-a.e. $b \in \B$, the ``empirical measures'' $\frac{1}{N} \sum_{i=1}^N \delta_{g_{b_1^i}x} \; (N\in\N)$ converge to a $\mu$-stationary measure on $X$ as $N \to \infty$. Therefore these measures must converge to $m_X$ and we are done.

In the noncompact case we use results from \cite{BenoistQuint3,BenoistQuint7}. Denote by $\bar{X} = X \cup \{\infty\}$ the one-point compactification of $X$. By Theorem \ref{thm: new main}(i), any $\mu$-stationary probability measure on $\bar{X}$ is a convex combination of $m_X$ and the Dirac measure at the point at infinity. Using again the Breiman law of large numbers we know that for any $x \in X$, for $\beta$-a.e. $b\in\B$, $\frac{1}{N} \sum_{i=1}^N \delta_{g_{b_1^i}x}$ converges to a $\mu$-stationary measure $\nu$ on $\bar{X}$. So it suffices to rule out escape of mass, i.e. to show that $\nu(\{\infty\})=0$. To this end we need to show that for all $x \in X$ and $\vre>0$ there is a compact set $K \subset X$ such that
\[
\liminf_{N \to \infty}
\frac{\#\{i \leq N : g_{b_1^i} x
\in K \}}{N}
> 1-\vre.
\]
According to \cite[Proposition~3.9]{BenoistQuint7}, it suffices to prove the existence of a proper function $u: X \to [0, \infty)$ such that there exist $a \in (0,1)$ and $C>0$ such that for all $x \in X$, we have
\eq{eq: u satisfies}{
\int_G u(gx) \,
\dee \mu(g) \leq au(x)+ C.}
But this is exactly the conclusion of \cite[Lemma 6.3]{BenoistQuint3}, and as we have argued above, this conclusion is valid in our setup as well.
\ignore{
We use the function $u$ defined in \cite[\S 3.2]{EskinMargulis}, where the following was shown: there exist $w_j \in V_{d_j} \, (j=1, \ldots, r)$ such that if for all sufficiently small $\delta$ there exist $c<1$ and $n\geq 1$ such that for all $v \in \rho_{d_j}(L)w_j$ we have
\eq{eq: condn EM}{
\int_G \frac{1}{\|\rho_{d_j}(g)v \|^\delta} \, \dee \mu^{*n}(g) \leq \frac{c}{\|v\|^\delta},
}
then $u$ satisfies \equ{eq: u satisfies}.

It was also shown in \cite[\S 4]{EskinMargulis} that in order to
obtain \equ{eq: condn EM} it suffices to prove that there exist $N\geq
1$ and $c>0$ such that for all $n \geq N$ and $v \in V^{\wedge d_j}
\sm \{0\}$, we have 
\[
\frac{1}{n} \int_G \log \frac{\|\rho_{d_j}(g)v\|}{\|v\|} \, \dee
\mu^{*n}(g) > c.
\]
This in turn follows from Proposition \ref{prop: positivity}.
}
\end{proof}

\section{Fiber bundle extensions}\name{sec: fiber bundle extensions}
In this section we will prove Theorem \ref{thm: fiber bundle
  extension}. This will follow from some results valid in a more
general framework. Let $X$ be a locally compact second countable
space, $G$ a locally compact second countable group acting
continuously on $X$, $m$ a $G$-invariant and ergodic probability
measure on $X$, and $\mu$ a probability measure on $G$ with compact
support $E$. Let $B =E^{ \N},\, \bar{\B} = E^{\Z}$ and $\beta
=\mu^{\otimes \N}, \, \bar{\beta} = \mu^{\otimes \Z}$. We will use the
letter $T$ to denote the shift map on both $\B$ and $\bar{\B}$.

\begin{prop}
\name{propositionforwardindependent}
Fix $x_0\in X$, and suppose that for $\beta$-a.e. $b\in \B$, the
random path $(g_{b_1^n}x_0)_{n\in\N}$ is equidistributed with respect
to the measure $m$ on $X$. Then for $\beta$-a.e. $b\in \B$, the
sequence 
\[
\big(g_{b_1^n} x_0, T^n b\big)_{n\in\N}
\]
is equidistributed with respect to the measure $ m \otimes \beta$ on $X \times \B$.
\end{prop}
\begin{proof}
Let $C_c(X\times B)$ be the space of compactly supported continuous
functions on $X \times \B$. We need to show that for
$\beta$-a.e. $b\in\B$, for all $\varphi \in C_c(X\times B)$ we have 
\eq{eq: need to show}{
\frac{1}{n}\sum_{i = 0}^{n - 1} \varphi\left(g_{b_1^i} x_0,T^i b \right)
\tendsto{n\to\infty} \int_{X \times \B} \varphi\;\dee(m \otimes \beta).
}
It suffices to check that \eqref{eq: need to show} holds for functions
$\varphi$ from a countable dense collection of functions $\FF \subset
C_c(X \times \B)$; moreover, we can choose $\FF$ so that for each
$\varphi \in \FF$ and for each $(x,b)\in X\times \B$, $\varphi(x,b)$
depends on only finitely many coordinates of $b$. Since $\FF$ is
countable, we can switch the order of quantifiers, so in the remainder
of the proof we fix $\varphi\in \FF$ and we will show that \eqref{eq:
  need to show} holds for $\beta$-a.e. $b\in \B$. Let $N$ be a number
large enough so that $\varphi(x,b)$ depends only on the first $N$
coordinates of $b$. 

For each $x \in X$, let
\[
\varphi_X(x) = \int_{\B} \varphi(x, b) \, \dee \beta(b).
\]
Then $\varphi_X:X\to\R$ is continuous and compactly supported. Let
\[
h(x, b) = \varphi(x, b) - \varphi_X(x).
\]
By assumption, for $\beta$-a.e. $b\in\B$ the random walk
$(g_{b_1^n}x_0)_{n\in\N}$ is equidistributed with respect to $m$, and
thus 
\[
\frac{1}{n}\sum_{i = 0}^{n - 1}
\varphi_X(g_{b_1^i} x_0)
\tendsto{n\to\infty}
\int_X 
\varphi_X\;\dee m
=\int_{X \times \B} \varphi
\;\dee(m\otimes \beta),
\]
so to complete the proof we need to show that for $\beta$-a.e. $b\in\B$,
\begin{equation}
\label{ETShempirical}
\frac{1}{n}\sum_{i = 0}^{n - 1} h(g_{b_1^i} x_0,T^i b)
\tendsto{n\to\infty} 0.
\end{equation}
In what follows we treat
$b$ as a random variable with distribution $\beta$. Fix $n\geq 0$. If
$f(b)$ is a number depending on $b$, let $\E[f(b) | b^n_1]$ denote the
conditional expectation of $f(b)$ with respect to the first $n$
coordinates of $b$. Then for all $i\geq 0$ we have 
\[
\E[h(g_{b_1^i} x_0,T^i b)|
b^n_1] =
\left \{\begin{matrix}
\int_{\B} h(g_{\altb_1^{i-n}} g_{b_1^n} x_0,
T^{i-n}(\altb))\, \dee \beta(\altb) & \text {
if } i \geq n \\
\int_{\B} h(g_{b_1^i}x_0 , b_{i+1} \cdots b_n \altb) \, \dee \beta(\altb) & \text{ if } i < n
\end{matrix} \right .
\]
Now consider the random variable
\[
M_n \df \sum_{i = 0}^\infty \E[h(g_{b_1^i} x_0,T^i b)| b^n_1].
\]
The sum is actually finite since, by the definition of $\varphi_X$,
for all $i \geq n$, we have $\E[h(g_{b_1^i} x_0,T^i b)| b^n_1] =
0$. Also, by the definition of $N$, for all $i \leq n - N$ we have
$\E[h(g_{b_1^i} x_0,T^i b)| b_1^n] = h(g_{b_1^i} x_0,T^i
b)$. Therefore 
\begin{equation}
\label{Xnasymp}
M_n = \sum_{i = 0}^{n - 1} h(g_{b_1^i} x_0,T^i b) + O(1).
\end{equation}
Now by construction, the sequence $(M_n)_{n\in\N}$ is a martingale,
and it has bounded steps by \eqref{Xnasymp}. It follows that
$\frac{1}{n} M_n \tendsto{n \to \infty} 0$ almost surely (see
e.g. \cite[Corollary~1.8 of Appendix]{BenoistQuint_book}). Combining
with \eqref{Xnasymp} gives \eqref{ETShempirical}. 
\end{proof}

Using a bootstrapping argument we now obtain a stronger version of
Proposition \ref{propositionforwardindependent}.

\begin{prop}
\name{prop: bootstrap}
Let the notation and assumptions be as in Proposition
\ref{propositionforwardindependent}. Let $Y$ be a locally compact
metric space, let $f: \bar{\B} \to Y$ be a measurable map, and let
$m_Y = f_* \bar{\beta}$. Then for $\bar{\beta}$-a.e. $b\in \bar{\B}$,
the sequence 
\eq{eq: bootstrap}{
\big(g_{b_1^n} x_0, f(T^n b)\big)_{n\in\N}
}
is equidistributed with respect to the measure $m \otimes m_Y$ on $X \times Y$.
\end{prop}

\begin{proof}
By Proposition \ref{propositionforwardindependent}, for $\beta$-a.e. $b \in B$ the random walk trajectory
\begin{equation}
\label{randomwalkprev}
\big(g_{b_1^n} x_0,T^n b\big)_{n\in\N}
\end{equation}
is equidistributed in $X\times \B$ with respect to $m\otimes
\beta$. Fix $\ell\in\N$, and let $B^{(\ell)} = \prod_{i=-\ell}^\infty
E$ and $\beta^{(\ell)} = \bigotimes_{i = -\ell}^\infty \mu$. We will
abuse notation slightly by letting $T$ denote the shift map on all
three of the spaces $\B$, $B^{(\ell)}$, and $\bar\B$. In addition we
let $T^\ell:B\to B^{(\ell)}$ be the isomorphism defined by the equation $T^\ell(b)_i = b_{i+\ell}$ ($i \geq -\ell$), which can be
thought of as an analogue of the $\ell$th power of the shift map, although it is not an endomorphism. With
these conventions, applying $T^\ell$ to the equidistributed sequence \eqref{randomwalkprev} (where $b\in B$ is a $\beta$-typical point) shows
that for $\mu^{(\ell)}$-a.e. $b\in B^{(\ell)}$, the random walk
trajectory \eqref{randomwalkprev} is equidistributed in $X\times
B^{(\ell)}$ with respect to $m\otimes \beta^{(\ell)}$. Thus if
$\varphi:X\times\bar\B \to \R$ is a bounded continuous function such
that $\varphi(x,b)$ depends only on $x$ and $b_{-\ell}^\infty \in
B^{(\ell)}$, then for $\bar{\beta}$-a.e. $b\in \bar{\B}$, the sequence
\eqref{randomwalkprev} is equidistributed for $\varphi$ with respect
to $m\otimes \bar{\beta}$. By choosing a countable dense sequence of such
functions $\varphi$, we can see that for $\bar{\beta}$-a.e. $b\in
\bar{\B}$, the random walk trajectory \eqref{randomwalkprev} is
equidistributed in $X\times \bar{\B}$ with respect to $m\otimes
\bar{\beta}$. 

Now by Lusin's theorem, for each $\ell\in\N$ there exists a compact
set $K_\ell \subset \bar{\B}$ of $\bar\beta$-measure at least $1 -
1/\ell$ such that $f|_{K_\ell}$ is continuous. By the ergodic theorem,
for $\bar{\beta}$-a.e. $b \in \bar{\B}$, for all $\ell\in\N$ we have 
\[
\frac{1}{n}\#\big\{i = 1,\ldots,n : T^i b \in K_\ell\big \} \tendsto{n \to \infty}
\bar{\beta} (K_\ell) \geq 1 - \frac 1\ell\cdot
\]
Fix $b \in \bar{\B}$ such that this is true, and such that
\eqref{randomwalkprev} is equidistributed. Let $\varphi:X\times
Y\to\R$ be a bounded continuous function, and for each $(x,b)\in
X\times \bar B$ let $F(x,b) = (x,f(b))$. Fix $\ell\in\N$. Then
$\varphi\circ F$ is continuous on $X\times K_\ell$ and bounded on
$X\times \bar B$. Using Tietze's extension theorem, let $\varphi_\ell$
be a continuous extension of $\varphi\circ F|_{X\times K_\ell}$ to
$X\times \bar\B$ such that $\|\varphi_\ell\|_\infty \leq
\|\varphi\|_\infty$. Then since we assumed that \eqref{randomwalkprev}
is equidistributed, we have 
\[
\frac{1}{n} \sum_{i = 1}^n \varphi_\ell(g_{b_1^i} x_0,T^i b)
\tendsto{n\to\infty} \int \varphi_\ell \;\dee(m\otimes\beta)
\]
and thus
\begin{align*}
&\limsup_{n\to\infty} \left|\frac{1}{n} \sum_{i = 1}^n
\varphi(g_{b_1^i} x_0,f(T^i b)) - \int \varphi\;\dee(m\otimes f_*\beta)\right|\\
\leq & \limsup_{n\to\infty} \frac{1}{n} \sum_{i = 1}^n
\Big|\varphi_\ell(g_{b_1^i} x_0,T^i b) - \varphi\circ F(g_{b_1^i}
x_0,T^i b)\Big| + \int |\varphi_\ell - \varphi\circ F| \;\dee(m\otimes \beta)\\
\leq & 2 \|\varphi_\ell - \varphi\circ F\|_\infty \, \beta(\B\sm K_\ell) \leq
4\|\varphi\|_\infty \, \beta(\B \sm K_\ell) \tendsto{\ell\to\infty} 0.
\end{align*}
Since $\varphi$ was arbitrary, this means that \eqref{eq: bootstrap} is equidistributed.
\end{proof}

\begin{prop}
\name{propositionrotationalpart}
Let $G, \mu, X, m$ be as before and let $\Gamma$ be the subgroup of
$G$ generated by $\supp(\mu)$. Let $K$ be a compact group, $m_K$ Haar
measure on $K$, and $\kappa: \Gamma\to K$ a homomorphism. Let $Z = X
\times K$ and consider the left action of $\Gamma$ on $Z$ defined by
the formula $\gamma (x,k)= (\gamma x, \kappa(\gamma)k)$. Assume that
this $\Gamma$-action is ergodic with respect to $m\otimes m_K$. Let
$\pi_X : Z \to X$ be the projection map onto the first factor, and let
$\nu$ be a $\mu$-stationary measure on $Z$ such that $(\pi_X)_*\nu =
m$. Then $\nu = m\otimes m_K$. 
\end{prop}
\begin{proof}
There is a right-action of $K$ on $Z$ given by $(x,k')k = (x,k'k),$
and this action commutes with the left-action of $\Gamma$ on $Z$. For
any measure $\theta$ on $Z$ and any smooth positive function $\psi$ on
$K$ such that $\int_K \psi \;\dee m_K = 1$, we can smooth $\theta$ by
averaging with respect to the $K$-action: 
\begin{equation}\label{eq: convolve}
\theta^{(\psi)}(A) \df \int_K \theta(Ak^{-1})\psi(k) \, \dee m_K(k).
\end{equation}
Note that if $(\psi_j)_{j\in\N}$ is an approximate identity then
$\theta^{(\psi_j)} \to \theta$. Since the $\Gamma$ and $K$ actions
commute and $\nu$ is $\mu$-stationary, so is $\nu^{(\psi)}$ for any
$\psi$. Since $(\pi_X)_*\nu = m$ and the $K$-action preserves the
first coordinate, we have $(\pi_X)_*\nu^{(\psi)} = m$ for all $\psi$. 

Since $(\pi_X)_*\nu = m$, by the Rokhlin disintegration theorem we can write
\[
\nu = \int_X \delta_x\otimes m_x\;\dee m(x)
\]
for some measurable map $X\ni x\mapsto m_x \in \Prob(K)$. Here
$\delta_x$ denotes the Dirac point measure centered at $x$. For each
$\gamma \in \Gamma$, by the definition of the $\Gamma$-action on $Z$,
we have $\gamma_*(\delta_x \otimes m_x) = \delta_{\gamma x} \otimes
\kappa(\gamma)_* m_x$. Since $\nu$ is $\mu$-stationary and $m$ is
$\Gamma$-invariant, we have 
\begin{align*}
\nu &= \int_G \gamma_*\nu \;\dee \mu(\gamma)\\
&= \int_G \int_X \delta_{\gamma x}\otimes \kappa(\gamma)_\ast m_x
\;\dee m(x) \;\dee \mu(\gamma)\\
&= \int_G \int_X \delta_x \otimes \kappa(\gamma)_\ast m_{\gamma^{-1} x}
\;\dee m(x) \;\dee \mu(\gamma) 
\\
&= \int_X \delta_x \otimes \left(\int_G \kappa(\gamma)_\ast m_{\gamma^{-1} x}
\;\dee \mu(\gamma)\right) \;\dee m(x),
\end{align*}
so by the uniqueness of disintegrations we have
\eq{eq: will also satisfy}{
m_x = \int_G \kappa(\gamma)_\ast m_{\gamma^{-1} x} \;\dee \mu(\gamma)
\;\;\;\; \text{ 
for $m$-a.e. $x\in X$}.
}
Repeating the same considerations for $\nu^{(\psi)}$, by the
uniqueness of disintegrations, we find that we have a measure
disintegration $\nu^{(\psi)} = \int_X \delta_x \otimes m_x^{(\psi)} \;
\dee m(x)$ where the probability measures $m_x^{(\psi)} \; (x\in X)$ are defined
via \eqref{eq: convolve} and satisfy 
\[
m^{(\psi)}_x = \int_G \kappa(\gamma)_\ast m^{(\psi)}_{\gamma^{-1}
x} \;\dee \mu(\gamma) \;\;\;\; \text{
for $m$-a.e. $x\in X$}.
\]
It follows from \eqref{eq: convolve} that each of the measures
$m^{(\psi)}_x \; (x\in X)$ is absolutely continuous with respect to
$m_K$. Thus we can write $\dee m^{(\psi)}_x = f_x \, \dee m_K$, where
$f_x = f^{(\psi)}_x \; (x\in X)$ are nonnegative functions in $C(K) \subset L^2(K,
m_K)$ which satisfy 
\begin{equation}
\label{fxequals}
f_x(k) = \int_G f_{\gamma^{-1} x} (\kappa(\gamma)^{-1} k) \;\dee \mu(\gamma) \;\;\;\; \text{
for $m \otimes m_K$-a.e. $(x, k) \in X\times K$}.
\end{equation}
Now for fixed $\psi$, by Jensen's inequality, for $m$-a.e. $x\in X$ we have 
\begin{equation}
\label{jensen}
\begin{split}
\|f_x\|^2 &= \int_K |f_x(k)|^2 \;\dee m_K(k)\\
&\leq \int_K \int_G |f_{\gamma^{-1} x}(\kappa(\gamma)^{-1} k)|^2
\;\dee\mu(\gamma) \;\dee m_K(k)\\ 
&= \int_G \|f_{\gamma^{-1} x}\|^2 \;\dee \mu(\gamma),
\end{split}
\end{equation}
with equality if and only if $f_x(k) = f_{\gamma^{-1}
  x}(\kappa(\gamma)^{-1} k)$ for $\mu\otimes m_K$-a.e. $(\gamma,k)\in
\Gamma\times K$. Here $\|\cdot\|$ denotes the norm on $L^2(K,m_K)$. On
the other hand, since $m$ is $\Gamma$-invariant we have 
\begin{align*}
\int_X \int_G \|f_{\gamma^{-1} x}\|^2 \;\dee \mu(\gamma) \;\dee m(x) & = \int_G \int_X
\|f_{\gamma^{-1} x}\|^2
\;\dee m(x)
\;\dee \mu(\gamma)
\\
& = \int_G \int_X \|f_x\|^2 \;\dee m(x) \;\dee \mu(\gamma) \\
& = \int_X \|f_x\|^2 \;\dee m(x),
\end{align*}
so for $m$-a.e. $x\in X$, equality holds in \eqref{jensen}, that is,
we have $m_x^{(\psi)} = \kappa(\gamma)_* m_{\gamma^{-1}x}^{(\psi)}$
for $\mu \otimes m$-a.e. $(\gamma,x)\in \Gamma \times X$. This implies
that $\nu^{(\psi)}$ is $\Gamma$-invariant, and since it is absolutely
continuous with respect to $m \otimes m_K$, and $\Gamma$ acts
ergodically with respect to $m \otimes m_K$, we must have
$\nu^{(\psi)} = m \otimes m_K$. Taking the limit along an approximate
identity, we obtain that $\nu = m \otimes m_K$, as claimed. 
\end{proof}
\ignore{
can arrange that $m_x^{(\psi_j)} \tendsto{j \to \infty} m_x$, yielding
\combarak{New stuff. $\nu^{(\psi_j)} \tendsto{j \to \infty} \nu$
so $m \otimes m_K = \nu$. }

\begin{equation}
\label{invariantcomponents}
m_x = \kappa(\gamma)_\ast m_{\gamma^{-1} x} \text{ for $m$-a.e. $x\in
  X$ and $\mu$-a.e. $\gamma\in \Gamma$}; 
\end{equation}
in particular, $\nu$ is $\Gamma$-invariant, and this gives the
stronger assertion that for any $\gamma \in \Gamma$, for $m$-a.e. $x
\in X$, $m_x = \kappa(\gamma)_* m_{\gamma^{-1}x}.$ 
Restricting to $\Gamma_0$ we see that for all $\gamma \in \Gamma_0$,
for $m$-a.e. $x \in X$ we have $m_x = m_{\gamma^{-1} x}$. Since by
assumption $\Gamma_0$ acts ergodically on $X$, the function $x\mapsto
m_x$ is constant almost everywhere, say $m_x = m_0$ for $m$-a.e. $x\in
X$. Then \eqref{invariantcomponents} becomes 
\[
m_0 = \kappa(\gamma)_\ast m_0 \text{ for all } \gamma\in \Gamma.
\]
That is, $m_0$ is invariant under $\kappa(\Gamma)$, which is a dense
subgroup of $K$. Since the stabilizer of a measure is a closed
subgroup of $G$, $m_0$ is $K$-invariant, and hence $m_0 = m_K$. 
\end{proof}
}
\begin{remark}
See \cite[Proof of Theorem~3.4]{Furstenberg_stiffness} for a similar argument.
\end{remark}

\begin{cor}\name{cor: extend compact}
With the assumptions and notations of Proposition
\ref{propositionrotationalpart}, if almost every random walk
trajectory 
\begin{equation}
\label{extend compact 1}
(g_{b_1^n} x_0)_{n\in\N}
\end{equation}
is equidistributed with respect to $m$, then almost every random walk trajectory
\begin{equation}
\label{extend compact 2}
(g_{b_1^n} x_0,\kappa(g_{b_1^n}))_{n\in\N}
\end{equation}
is equidistributed with respect to $m \otimes m_K$.
\end{cor}
\begin{proof}
Let $\hat X$ denote the one-point compactification of $X$, and let
$\nu\in\Prob(\hat X\times K)$ be a weak-* limit of the empirical
measures of the sequence \eqref{extend compact 2}. By the Breiman law
of large numbers, $\nu$ is $\mu$-stationary, and since \eqref{extend
  compact 1} is equidistributed, the projection of $\nu$ to $\hat X$
is equal to $m$. So by Proposition \ref{propositionrotationalpart}, we
have $\nu = m\otimes m_K$. (Note that since $(\pi_X)_* \nu = m$, we
actually have $\nu\in\Prob(X\times K)$ rather than just
$\nu\in\Prob(\hat X\times K)$.) 
\end{proof}

\begin{proof}[Proof of Theorem \ref{thm: fiber bundle extension}]
First apply Corollary \ref{cor: extend compact} to $X$ and the homomorphism $\kappa$, and then apply Proposition \ref{prop: bootstrap} to $X \times K $ and the map $f$.
\end{proof}

\section{Examples}\name{sec: main example}
The purpose of this section is to introduce some situations in which the hypotheses of Theorem \ref{thm: new main} are satisfied. We will need some additional information about Lyapunov exponents in the case of reducible representations. Let $V$ be a finite-dimensional real vector space, $W \subset V$ a subspace, and $G$ a closed subgroup of $\SL^\pm(V)$ which leaves $W$ invariant, so that $G$ acts on $V$, on $W$ (via the restriction of the $G$-action on $V$) and on $V/W$ (via the induced quotient action):
\[
1 \longrightarrow W \longrightarrow V \longrightarrow V/W \longrightarrow 1.
\]
Let $\mu$ be a compactly supported probability measure on $G$. We
introduce the following notation for recording the Lyapunov exponents
and their multiplicities for an action on $V$: $\mathcal{L}_V =
\sum_{i=1}^k d_i \delta_{\chi_i}$, where $k, \, d_i, \, \chi_i$ are as
in Theorem \ref{theoremoseledets}, and $\delta_\chi$ is a formal
Kronecker symbol. Here we think of $\mathcal L_V$ as a formal sum, so
that expressions of the form $\mathcal L_W + \mathcal L_{V/W}$ make
sense.

\begin{lem}\name{lem: furstenberg kifer}
With the above notation, assume that
\eq{eq: strict}{
\inf \supp (\mathcal{L}_W) > \sup \supp (\mathcal{L}_{V/W}),
}
i.e. each of the (Lyapunov) exponents of (the action of $G$ on) $W$ is strictly
larger than each of the exponents of $V/W$. 
Then
\begin{equation}\label{eq: multiplicities mix}
\mathcal{L}_V = \mathcal{L}_W +\mathcal{L}_{V/W};
\end{equation}
i.e. each of the exponents of $W$ and of $V/W$ appears as an exponent
of $V$, with the same multiplicity. Furthermore: 
\begin{itemize}
\item[(a)]
For $\beta$-a.e. $b\in\B$, $W$ is complementary to $V^{< W} (b)$,
where $V^{<W}(b)$ denotes the Oseledec space corresponding to the
smallest exponent of $W$. 
\item[(b)]
If there is a basis for $V $ with respect to which the matrices
$\rho(g) \, (g \in E)$ are all in upper triangular block form, and the
$i$-th diagonal block is a similarity map with expansion factor
$e^{\alpha_i(g)}$, then (after re-indexing) the exponents of $V$ are
the same as the numbers $\int \alpha_i \; \dee \mu \, (i =
1,\ldots,k)$, with the same multiplicities. 
\end{itemize}
\end{lem}

\ignore{
For each of these actions we use the notation of Theorem
\ref{theoremoseledets}, adding a superscript to denote the space,
e.g. $\chi^{(V)}_i$ is the $i$th Lyapunov exponent on $i$, $d_i^{(V)}$
is its multiplicity, and $k^{(W)}$ is the number of Lyapunov exponents
of $W$.

Then $k^{(V)}=k^{(W)}+k^{(V/W)}$,
\eq{eq: exponents mix}{
\chi^{(V)}_i = \left \{ \begin{matrix} \chi^{(W)}_i & i=1, \ldots,
k(W) \\ \chi^{(V/W)}_{i-k(W)} & i= k(W)+1, \ldots,
k(V)\end{matrix} \right.
}
and
\begin{align}\label{eq: multiplicities mix}
d^{(V)}_i = d^{(W)}_i, \ & i=1, \ldots, k^{(W)} \\
d^{(V)}_{k^{(W)}+i} = d^{(V/W)}_i, \ & i=1, \ldots, k^{(V/W)}.
\end{align}
}

\begin{proof}
Note that assertion (a) is an immediate consequence of \eqref{eq:
  strict} and \eqref{eq: multiplicities mix}, which imply that the
growth rate of any nonzero vector in $W$ is greater than that of any
nonzero vector in $V^{< W}(b)$, and that $\dim W + \dim V^{<W}(b) =
\dim V$. Assertion (b) follows from \eqref{eq: multiplicities mix} by
a simple induction (its special case where the diagonal blocks are
1-dimensional was actually proven in the original paper
\cite{Oseledec} as part of the proof of Theorem
\ref{theoremoseledets}). 

In order to prove \eqref{eq: multiplicities mix}, choose $b$ to belong
to the full measure subset of $\B$ where the conclusions of Theorem
\ref{theoremoseledets} are satisfied on all three spaces $V, W,
V/W$. With the natural notations, fix $1 \leq i \leq \dim(V/W)$,
consider a vector $u$ in the set $(V/W)_{i-1}(b) \sm
(V/W)_i(b)$ corresponding to the exponent $\chi =
\chi^{(V/W)}_i$, and let $V_u = \pi^{-1}(\spa (u)),$ where $\pi: V
\to V/W$ is the projection map. We claim that $\chi$ is the minimal
exponential rate of growth of a vector in $V_u$; that is, 
\begin{equation}\label{eq: what we want}
\chi = \min \left\{\chi^{(V)}_j : 1\leq j \leq \dim(V), \; V_j(b) \cap V_u \neq \{0\} \right\}.
\end{equation}
Assume that \eqref{eq: what we want} holds for all $u\in (V/W)_{i-1}(b)
\sm (V/W)_i(b)$. Then each such $u$ has a lift $v = v_u \in
\pi^{-1}(u)$ with asymptotic exponential growth rate $\chi$; that is,
all Lyapunov exponents of $V/W$ are also Lyapunov exponents of $V$. It
follows from \eqref{eq: strict} that $v_u$ is unique, since if $v'_u$
and $v_u$ are two lifts with this property then the vector $v'_u - v_u
\in W$ has growth rate strictly greater than $\chi$. From the
uniqueness it follows that the map $u\mapsto v_u$ can be extended to a
linear map from $(V/W)_{i-1}(b)$ to $V$ such that $\pi(v_u) = u$. In
other words, for each Oseledec space $(V/W)_{i-1}(b)$ there 
is a lifted subspace in $V$ of the same dimension corresponding to the
same exponent $\chi_i$. This completes the proof assuming \eqref{eq: what we
  want}. 

It remains to prove \eqref{eq: what we want}. Let $\lambda$ denote the
quantity defined on the right-hand side of \eqref{eq: what we want},
and let $\bar{g}$ denote the action of a matrix $g\in G$ on
$V/W$. Choose an inner product on $V$ and use it to define norms on
$V, W, V/W$, where the latter space is identified with $W^\perp$. For
each $v \in V_u \sm W$, $\pi(v)$ is a nonzero multiple of $u$, so for
any $\vre>0$ and any $n$ large enough, we have 
\[
\|g_{b_1^n} v\| \geq \|\bar{g}_{b_1^n} \pi(v)\| \geq e^{(\chi - \vre)n}.
\]
Moreover for any $v \in W \sm \{0\}$, $\|g_{b_1^n} v\| \geq e^{(\chi -
  \vre)n}$ holds for large enough $n$ by \eqref{eq: strict}. This
proves $\chi \leq \lambda$. For the converse, for each $n$ fix $v_n \in V_u$ such that $\pi(v_n) = u$
and $ g_{b_1^n}v_n \in W^\perp$. The identity $\pi(v_n) = u$ implies
that the sequence $(v_n)_{n\in\N}$ is uniformly bounded away from
zero, and since the convergence in Theorem \ref{theoremoseledets} is
uniform on compact sets, it follows that for any $\vre>0$, for all
sufficiently large $n$, we have 
$
\|g_{b_1^n}v_n\| \geq e^{(\lambda - \vre)n}.
$
On the other hand, by the definition of the norms and of $\chi$, for
all sufficiently large $n$ we have 
\[
e^{(\chi+\vre)n} \geq \|\bar{g}_{b_1^n} u\| = \|g_{b_1^n} v_n \| \geq
e^{(\lambda - \vre)n}.
\]
This implies the inequality $\chi \geq \lambda$.
\end{proof}

\subsection{The main example}\label{subsec: main example}
We now present our main example. It will be used in Part \ref{part2} of this
paper to deduce Diophantine results. Let $\pdim, \qdim$ be positive
integers, let $\dimsum = \pdim+\qdim $, let $G = \PGL_\dimsum(\R)$ and
$\Lambda = \PGL_\dimsum(\Z)$ (we recall that these are our respective
notations for the quotients of $\SL^\pm_\dimsum(\R)$ and
$\SL^\pm_\dimsum(\Z)$ by their subgroups of scalar matrices), and let
$\mu$ be a compactly supported probability measure on $G$. 
At the risk of annoying the reader, in what follows we will refer to elements of $G$ as matrices, when in fact they are equivalence classes of matrices modulo multiplication by scalars. Fix inner products on $\R^\pdim$ and $\R^\qdim$, and let $\GO_\pdim$ and $\GO_\qdim$ respectively denote the groups of matrices preserving these inner products (not necessarily orientation preserving). Let $\MM$ denote the space of all $\pdim \times \qdim$ real matrices. For each $t \in \R$ and $\bfalpha \in \MM$, let
\begin{equation}\label{eq: at}
a_t = \left[\begin{array}{ll}
e^{t/\pdim}I_\pdim & \\
& e^{-t/\qdim}I_\qdim
\end{array}\right], \ \ \ u_{\bfalpha} = \left[\begin{array}{ll}
I_\pdim & - \bfalpha \\
& I_\qdim
\end{array}\right].
\end{equation}
For each $O_1\in \GO_\pdim$ and $O_2\in \GO_\qdim$, let $O_1 \oplus
O_2$ denote the direct sum of $O_1$ and $O_2$, i.e. 
\begin{equation}
O_1 \oplus O_2 = \left[\begin{array}{ll}
O_1 & \\
& O_2
\end{array}\right].
\end{equation}
Finally, let $A = \{a_t : t\in\R\}$, $K = \{O_1 \oplus O_2 : O_1\in
\GO_\pdim, \; O_2\in \GO_\qdim\}$, $U = \{u_{\bfalpha} : \bfalpha \in
\MM\}$, and $\para = AKU$. Note that $A$ and $K$ commute with each other and normalize $U$.

Let $V^+$ denote the Lie algebra of $U$, that is, $V^+$ consists of
those matrices whose $(i,j)$th entry vanishes if $i > \pdim$ or $j
\leq \pdim$. Let $\Hmu$ denote the Zariski closure (in $G$) of the
group generated by $\supp(\mu)$. 
\begin{dfn}\label{def: block form}
We say that $\mu$ is in {\em $(\pdim,\qdim)$-upper block form} if

\begin{itemize}
\item[(i)] $\supp(\mu) \subset \para$, i.e. for all $g \in \supp(\mu)$
  there exist $a_g = a_t \in A$, $k_g = O_1\oplus O_2 \in K$, and $u_g
  = u_\bfalpha \in U$ such that $g = a_g k_g u_g$. In what follows we
  will write $t = \theta_1(g)$ and $\bfalpha = \theta_2(g)$. 
\item[(ii)]
The function $\theta_1:\para\to\R$ implicitly defined by (i) satisfies
\begin{equation}
\label{c1def}
c_1 \df \int_G \theta_1(g) \, \dee \mu(g)>0.
\end{equation}
\item[(iii)]
The Lie algebra of $\Hmu$ contains $V^+$.
\end{itemize}

\end{dfn}

\begin{thm}\name{thm: assumptions satisfied}
Let $G, \Lambda, \mu$ be as above, where $\mu$ is in
$(\pdim,\qdim)$-upper block form. Then for each $d$ there is a proper
subspace $W^{\wedge d} \subset V^{\wedge d}$
such that the assumptions
of Theorem \ref{thm: new main} satisfied. 
\end{thm}

\begin{proof}
It follows by direct calculation that $X$ is connected and in
particular that $\Gamma$ acts transitively on the connected components
of $X$. It follows from (iii) that $\Gamma$ contains two elements
$g_1, g_2$ with $1\neq u_{g_2}$ and an easy computation (see
the proof of Lemma \ref{lem: verify iv} below) shows that the sequence
$(g_1^{-n} g_2 g_1^{n })$ has a convergent subsequence but is not
eventually constant. Thus $\Gamma$ is not discrete and in particular
is not virtually contained in a conjugate of $\Lambda$.

Now we construct a subspace $W^{\wedge d} \subsetneqq V^{\wedge d}$
such that assumptions (I), (II), and (III) hold. We first express the
adjoint action of $g = aku \in \para$ on $V = \Lie(G) = \{\Matrix \in
\MM_{\dimsum\times\dimsum} : \Tr[\Matrix] = 0\}$. For each $1 \leq i,j
\leq \dimsum$ let $E_{i,j}$ denote the matrix with 1 in the $(i,j)$th
entry and 0 elsewhere. Let $\block_1 = \{1,\ldots,\pdim\}$ and
$\block_2 = \{\pdim + 1,\ldots,\dimsum\}$. For each $j_1,j_2 \in
\{1,2\}$, let $V_{j_1,j_2} = \spa (E_{i_1, i_2} : i_1 \in
\block_{j_1}, i_2 \in \block_{j_2})$. Finally, let $V^+ = V_{1,2}$,
$V^0 = \{\Matrix\in V_{1,1} +  V_{2,2} : \Tr[\Matrix] = 0\}$, and $V^-
= V_{2,1}$.

By \eqref{eq: at}, each of the spaces $V^+, V^-, V^0$ is an eigenspace
for $\Ad(a_t)$ with respective eigenvalues $e^{t/\pdim + t/\qdim}, \,
e^{-(t/\pdim + t/\qdim)}, \, 1.$ The action of $\Ad(K)$ preserves
$V^+, V^-, V^0$, and we can equip $V$ with an inner product which is
preserved by the $\Ad(K)$-action. 
For each $u\in
U$ and $v \in V$, we have 
\eq{eq: action of u}{
\Ad(u)v - v \in
\left \{ \begin{array}{ll}
\{0\} & \text{if } v \in V^+ \\
V^+ & \text{if } v \in V^0 \\
V^+ + V^0 & \text{if } v \in V^-
\end{array}
\right.
}

Fix $d = 1,\ldots,\dim(G) - 1$, and we will define the space
$W^{\wedge d}$. Let $\bfa\in \Lie(A)$ be chosen so that $\exp(t\bfa) = a_t$ for all $t\in\R$. 
Then the space
$V^{\wedge d}$ can be decomposed as the sum of the eigenspaces of
$\bfa$: 
\begin{equation}\label{eq: eigenspaces V^d}
V^{\wedge d} = \bigoplus_{\chi \in \Psi_d} V^{\wedge d}_\chi,
\end{equation}
where $\Psi_d$ is the collection of eigenvalues of the action of
$\bfa$ on $V^{\wedge d}$, and for each $\chi \in \Psi_d$, $V^{\wedge
  d}_\chi$ is the eigenspace of $D\rho_d(\bfa)$ with eigenvalue
$\chi$ (here $D\rho_d : \Lie(G) \to \mathrm{End}(V)$ is the
derivative of $\rho_d$ at the identity). We endow the
expressions $V^{\wedge d}_{\geq \chi}$ and 
$V^{\wedge d}_{> \chi}$ with their obvious meanings. It follows from
the remarks of the previous paragraph that for all $\chi\in\Psi_d$, 
\begin{itemize}
\item[(A)] the spaces $V^{\wedge d}_{\geq \chi}$ and $V^{\wedge d}_{>
    \chi}$ are invariant under the action of $\para$; 
\item[(B)] each $g\in \para$ acts on the quotient space $V^{\wedge
    d}_{\geq \chi}/V^{\wedge d}_{> \chi}$ as a similarity with
  expansion coefficient $e^{\chi \theta_1(g)}$; 
\item[(C)] the action of $\para$ on $V^{\wedge d}_{\geq
    \chi}/V^{\wedge d}_{> \chi}$ has only one Lyapunov exponent,
  namely $c_1 \chi$, where $c_1$ is as in \eqref{c1def}. By assumption
  (ii), we have $c_1 > 0$. 
\end{itemize}
Indeed, letting $\gamma = \tfrac{1}{\pdim} + \tfrac{1}{\qdim}$ we have
$\Psi_1 = \{-\gamma,0,\gamma\}$, $V^{\wedge 1}_\gamma = V^+$,
$V^{\wedge 1}_0 = V^0$, and $V^{\wedge 1}_{-\gamma} = V^-$, and
combining with our previous observations demonstrates the case $d =
1$. The general case follows by induction. 

Now let
\begin{equation}\label{eq: multiplication by}
W^{\wedge d} = V^{\wedge d}_{> 0} = \bigoplus_{\chi > 0} V^{\wedge d}_\chi.
\end{equation}
By (A), $W^{\wedge d}$ is invariant under $\para$, and in particular
under $\Hmu$. Since $\det \rho_d (\exp(\bfa))=1$ but $\det
\rho_d(\exp(\bfa))|_{W^{\wedge d}} > 1$, $W^{\wedge d}$ is a proper subspace
of $V^{\wedge d}$. Since $W^{\wedge 1} = V^{\wedge d}_\gamma$,
\textup{(II)} follows from (B) and (C) above. 

We now prove \textup{(I)}. To this end we will apply Lemma \ref{lem:
  furstenberg kifer} with $V = V^{\wedge d}, \, W = W^{\wedge d}$, and
obtain that $W^{\wedge d}$ is complementary to $V^{<W}(b)$. Then we
will show that for $\beta$-a.e. $b\in\B$, if $d=1$, then $V^{<W}(b) =
V_b^{<\max}$ and if $d>1$, then $V^{<W}(b)=V_b^{\leq 0}$. 

We claim that for $\beta$-a.e. $b\in\B$, all the Lyapunov exponents of
$\rho_d|_{W^{\wedge d}}$ are positive. If $d=1$ this is immediate from
assumption (ii), while if $d>1$ this follows from combining (C) above
with (b) of Lemma \ref{lem: furstenberg kifer}. 

On the other hand, let $\bar{\rho}_d$ denote the quotient action on
$V^{\wedge d}/W^{\wedge d}$. Again combining (C) above with (b) of
Lemma \ref{lem: furstenberg kifer}, we see that all the Lyapunov
exponents of $\bar{\rho}_d$ are nonpositive. In particular \eqref{eq:
  strict} holds, and $W^{\wedge d}$ is complementary to
$V^{<W}(b)$. Moreover, since all Lyapunov exponents of $W^{\wedge d}$
(resp. on $V^{<W}(b)$) are positive (resp. nonpositive), $V^{<W}(b) =
V^{\leq 0}_b$ for $\beta$-a.e. $b\in\B$, and since, in case $d=1$,
there is only one Lyapunov exponent on $W^{\wedge 1}$, we have
$V^{\leq 0}_b = V^{<\max}_b$ for $d=1$. This completes the proof of
(I). 

We now prove \textup{(III)}. Suppose that $\{\subsp_1, \ldots,
\subsp_r\}$ is a finite collection of linear subspaces of $V^{\wedge
  d}$ which is permuted by the elements of $\supp(\mu)$. Then every
element of $\Gamma$ permutes the elements of
$\{\subsp_1,\ldots,\subsp_r\}$, and thus the same is true of the
Zariski closure $\Hmu$. It follows that the identity component
$\Hmu_0$ of $\Hmu$ preserves the subspaces $\subsp_1,\ldots,\subsp_r$
individually. By assumption (iii), $\Lie(\Hmu)$ contains $V^+$, and
hence $\Hmu_0$ contains $U$. We claim that $\Hmu_0$ also contains
$A$.  To see this, recall (see
\cite[\S 15]{Borel}) that any connected real algebraic
group has a maximal $\R$-split solvable subgroup which is unique up to
conjugation. Since  $AU$ is
a maximal $\R$-split solvable 
subgroup of $P$, and it is normal in $P$, any maximal $\R$-split
solvable subgroup of $\Hmu_0$ is contained in $AU$. Let $S \subset
\Hmu_0$ be a maximal $\R$-split solvable subgroup of $\Hmu_0$
containing $U$. If $\Hmu_0$ did
not contain $AU$ we would have $U \subseteq S \subsetneqq AU$ and
thus $\pi_A(S) \subsetneqq A$, where $\pi_A$ is the algebraic
homomorphism $g \mapsto a_g$. Since $\dim A =1$ this would imply that
$\pi_A(S)$ is trivial. By \cite[Prop. 9.3]{BT}, $S$ is cocompact in
$\Hmu_0$, and so we would get that $\pi_A(\Hmu_0)$ is compact. This
would contradict 
the fact 
that $\{a_\gamma: \gamma \in \Gamma\}$ is 
infinite, which follows from assumption (ii). Therefore $\Hmu_0
\supset S = AU$, as claimed. To complete the proof it suffices to show that
any nontrivial subspace of $V^{\wedge d}$ which is $AU$-invariant must
intersect $W^{\wedge d}$ nontrivially. 

Let $Q$ be the parabolic subgroup of $G$ with Lie algebra $V^0 +
V^-$, and let $V^{\wedge d}_{\leq 0} = \bigoplus_{\chi \leq 0}
V^{\wedge d}_\chi$ be the direct sum of the $\bfa$-eigenspaces with
nonpositive eigenvalues. It is easy to check that $V^{\wedge d}_{\leq
  0}$ is $Q$-invariant, i.e. that $\rho_d(Q) V^{\wedge d}_{\leq 0} =
V^{\wedge d}_{\leq 0}$. Moreover, since $\Lie(U) = V^+$ and $\Lie(Q)
= V^0 + V^-$, the product set $Q U$ contains a neighborhood of the
identity in $G$ and in particular is Zariski dense in $G_0$, the
identity component of $G$. 

Let $\subsp \subset V^{\wedge d}$ be a nontrivial $AU$-invariant
subspace, and assume by contradiction that $\subsp \cap W^{\wedge d} =
\{0\}$. Since $\subsp$ is $A$-invariant, it can be written as a sum of
$\bfa$-eigenspaces $\subsp = \bigoplus_\chi \subsp_\chi$, and since
$\subsp \cap W^{\wedge d} = \{0\}$, we have $\subsp_\chi = \{0\}$ for
all $\chi > 0$ and thus $\subsp \subset V^{\wedge d}_{\leq 0}$. Since
$V^{\wedge d}_{\leq 0}$ is $Q$-invariant and $\subsp$ is
$U$-invariant, we have $\rho_d(Q U) \subsp \subset V^{\wedge d}_{\leq
  0}$ and thus since $Q U$ is Zariski dense in $G_0$, we have
$\rho_d(G_0) \subsp \subset V^{\wedge d}_{\leq 0}$. 

Let $\subsp' = \spa(\rho_d(G_0) \subsp) \subset V^{\wedge d}_{\leq
  0}$, and let $T \subset G_0$ be a maximal torus containing $A$. Then
since $\subsp'$ is $G_0$-invariant, it can be written as a sum of
joint eigenspaces for the $\rho_d(T)$-action, i.e. $\subsp' =
\bigoplus_{\lambda \in \Psi'} \subsp'_\lambda$, where $\Psi'$ is the
set of weights for the action of $G_0$ on $\subsp'$. The normalizer of $T$ in $G_0$ acts on
$\Psi'$ by dual conjugation: if $g\in N_{G_0}(T)$ then
$g(\subsp'_\lambda) = \subsp'_{g_* \lambda}$, where $g_* \lambda$
denotes the weight defined by the formula $g_* \lambda(\mathbf t) =
\lambda(\Ad_g^{-1} \mathbf t)$ ($\mathbf t \in \Lie(T)$). Thus, $g_*
\Psi' = \Psi'$ for all $g\in N_{G_0}(T)$. In other words, $\Psi'$ is
invariant under the Weyl group of $G_0$. It can be checked by direct
computation that if $\lambda\in \Psi'$ is a nonzero weight, then the
convex hull of $\{g_* \lambda : g\in N_{G_0}(T)\}$ contains a
neighborhood of the origin. But this implies that there exists
$\lambda' \in \Psi'$ such that $\lambda'(\bfa) > 0$, contradicting
that $\subsp' \subset V^{\wedge d}_{\leq 0}$. It follows that $\Psi'$
does not contain any nonzero weights, i.e. $\Psi' = \{0\}$. In
particular $\rho_d(T)$ acts trivially on $\subsp'$, and thus the
action of $G$ on $\subsp'$ has a nontrivial kernel. Since $G$ is
simple this means that $G$ acts trivially on $\subsp'$, and hence
$\subsp'$ is trivial, and therefore so is $\subsp$. This is a
contradiction. 
\end{proof}

We will state a useful lemma for verifying condition (iii) of
Definition \ref{def: block form}. Let $\exp$ be the exponential map
from $\Lie(G)$ to $G$, and recall that $U = \{u_{\bfalpha} : \bfalpha
\in \MM\}$. Then $\exp$ restricts to a homeomorphism from $\Lie(U)$ to
$U$. We denote the inverse of this homeomorphism by $\log$, i.e. $\log
(u) = u- \id$. As before
we let $\Gamma$ denote the group generated by $\supp(\mu)$. 

\begin{lem}\name{lem: verify iv}
Retaining the notation of Definition \ref{def: block form}, suppose
that $\mu$ satisfies \textup{(i)}, and that there exists $g_0 \in
\Gamma$ with $u_{g_0} = \id$ and $a_{g_0} \neq \id$. Then for any $g
\in \Gamma$, if we write $g = a_g k_g u_g = u'_g a_g k_g$, then the
Lie algebra of the closure of $\Gamma$ contains both $\log(u_g) $ and
$\log (u'_g)$. 
\end{lem}

\begin{proof}
Write $g_0 = a_{t_0} k_0$ and let $n_i \to \infty$ be a sequence such
that $k_0^{n_i} \to 1$. Without loss of generality suppose that $t_0 >
0$. Then
\[
g_0^{-n_i} g g_0^{n_i} = a_g (k_0^{-n_i} k_g k_0^{n_i}) (k_0^{-n_i}
a_{-n_i t_0} u_g a_{n_i t_0} k_0^{n_i}) \tendsto{i\to\infty} a_g
k_g. 
\]
It follows that $a_g k_g \in \cl{\Gamma}$ and thus $u_g \in
\cl{\Gamma}$. Applying the same logic to $ u_g $ in place of $g$ shows that 
\[
k_0^{-n_i} a_{-n_i t_0} u_g a_{n_i t_0} k_0^{n_i} \in \cl{\Gamma}
\]
and thus
\[
\lim_{i\to\infty} \frac{k_0^{-n_i} a_{-n_i t_0} u_g a_{n_i t_0}
  k_0^{n_i} - \id}{n_i t_0 \big(\tfrac{1}{\pdim} +
  \tfrac{1}{\qdim}\big)} = \log(u_g) \in \Lie(\cl{\Gamma}). 
\]
%
Since $\Lie(\cl{\Gamma})$ is closed under $\Ad(a_g k_g)$ we obtain $\log(u'_g) = \Ad(a_g k_g)(\log (u_g))\in \Lie(\Hmu)$ as well.
\end{proof}

\begin{proof}[Proof of Theorem \ref{thm: illustrative measures}]
We will apply Theorem \ref{thm: new main}, and need to check that
assumptions (I)--(III) are satisfied. Let $h_i$ be as in the
statement, and write $h_i = u'_i a_i k_i$, where for $i=1, \ldots, t$
we have 
\[
a_i
= \left[\begin{matrix}
c_i I_d & 0 \\
0 & c_i^{-d}
\end{matrix} \right], \ k_i
= \left[\begin{matrix}
O_i & 0\\
0 & 1
\end{matrix} \right], \ u'_i
= \left[\begin{matrix}
I_d & c_i^d \mathbf{y}_i \\
0 & 1
\end{matrix} \right].
\]
Then (i) and (ii) of Definition \ref{def: block form} are clearly
satisfied, and we use Lemma \ref{lem: verify iv} and the assumptions
that $\mathbf{y}_1 =0$ and $\spa(\mathbf{y}_i:i = 1,\ldots,t) = \R^d$
to verify (iii). Now the argument of Theorem \ref{thm: assumptions
  satisfied} (replacing everywhere $\PGL_\dimsum(\R)$ with
$\SL_{d+1}(\R)$) goes through. 
\end{proof}

\ignore{
\begin{remark}\label{remark: PGL1}
Note that combining Theorem \ref{thm: assumptions satisfied} with
Theorem \ref{thm: new main} yields a statement which is slightly
stronger than Theorem \ref{thm: illustrative measures}, in that it
allows for {\em orientation-reversing} orthogonal transformations;
that is allows us to take $O_i \in \GO_d(\R)$ rather than $O_i \in
\SO_d(\R)$. This is our reason for working with $\PGL_\dimsum(\R)$ in
this section, rather than $\SL_{d+1}(\R)$. \combarak{There is some
  redundancy because I also added Remark \ref{remark: PGL2}. Shoud one
  of them be omitted?} 
\end{remark}
}

\subsection{Another example}
Theorem \ref{thm: assumptions satisfied} can be generalized to $k \geq
2$ blocks as follows. Let $s_1, \ldots, s_k$ be positive integers with
$\sum s_i = \dimsum$, and for each $j = 1,\ldots,\dimsum$, let $m_j =
s_1 + \cdots + s_j$ and $\block_j = \{m_{j - 1} + 1, \ldots, m_j\}$,
with the convention that $m_0 = 0$. Then $\{\block_j : j =
1,\ldots,k\}$ is a partition of $\{1, \ldots, \dimsum\}$ into blocks
of length $s_j$, $j=1, \ldots, k$. Let $L_j = \spa\{e_i : i \in
\block_j\}$ and $E_{i_1,i_2}$ as in the proof of Theorem \ref{thm:
  assumptions satisfied}, so that $\R^\dimsum = L_1 + \cdots +
L_k$. For $j_1, j_2 \in \{1, \ldots, k\}$ let $V_{j_1,j_2} = \spa
(E_{i_1, i_2} : i_1 \in \block_{j_1}, i_2 \in \block_{j_2}),$ and let
$V^+ = \bigoplus_{j_1<j_2} V_{j_1, j_2}$. 

We say that $\mu$ is in {\em upper block form with respect to
  $\block_1, \ldots, \block_k$} if for every $g \in \supp(\mu)$ we can
write $g = aku$ for elements $a, k, u \in G$ satisfying 
\begin{itemize} 
\item[(i)$'$]
$a$ is a diagonal matrix, $k$ belongs to the compact group $\GO_{i_1}
\oplus \cdots \oplus \GO_{i_k}$, and $u\in V^+$. Here $\oplus$ denotes
the direct sum of matrices. 
\item[(ii)$'$]
For each $j = 1,\ldots,k$, the restriction of $a$ to $L_j$ is the
scalar matrix which multiplies by $e^{\theta_j(g)}$, where
$\theta_j:\supp(\mu)\to\R$ is a function such that $\int \theta_i \;
\dee \mu > \int \theta_j \; \dee \mu$ whenever $i<j$. In particular,
$a_g$ commutes with $k_{g'}$ for all $g' \in \supp(\mu)$. 
\item[(iii)$'$]
The Lie algebra of the Zariski closure of the group generated by
$\{u_g : g \in \supp(\mu)\}$ is equal to $V^+$. 
\end{itemize}

The generalization of Theorem \ref{thm: assumptions satisfied} is that
if $\mu$ is in upper block form then assumptions \textup{(I)--(III)}
are satisfied. To see this one defines $ W^{\wedge 1} = V_{1,k}$ for
$d=1$ and for $d \geq 2$ one defines a diagonal matrix $\bfa = \log(a_g)$ for
some $g \in \supp(\mu)$, and $W^{\wedge d} = \bigoplus_{\chi(\bfa)>0}
V^{\wedge d}_\chi$ in the notation of \eqref{eq: multiplication
  by}. The case $d=1$ of condition (III) follows from the irreducibility
of the adjoint representation, and the rest of the arguments in the
proof of Theorem \ref{thm: assumptions satisfied} go through with
minor modifications. We will not be using this result and leave its
verification to the reader. 

\part{Diophantine approximation on fractals}
\label{part2}
\section{Background}\name{sectionintro2}
We first recall some standard notions from Diophantine approximation
(more definitions will appear further below). A point
$\bfalpha\in\R^d$ is called \emph{badly approximable} if there exists
$c > 0$ such that for all $\pp/q\in\Q^d$, we have $\|q\bfalpha - \pp\|
\geq c q^{-1/d}$, and \emph{very well approximable} if there exists
$\vre > 0$ and infinitely many $\pp/q\in\Q^d$ such that $\|q\bfalpha -
\pp\| \leq q^{-(1/d + \vre)}$. The sets of points with these
properties are denoted respectively by $\BA$ and $ \VWA$. A point is
called \emph{well approximable} if it is not badly approximable; all
very well approximable points are well approximable but not
vice-versa. It is notoriously difficult to determine whether specific
numbers such as $\pi$ or $2^{1/3}$ are badly approximable or very well
approximable, but the properties of points typical for Lebesgue
measure are well-understood. In particular, the sets $\BA$ and $\VWA$
are both Lebesgue nullsets which nevertheless have full Hausdorff
dimension (a fact which shows that the exponent $1/d$ appearing in
both definitions is a critical exponent at which a transition
occurs). Over the last several decades, much work has revolved around
determining what properties are typical with respect to measures other
than Lebesgue measure; e.g. measures supported on fractal sets. 

Questions about Diophantine approximation on fractals can be naturally
divided into two classes: those concerned with determining the
largeness (in some sense) of the set of points on a given fractal that
are difficult to approximate by rationals, and those concerned with
determining the largeness of the set of points that are easy to
approximate by rationals. Over the last decade there has been much
progress regarding the first type of question. Suppose that $\KK$ is a
sufficiently regular fractal, so that $\mu_{\KK} \df
\HH^{\delta}|_{\KK}$ is a positive and finite measure, where $\delta$
denotes the Hausdorff dimension of $\KK$ and $\HH^\delta$ denotes
$\delta$-dimensional Hausdorff measure. This holds for example if
$\KK$ is the middle-thirds Cantor set, and for this choice we have: 
\begin{itemize}
\item the set $\BA$ has full Hausdorff dimension in $\KK$
  \cite{KleinbockWeiss1, KTV}, and 
\item the set $\VWA$ has measure zero with respect to $\mu_{\KK}$ \cite{Weiss, KLW}.
\end{itemize}
Both of these results are proven using fairly robust and
straightforward geometric methods, and are true in much greater
generality (see in particular \cite{DFSU_GE1, DFSU_BA_conformal} for
some recent results). For example, they are both true if $\KK$ is any
Ahlfors regular subset of $\R$ (a set $A \subset \R$ is called {\em
  Ahlfors regular} if there is a measure $\mu$ with $\supp(\mu) = A$
and such that for some positive constants $\delta, c_1, c_2$, for all
$x \in A$ and $r \in (0,1)$, we have $c_1 r^\delta \leq \mu (B(x,r))
\leq c_2 r^\delta$). 

The second type of question is more difficult to answer. The only
relevant work of which we are aware is the paper of Einsiedler,
Fishman, and Shapira \cite{EFS}, whose main result implies that if
$\mathcal{C}$ is the standard middle-thirds Cantor set, then 
$\mu_{\mathcal{C}}(\BA) = 0$.
Regarding very well approximable points, even the Hausdorff dimension of $\VWA\cap \mathcal{C}$ is not known (for a nontrivial lower bound, see \cite{LSV}).

There is a good reason why the second type of question is harder to
answer than the first. For both types of questions, one might expect
that a sufficiently nice fractal ``inherits'' the properties of the
ambient space, and the above results imply that for a large class of
fractals, this is true with respect to the first type of
question. However, there is a class of very nice and simple fractals
whose points do \emph{not} have typical behavior with respect to the
second type of question. Namely, for each $N\geq 2$ consider the set
$F_N$ consisting of those points in $(0,1)$ whose continued fraction
expansion has partial quotients bounded above by $N$. It is well-known
that $F_N$ consists entirely of badly approximable points (in fact, we
have $\BA\cap (0,1) = \bigcup_N F_N$, see
e.g. \cite[Theorem~23]{Khinchin_book}).

On the other hand, the set $F_N$ can be expressed as the limit set
(cf. \S\ref{subsectionsimilarityIFS}) of the finite iterated function
system consisting of the conformal contractions 
\begin{align}
\label{FN}
\phi_n(\alpha) &= \frac{1}{n + \alpha},&
n &= 1,\ldots,N.
\end{align}
This implies that $F_N$ is Ahlfors regular
\cite[Lemma~3.14]{MauldinUrbanski1}. Since Ahlfors regularity is one
of the strongest geometric properties held by the Cantor set, this
means that it will be difficult to distinguish $F_N$ from the Cantor
set using geometric properties. In particular, taking $\KK =F_N$ shows
that there are Ahlfors regular sets $\KK$ for which the expected
formula $\mu_{\KK}(\BA)=0$ fails. 

It is thus natural to ask what kind of regularity hypotheses on a fractal $\KK$ might imply that $\mu_{\KK}(\BA)=0$. We partially answer this question via Theorem \ref{thm: illustrative application}, showing that $\mu_{\KK}(\BA)=0$ whenever $\KK$ is the limit set of an irreducible finite IFS of contracting similarities. Let us point out a few cases where Theorem \ref{thm: illustrative application} applies while the results of \cite{EFS} do not apply:
\begin{itemize}
\item $\KK = \mathcal{C}+x$ is a translate of $\mathcal{C}$;
\item $\KK$ is the middle-$\vre$ Cantor set constructed by starting with the closed interval $[0,1]$ and removing at each stage the open middle subinterval of relative length $\vre$ from each closed interval kept in the previous stage of the construction, for some $\vre \in (0,1)\sm \{1/3,2/4,3/5,\ldots\}$;\footnote{When $\vre \in \{1/3,2/4,3/5,\ldots\}$, the middle-$\vre$ Cantor set falls under the framework of \cite{EFS} because it is $\times b$ invariant for some $b\geq 3$.}
\item $\KK$ is the limit set of the the iterated function system
\begin{align}
\label{1314}
\phi_1(x) &= \frac{x}{3},&
\phi_2(x) &= \frac{3 + x}{4};
\end{align}
\item $\KK$ is a fractal in higher dimensions, such as $\KK = \mathcal{C}\times \mathcal{C}\subset \R^2$.
\end{itemize}
In fact, Theorem \ref{thm: illustrative application} shows more,
namely that almost every point on the fractals listed above is
\emph{of generic type}, a term which we will define in
\S\ref{subsectionrefined}. In particular, almost every point on a
one-dimensional fractal has a typical distribution of partial
quotients in its continued fraction expansion. In addition to these
results, in what follows we will also prove several other Diophantine
results about the measures supported on self-similar fractals, as well
as considering analogous questions regarding intrinsic Diophantine
approximation on spheres \cite{KleinbockMerrill, FKMS2} and on
Kleinian lattices (cf. \cite{FSU4} and the references therein).

\section{Main results -- Similarity IFSes}
\name{sectionsim}
We begin by introducing the class of sets that we will consider.

\subsection{Similarity IFSes and their limit sets}
\label{subsectionsimilarityIFS}
We start working in higher dimensions now and accordingly fix $d\geq
1$ and an inner product on $\R^d$. A {\em contracting similarity} is a
map $\R^d \to \R^d$ of the form $\mathbf{x} \mapsto
cO(\mathbf{x})+\mathbf{y}$ where $O$ is a $d\times d$ matrix
orthogonal with respect to the chosen inner product, $c \in (0,1)$,
and $\mathbf{y} \in \R^d$. A \emph{finite similarity IFS} on $\R^d$ is
a collection of contracting similarities $\Phi = (\phi_e :
\R^d\to\R^d)_{e\in E}$ indexed by a finite set $E$, called the
\emph{alphabet}. As in Part 1, let $\B = E^\N$. However, now we let
$b^1_n$ denote the reversal of the first $n$ coordinates of $b$,
i.e. $b^1_n = (b_n,\ldots,b_1)$, in contrast to $b_1^n =
(b_1,\ldots,b_n)$ which was defined earlier. The \emph{coding map} of
an IFS $\Phi$ is the map $\pi:\B \to \R^d$ defined by the formula 
\begin{equation}
\label{codingmap}
\pi(b) = \lim_{n\to\infty} \phi_{b^1_n}(\bfalpha_0),
\end{equation}
where $\bfalpha_0\in\R^d$ is an arbitrary but fixed point, and
\begin{equation}\label{codingmap1}
\phi_{b^1_n} =
\phi_{b_1}\circ\cdots\circ\phi_{b_n}.
\end{equation}
(Note that in both \eqref{codingmap1} and \eqref{randomwalkorder}, we use the convention that $\phi_{ab} = \phi_{(a,b)} = \phi_b \circ \phi_a$.) It is easy to show that the limit in \eqref{codingmap} exists and is independent of the choice of $\bfalpha_0$, and that the coding map is continuous. Thus the image of $\B$ under the coding map, called the \emph{limit set} of $\Phi$, is a compact subset of $\R^d$, which we denote by $\KK=\KK(\Phi)$.

A similarity IFS $\Phi$ is said to satisfy the \emph{open set condition} if there exists an open set $U \subset \R^d$ such that $(\phi_e(U))_{e\in E}$ is a disjoint collection of subsets of $U$, and is said to be \emph{irreducible} if there is no affine subspace $\LL \subsetneqq \R^d$ such that $\phi_e(\LL) = \LL$ for all $e\in E$. We remark that this assumption is equivalent to the apparently stronger assumption that there is no affine subspace with a finite orbit under the semigroup generated by $\Phi$, which follows from making minor modifications to the proof of \cite[Proposition~3.1]{BFS1}. It is well-known that with these assumptions, $\mu_{\KK} = \HH^\delta|_\KK$ is a finite nonzero measure.

Using this terminology, the first part of Theorem \ref{thm: illustrative application} can be stated as follows:

\begin{thm}
\label{theoremhutchinsonHdelta}
Let $\KK$ be the limit set of an irreducible finite similarity IFS satisfying the open set condition. Then $\mu_{\KK}(\BA) = 0$.
\end{thm}

It is readily verified that the examples of fractals given in \S \ref{sectionintro2} (i.e. translates of the Cantor set $\mathcal{C}$, middle-$\vre$ Cantor sets, the limit set of \eqref{1314}, and $\mathcal{C}\times \mathcal{C}$) all satisfy the hypotheses of this theorem. The same is true for the Koch snowflake and the Sierpi\'nski triangle. On the other hand, the sets $F_N$ ($N\in\N$) cannot be written as the limit sets of similarity IFSes. Note that since the inner product used to define the notion of a similarity can be chosen arbitrarily, the class of fractals $\KK$ to which our results apply is invariant under invertible affine transformations.

We also consider more general measures on a set $\KK$ than just the Hausdorff measure $\mu_{\KK}$. Namely, let $\Prob(E)$ denote the space of probability measures on $E$. For each $\mu\in\Prob(E)$ we can consider the measure $\pi_*\mu^{\otimes\N}$ on $\KK$, i.e. the pushforward of $\mu^{\otimes \N}$ under the coding map. A measure of the form $\pi_*\mu^{\otimes \N}$ is called a {\em Bernoulli} measure. If $\Phi$ satisfies the open set condition, then there exists $\mu\in\Prob(E)$ with $\mu(e) > 0$ for all $e\in E$ such that $\mu_{\KK} = c\pi_*\mu^{\otimes\N}$ for some constant $c > 0$ \cite[(3)(iv)]{Hutchinson}. So Theorem \ref{theoremhutchinsonHdelta} is a consequence of the following more general theorem:

\begin{thm}
\label{theoremhutchinsonbernoulli}
Let $\Phi$ be an irreducible finite similarity IFS on $\R^d$, and fix $\mu\in\Prob(E)$ such that $\mu(e) > 0$ for all $e\in E$. Then $\pi_*\beta(\BA) = 0$, where $\beta = \mu^{\otimes\N}$.
\end{thm}

Note that in this theorem we do not require $\Phi$ to satisfy the open
set condition. The only reason we need the open set condition in
Theorem \ref{theoremhutchinsonHdelta} is to guarantee that $\mu_{\KK}$
is proportional to $\pi_*\beta$; if the open set condition is not
satisfied, then this equivalence does not hold, and the Hausdorff
dimension of $\KK$ does not necessarily reflect the dynamical
structure (see e.g. \cite{PRSS}). 

\subsection{More general measures}
Once we take the point of view that the Bernoulli measures associated
with an IFS are more important than the limit set of the IFS, it is
possible to relax the assumption that the IFS is finite, instead
assuming that it is compact. 
There is also no reason to restrict to uniformly contracting IFSes; it
is enough to have a ``contracting on average'' assumption. Let $E$ be
a compact set and let $\Phi = (\phi_e)_{e\in E}$ be a continuously
varying family of similarities of $\R^d$, called a \emph{compact
  similarity IFS}. We say that a measure $\mu\in\Prob(E)$ is
\emph{contracting on average} if 
\[
\int \log\|\phi_e'\| \;\dee \mu(e) < 0,
\]
where $\|\phi_e'\|$ denotes the scaling constant of the similarity
$\phi_e$ (equal to the norm of the derivative $\phi_e'$ at any point
of $\R^d$). If $\mu$ is contracting on average, then by the ergodic
theorem $\|\phi_{b^1_n}'\| \to 0$ exponentially fast for
$\beta$-a.e. $b\in \B$, and thus the limit \eqref{codingmap} converges
almost everywhere, thereby defining a measure-preserving map
$\pi:(\B,\beta) \to (\R^d,\pi_*\beta)$. In the case where all the
elements of a compact similarity IFS are strict contractions (and
thus, by compactness, contract by a uniform amount), it is easy to show that the coding map $\pi$ is continuous and thus the image of $\B$ under $\pi$ is compact. However,
in the case of contraction on average, $\pi$ is only measurable and not continuous, and the set $\pi(\B)$ need not be
compact. 

%
Now Theorem \ref{theoremhutchinsonbernoulli} is obviously a special
case of the following: 

\begin{thm}
\label{theoremCOAbernoulli}
Let $\Phi$ be an irreducible compact similarity IFS on $\R^d$, and fix
$\mu\in\Prob(E)$, contracting on average, such that $\supp(\mu) =
E$. Then $\pi_*\beta(\BA) = 0$, where $\beta = \mu^{\otimes\N}$. 
\end{thm}

\subsection{Other types of measures}
\label{subsectiondoubling} 
A completely different direction in which to generalize Theorem
\ref{theoremhutchinsonHdelta} is to consider measures on the limit set
$\KK$ other than Bernoulli measures. We will need an assumption that
ties the measure to the set $\KK$, i.e. that its topological support
is equal to $\KK$. We will also need a fairly weak geometric
assumption. A measure $\nu$ on $\R^d$ is called \emph{doubling} if for
all (equiv. for some) $\lambda > 1$, there exists a constant
$C_\lambda \geq 1$ such that for all $x\in \supp(\nu)$ and $r \in (
0,1)$, we have 
\begin{equation}
\label{doubling}
\nu(B(x,\lambda r)) \leq C_\lambda \nu(B(x,r)).
\end{equation}
\begin{thm}
\label{theoremdoubling}
Let $\KK$ be the limit set of an irreducible finite similarity IFS satisfying the open set condition. If $\nu$ is a doubling measure such that $\supp(\nu) = \KK$, then $\nu(\BA) = 0$.
\end{thm}

Since the measure $\mu_{\KK}$ is doubling and has full topological support (e.g. this follows from \cite[(3)(iii)]{Hutchinson}), Theorem \ref{theoremdoubling} provides another proof of Theorem \ref{theoremhutchinsonHdelta}. Note that we need the open set condition in Theorem \ref{theoremdoubling} in order to relate the doubling condition, which describes geometry in $\R^d$, to information about the space $\B$.

\subsection{Approximation of matrices}
The preceding theorems can be generalized to the framework of Diophantine approximation of matrices. In what follows, we fix $\pdim,\qdim\in\N$ and let $\MM$ denote the space of $\pdim\times \qdim$ matrices. Recall that a matrix $\bfalpha\in\MM$ is called \emph{badly approximable} if there exists $c> 0$ such that for all $\qq\in\Z^\qdim\sm\{0\}$ and $\pp\in\Z^\pdim$, $\|\bfalpha\qq - \pp\| \geq c\|\qq\|^{-\qdim/\pdim}$. As before, we denote the set of badly approximable matrices by $\BA$.

Rather than considering an arbitrary compact similarity IFS acting on $\MM$, we will need to be somewhat restrictive about which similarities we allow: they will need to be somewhat compatible with the structure of $\MM$ as a space of matrices. We define an \emph{algebraic similarity} of $\MM$ to be a map of the form $\bfalpha \mapsto \lambda \bfbeta \bfalpha\bfgamma + \bfdelta$, where $\lambda > 0$, $\bfbeta\in\GO_\pdim$, $\bfgamma\in\GO_\qdim$, and $\bfdelta\in\MM$. Here $\GO_\pdim$ denotes the group of $\pdim \times \pdim$ real matrices which preserve some fixed inner product on $\R^\pdim$. Thus an algebraic similarity is a composition of a translation and pre- and post-composition of $\bfalpha$ with similarity mappings on its domain and range. Note that if $\pdim = 1$ or $\qdim = 1$, then every similarity is algebraic. A similarity IFS will be called \emph{algebraic} if it consists of algebraic similarities. It will be called \emph{irreducible} if it does not leave invariant any proper affine subspace of $\MM \cong \R^{\pdim \cdot \qdim}$. For convenience we make the following definition:

\begin{dfn}\label{def: self-similar measure}
Let $\Phi$ be an irreducible compact algebraic similarity IFS on $\MM$, and fix $\mu\in\Prob(E)$, contracting on average, such that $\supp(\mu) = E$. Then the Bernoulli measure $\pi_*\beta$ is called a \emph{general algebraic self-similar measure}, where $\beta = \mu^{\otimes\N}$.
\end{dfn}

As explained in \S \ref{subsectionsimilarityIFS}, we are free to
specify our inner product structures on $\R^\pdim, \R^\qdim$ in
advance, and the groups $\GO_\pdim, \GO_\qdim$ appearing above should
be understood as the groups preserving these inner products. This
implies that the pushforward of a general algebraic self-similar
measure under a map of the form $\bfalpha \mapsto \bfbeta \bfalpha
\bfgamma + \bfdelta$, where $\bfbeta\in \GL_\pdim(\R)$, $\bfgamma\in
\GL_\qdim(\R)$, and $\bfdelta\in \MM$, is also a general algebraic
self-similar measure. 

We can now state generalizations of Theorems \ref{theoremCOAbernoulli} and \ref{theoremdoubling}, respectively:

\begin{thm}
\label{theoremBAmatrix}
If $\nu$ is a general algebraic self-similar measure on $\MM$, then $\nu(\BA) = 0$.
\end{thm}

\begin{thm}
\label{theoremdoublingmatrix}
Let $\KK$ be the limit set of an irreducible finite algebraic similarity IFS on $\MM$ satisfying the open set condition. If $\nu$ is a doubling measure such that $\supp(\nu) = \KK$, then $\nu(\BA) = 0$.
\end{thm}

Theorem \ref{theoremdoublingmatrix} will be proven in Section \ref{sectiondoubling}, while Theorem \ref{theoremBAmatrix} follows from Theorem \ref{theoremgtequi} below.

\subsection{More refined Diophantine properties}
\label{subsectionrefined}
Beyond showing that a typical point of a measure is well approximable, one can also ask about finer Diophantine properties of that point. Recall that a matrix $\bfalpha\in\MM$ is called \emph{Dirichlet improvable} if there exists $\lambda \in (0,1)$ such that for all sufficiently large $Q\geq 1$, there exist $\qq\in\Z^\qdim\sm\{0\}$ and $\pp\in\Z^\pdim$ such that $\|\qq\|_\infty \leq Q$ and $\|\bfalpha\qq - \pp\|_\infty \leq \lambda Q^{-\qdim/\pdim}$. Here $\|\cdot\|_\infty$ denotes the max norm, in contrast to the notation $\|\cdot\|$ which we use when it is irrelevant what norm we are using. Dirichlet's theorem states that this condition holds for all $\bfalpha\in\MM$ when $\lambda = 1$, so a matrix is Dirichlet improvable if and only if Dirichlet's theorem can be improved by a constant factor strictly less than 1. The concept of Dirichlet improvable matrices was introduced by Davenport and Schmidt, who showed that Lebesgue-a.e. matrix is not Dirichlet improvable, and that every badly approximable matrix is Dirichlet improvable \cite{DavenportSchmidt3}.
\ignore{
\footnote{It is possible to give a short
conceptual proof of their theorem using modern technology. If
$\bfalpha\in\MM$ is not Dirichlet improvable, then by the Dani
correspondence principle \cite[Proposition~2.1]{KleinbockWeiss5} the
orbit $(a_t u_\bfalpha x_0)_{t\geq 0}$ (see \eqref{atubfalpha} for the
notation) accumulates along the set of points in $X$ corresponding
to lattices whose intersection with $(-1,1)^{\pdim + \qdim}$ is
trivial. By Haj\'os' theorem \cite{Hajos} \comdavid{is there a more
modern reference?}\internal, every lattice in this set contains a basis
vector. Since the set of accumulation points of $(a_t u_\bfalpha
x_0)_{t\geq 0}$ is invariant under $(a_t)_{t\in\R}$, it contains the
orbit of a basis vector under $(a_t)_{t\in\R}$, and thus $(a_t
u_\bfalpha x_0)_{t\geq 0}$ accumulates at the origin. So by Dani
correspondence \cite[Theorem~2.20]{Dani4}, $\bfalpha$ is well
approximable.}
}
The converse to the last assertion is false except when $\pdim = \qdim =
1$.
Thus the following theorem gives strictly more information than Theorem \ref{theoremBAmatrix}:

\begin{thm}
\label{theoremDImatrix}
If $\nu$ is a general algebraic self-similar measure on $\MM$, then $\nu(\DI) = 0$, where $\DI$ is the set of Dirichlet improvable matrices.
\end{thm}

The properties of being well approximable and not Dirichlet improvable
both indicate that a point is ``typical'' in some sense. Another way
of indicating that a point is typical is to show that its orbit under
an appropriate dynamical system equidistributes in an appropriate
space. In dimension 1 (i.e. $\pdim = \qdim = 1$), an appropriate
dynamical system from the point of view of Diophantine approximation
is the \emph{Gauss map} 
\begin{align*}
\mathcal{G}:& (0,1)\to (0,1),&
\mathcal{G}(\alpha) &= \frac{1}{\alpha} - \left\lfloor \frac{1}{\alpha}\right\rfloor,
\end{align*}
which is invariant and ergodic with respect to the \emph{Gauss
  measure} $\dee \mu_{\mathcal{G}}(\alpha) = \frac{1}{\log(2)}\frac{\,
  \dee \alpha}{1 + \alpha}$ (see e.g. \cite[Theorems~9.7 and
9.11]{Karpenkov}). The Gauss map acts as the shift map on the
continued fraction expansion of a number, so if $\alpha\in (0,1)$,
then the forward orbit of $\alpha$ is equidistributed with respect to
the Gauss measure if and only if the continued fraction expansion of
$\alpha$ contains each possible pattern with exactly the expected
frequency. 

\begin{thm}
\label{theoremCFequi}
If $\nu$ is a general algebraic self-similar measure on $\R$, then for $\nu$-a.e. $\alpha\in\R$, the forward orbit of the point $\alpha - \lfloor \alpha\rfloor$ under the Gauss map is equidistributed with respect to the Gauss measure.
\end{thm}

In higher dimensions, there is no direct analogue of the Gauss map but
there is another dynamical system for which the orbits of points
describe their Diophantine properties: the one given by the Dani
correspondence principle \cite{Dani4, KleinbockMargulis}. Let $\dimsum
= \pdim + \qdim$, $G = \PGL_\dimsum(\R)$, $\Lambda =
\PGL_\dimsum(\Z)$, and $X = G/\Lambda$, and let $x_0$ be the element
of $X$ corresponding to the coset $\Lambda$.\footnote{As in Part 1,
  $\SL^\pm_\dimsum(\R)$ and $\SL^\pm_\dimsum(\Z)$ denote respectively
  the groups of $\dimsum \times \dimsum$ real (integer) matrices of
  determinant $\pm 1$, and $\PGL_\dimsum(\R), \, \PGL_\dimsum(\Z)$ are
  their factor groups obtained by identifying matrices which differ by
  multiplications by scalars.} As in Part 1, for each $t\in\R$ and
$\bfalpha\in \MM$, let 
\begin{align}
\label{atualpha}
a_t &= \left[\begin{array}{ll}
e^{t/\pdim}I_\pdim &\\
& e^{-t/\qdim}I_\qdim
\end{array}\right],&
u_\bfalpha &= \left[\begin{array}{ll}
I_\pdim & -\bfalpha\\
& I_\qdim
\end{array}\right],
\end{align}
which we consider as elements of $\PGL_\dimsum(\R)$ by identifying a
matrix with its equivalence class. Then the Dani correspondence
principle says that the forward orbit $(a_t u_\bfalpha x_0)_{t\geq 0}$
encodes the Diophantine properties of the matrix $\bfalpha$. We will
say that $\bfalpha$ is {\em of generic type} if the orbit $(a_t
u_\bfalpha x_0)_{t\geq 0}$ is equidistributed in $X$ with respect to
the $G$-invariant probability measure on $X$. 

\begin{remark}\label{remark: PGL2}
Note that in \cite{Dani4} (and most subsequent papers) the space $X' =
\SL_{\dimsum}(\R)/ \SL_\dimsum(\Z)$ was used instead of $X$. But the
natural map $X' \to X$ (induced by the homomorphism $\SL_\dimsum(\R)
\to \PGL_\dimsum(\R)$) is an equivariant isomorphism of homogeneous
spaces and hence does not affect the definition of generic type. Using
$\PGL_\dimsum(\R)$ will make it possible to encode more general maps
coming from orthogonal transformations that are not
orientation-preserving. 
\end{remark}

\begin{thm}
\label{theoremgtequi}
If $\nu$ is a general algebraic self-similar measure on $\MM$, then
$\nu$-a.e. $\bfalpha\in\MM$ is of generic type. 
\end{thm}

Since an equidistributed orbit is dense, \cite[Theorem~2.20]{Dani4}
and \cite[Proposition~2.1]{KleinbockWeiss5} show that Theorem
\ref{theoremgtequi} implies Theorems \ref{theoremBAmatrix} and
\ref{theoremDImatrix}, respectively. When $\pdim = \qdim = 1$, the
equidistribution of the 
orbit $(a_t u_\alpha x_0)_{t\geq 0}$ implies the equidistribution of
$(\mathcal{G}^n(\alpha))_{n\in\N}$, in other words Theorem
\ref{theoremCFequi} follows from Theorem \ref{theoremgtequi}. The
converse however is false, see 
Section \ref{sectionCFequi} for details. Theorem
\ref{theoremgtequi} will be proven in Section \ref{sectionbernoulli}. 

\begin{remark}
Einsiedler, Fishman, and Shapira actually proved more than just $\mu_{\KK}(\BA)=0$: they showed that if $\nu$ is any measure on $\R/\Z$ invariant under the $\times k$ map for some $k\geq 2$, then for $\nu$-a.e. $\alpha\in\R$, the orbit $(a_t u_\alpha x_0)_{t\geq 0}$ is dense in $X$, and $\alpha$ has all finite patterns in its continued fraction expansion. Theorem \ref{theoremgtequi} improves density to equidistribution. See \cite{Shi} for another result in this direction.
\end{remark}

\section{Main results -- M\"obius IFSes}
\label{sectionconf}
Theorems regarding similarity IFSes can often be extended to the realm of \emph{conformal IFSes}, whose definition is somewhat technical (see e.g. \cite[p.6]{MauldinUrbanski1}), or to the subclass of \emph{M\"obius IFSes}, which can be defined more succinctly (see \S\ref{subsectionmobiusIFS} below). However, we know that the results of the previous section cannot be extended directly, because the sets $F_N$ can be written as the limit sets of M\"obius IFSes, even though they contain only badly approximable points. The reason for this appears to be a very special coincidence, namely the fact that the defining transformations of the IFS defining $F_N$ are all represented by elements of the integer lattice $\Lambda = \PGL_2(\Z) \subset G = \PGL_2(\R)$ (cf. \eqref{FN}). In fact, it turns out that the limit set of any M\"obius IFS with this property consists entirely of badly approximable numbers; see Theorem \ref{theoremmobiusIFS}(i) below. Thus, an additional restriction will be needed in order to rule out this case and similar cases.

It is also natural to ask about higher dimensions, but here the situation is less clear. The reason for this is that the Diophantine structure of $\R^d$ is naturally related to the group $G = \PGL_d(\R)$ of projective transformations on $\R^d$, and this group is the same as the group of M\"obius transformations if $d = 1$ but not in higher dimensions. On the other hand, a Diophantine setting that is naturally related to the group of M\"obius transformations is the setting of \emph{intrinsic approximation on spheres}, which has been studied by Kleinbock and Merrill \cite{KleinbockMerrill} and related to hyperbolic geometry by Fishman, Kleinbock, Merrill, and the first-named author \cite[\S3.5]{FKMS2}. In this setting, points on the unit sphere $S^d \subset \R^{d + 1}$ are approximated by rational points of $S^d$. When $d = 1$, there is a conformal isomorphism between $S^1$ and $\R^1$ that preserves Diophantine properties, given by stereographic projection; in higher dimensions stereographic projection still provides a conformal isomorphism between $S^d$ and $\R^d$, but this isomorphism does not preserve Diophantine properties. Moving the Diophantine structure from $S^d$ to $\R^d$ yields a structure on $\R^d$ that is naturally related to the group of M\"obius transformations.

In what follows, we will show that if $\KK$ is the image under stereographic projection of the limit set of a conformal iterated function system on $\R^d$, then almost every point of $\KK$ is not badly approximable with respect to intrinsic approximation on $S^d$.

The proofs in this section use the results of Benoist and Quint directly, without appealing to Part 1.

\subsection{M\"obius IFSes}
\label{subsectionmobiusIFS}
A \emph{M\"obius transformation} of $\wbar{\R^d} = \R^d\cup\{\infty\}$
is a finite composition of spherical inversions and reflections in
hyperplanes. See e.g. \cite{Hertrich} for an introduction to the
geometry of M\"obius transformations. A \emph{(finite) M\"obius IFS}
on $\wbar{\R^d}$ is a finite collection of M\"obius transformations
$\Phi = (\phi_e : \wbar{\R^d}\to\wbar{\R^d})_{e\in E}$ such that for
some nonempty compact set $\FF \subset \wbar{\R^d}$, for all $e\in E$,
we have $\phi_e(\FF) \subset \FF$, and $\phi_e\given_\FF$ is a strict
contraction relative to some Riemannian metric independent of
$e$.\footnote{Any M\"obius IFS according to this definition that satisfies the open set condition is (after possibly passing to an iterate) a conformal IFS according to the definition given in \cite[p.6]{MauldinUrbanski1}. To see this, let $U$ be the set coming from the open set condition, and let $X$ be the intersection of $\cl U$ with a closed neighborhood of $\FF$ small enough so that $\Phi$ is still strictly contracting on $X$, and smooth enough so that the cone condition holds. Then let $V$ be a slightly larger open neighborhood. It is obvious that \cite[(2.6)-(2.8)]{MauldinUrbanski1} hold, and \cite[(2.9)]{MauldinUrbanski1} follows from \cite[Remark~2.3]{MauldinUrbanski1}.} As in the case of similarity
IFSes the \emph{coding map} $\pi:\B \to \FF, \ \B = E^\N$ is defined
by the formula \eqref{codingmap}, with the additional restriction that
$\bfalpha_0\in \FF$ (otherwise the limit may not exist). Similarly, a
M\"obius IFS $\Phi$ is said to satisfy the \emph{open set condition}
if there exists a nonempty open set $U \subset \wbar{\R^d}$ such that
$(\phi_e(U))_{e\in E}$ is a disjoint collection of subsets of
$U$. Finally, $\Phi$ is \emph{irreducible} if there is no generalized
sphere $\LL \subsetneqq \wbar{\R^d}$ such that $\phi_e(\LL) = \LL$ for
all $e\in E$. Here a \emph{generalized sphere} in $\wbar{\R^d}$ is
either an affine subspace of $\wbar{\R^d}$ (including the point at
infinity) or a sphere inside of a (not necessarily proper) affine
subspace of $\wbar{\R^d}$. Note that in dimension 1, a nonempty proper
generalized sphere is just a point. For the purposes of this paper, we
consider $\{\infty\}$ to be a generalized sphere. Since $\{\infty\}$ is invariant under all similarities, this means that the classes of similarity IFSes and irreducible M\"obius IFSes are disjoint. 

The group of M\"obius transformations on $\R$ is isomorphic to $G = \PGL_2(\R)$, where each matrix $\left[\begin{smallmatrix} a & b \\ c & d \end{smallmatrix}\right] \in \PGL_2(\R)$ represents the M\"obius transformation $x \mapsto \frac{ax + b}{cx + d}$. In what follows we implicitly identify these two groups via this isomorphism.

\begin{thm}
\label{theoremmobiusIFS}
Let $\Phi = (\phi_e)_{e\in E}$ be an irreducible finite M\"obius IFS on $\R$ satisfying the open set condition, and let $\KK$ be its limit set. Let $\Gamma$ denote the group generated by $\Phi$.
\begin{itemize}
\item[(i)] If $\Gamma$ is virtually contained in $\Lambda \df \PGL_2(\Z)$, then $\KK \subset \BA$.
\item[(ii)] Suppose that $\Gamma$ is not virtually contained in any group of the form $g\Lambda g^{-1}$ $(g\in G)$. Then $\mu_{\KK}(\BA) = 0$, and more generally, if $\nu$ is a doubling measure on $\KK$ such that $\supp(\nu) = \KK$, then $\nu(\BA) = 0$.
\end{itemize}
\end{thm}

Recall that a subgroup $\Gamma$ of a group $G$ is \emph{virtually contained} in another subgroup $\Lambda \subset G$ if some finite index subgroup of $\Gamma$ is contained in $\Lambda$.

\begin{example}\label{ex: Fn}
The system of M\"obius transformations \eqref{FN} is an irreducible M\"obius IFS. So the set $F_N$, and all of its translations, are the limit sets of irreducible M\"obius IFSes. Thus Theorem \ref{theoremmobiusIFS} says that for all $\alpha\in\Q$, we have $F_N + \alpha \subset \BA$ (this also follows directly). However, Theorem \ref{theoremmobiusIFS} does not say anything about the sets $F_N + \alpha$ where $\alpha$ is irrational, because then the corresponding IFS $\Phi$ falls into neither case (i) nor case (ii).

It follows from Theorem \ref{theoremkleinian2} below that if $\alpha$
is irrational, then any Bernoulli measure on $F_N + \alpha$ gives zero
measure to the set of badly approximable points. However, the natural
measure $\mu_{F_N + \alpha} = \HH^\delta\given_{F_N + \alpha}$ (where
$\delta = \dim_H(F_N)$) is not a Bernoulli measure, and our results
say nothing about this measure. 
\end{example}

\begin{example}
If the IFS $\Phi = (\phi_a)_{a\in E}$ contains at least two
similarities with distinct fixed points, but is not entirely composed
of similarities, then we are in case (ii). This is because it follows
from applying Lemma \ref{lem: verify iv} to the subgroup of $\Gamma$
generated by these two similarities (thinking of it as a subgroup of
the Lie group of all similarities) that the closure of $\Gamma$ contains a
positive-dimensional unipotent subgroup. Therefore it cannot have a finite index
subgroup contained in $g \Lambda g^{-1}$ for any $g\in G$. 
\end{example}

\subsection{Intrinsic approximation on spheres}
Fix $d\geq 1$, and let $S^d$ be the unit sphere in $\R^{d + 1}$. We
recall that a point $\bfalpha\in S^d$ is \emph{badly approximable with
  respect to intrinsic approximation on $S^d$}, or just \emph{badly
  intrinsically approximable}, if there exists $c> 0$ such that for
all $\pp/q\in \Q^{d + 1}\cap S^d$, we have $\|q\bfalpha - \pp\| \geq
c$. The set of badly intrinsically approximable points is similar in
many ways to the set of badly approximable points; for example, it has
full Hausdorff dimension but zero Lebesgue measure
\cite{KleinbockMerrill}. We denote the set of badly intrinsically
approximable points by $\BA_{S^d}$. 

We define a \emph{M\"obius IFS on $S^d$} to be a M\"obius IFS on $\R^{d + 1}$ that preserves $S^d$. Such an IFS is said to be \emph{irreducible (relative to $S^d$)} if it does not preserve any generalized sphere $\LL \subsetneqq S^d$. Let $G$ (resp. $\Lambda$) denote the group $\PO(d + 1,1;\R)$ (resp. $\PO(d+1,1; \Z)$) of $(d+2) \times (d+2)$ real (resp. integer) matrices preserving the quadratic form $Q(x_0,x_1,\ldots,x_{d+1}) = -x_0^2 + x_1^2 + \cdots + x_{d+1}^2$, where matrices which are scalar multiples of each other are identified. Note that the group of M\"obius transformations that preserve $S^d$ is isomorphic to $G$ via the following isomorphism: each element $g\in G$ acts conformally on $S^d$ via the restriction of a projective transformation of $\mathbb{P}^{d + 1}(\R) \supset \R^{d + 1}$, and this conformal isomorphism of $S^d$ extends uniquely to a M\"obius transformation of $\wbar{\R^{d + 1}}$. (The resulting M\"obius transformation is not the same as the projective action of $g$ on $\R^{d + 1}$, unless $g$ preserves the origin of $\R^{d + 1}$.) Using this identification, we can now state the following theorem:

\begin{thm}
\label{theoremBAintrinsic}
Let $G, \, \Lambda$ be as above, let $\Phi = (\phi_e)_{e\in E}$ be an irreducible finite M\"obius IFS on $S^d$ satisfying the open set condition, and let $\KK$ be its limit set. Let $\Gamma \subset G$ denote the group generated by $\Phi$.
\begin{itemize}
\item[(i)] If $\Gamma$ is virtually contained in $\Lambda$, then $\KK \subset \BA_{S^d}$.
\item[(ii)] Suppose that there is no $g \in G$ for which $\Gamma$ is virtually contained in $g\Lambda g^{-1}$. Then $\mu_{\KK}(\BA_{S^d}) = 0$, and more generally, if $\nu$ is a doubling measure on $\KK$ such that $\supp(\nu) = \KK$, then $\nu(\BA_{S^d}) = 0$.
\end{itemize}
\end{thm}

\subsection{Kleinian lattices}\label{subsec: kleinian}
We conclude this section by considering an approximation problem in
hyperbolic geometry that generalizes both of the setups considered
above. Let $\Hh^{d + 1}$ denote $(d + 1)$-dimensional hyperbolic
space, let $G = \Isom(\Hh^{d + 1})$, and let $\Lambda \subset G$ be a
lattice. A point $\bfalpha\in \partial\Hh^{d + 1}$ is said to be
\emph{uniformly radial} with respect to $\Lambda$ if any geodesic ray
with endpoint $\bfalpha$ stays within a bounded distance of the orbit
$\Lambda(\zero)$, where $\zero\in \Hh^{d + 1}$ is arbitrary but
fixed. We denote the set of uniformly radial points of $\Lambda$ by
$\UR_\Lambda$. Uniformly radial points can also be thought of as
``badly approximable with respect to the parabolic points of
$\Lambda$''; see \cite[Proposition~1.21]{FSU4}. In particular, 
\begin{itemize}
\item If $\Hh^2$ is the upper half-plane model of hyperbolic geometry,
  then $\partial\Hh^2 = \wbar\R$, and the parabolic points of the
  lattice $\Lambda \df \PGL_2(\Z) \subset G \df \PGL_2(\R)$ are
  exactly the rational points of $\wbar\R$ (including $\infty$). The
  heights of these rational points correspond to the diameters of an
  invariant collection of horoballs centered at these points, which
  implies that $\UR_\Lambda = \BA$ \cite[Obs. 1.15 and 1.16 and
  Proposition~1.21]{FSU4}. 
\item If $\Hh^{d + 1}$ is the Poincar\'e ball model of hyperbolic geometry, then $\partial\Hh^{d + 1} = S^d$, and the parabolic points of the lattice $\Lambda = \PO(d + 1,1;\Z) \subset G = \PO(d + 1,1;\R)$ are exactly the rational points of $S^d$. Again the heights of these rational points correspond to the diameters of horoballs, so $\UR_\Lambda = \BA_{S^d}$ \cite[\S3.5]{FKMS2}.
\end{itemize}
These facts show that the following theorem generalizes both Theorem \ref{theoremmobiusIFS} and Theorem \ref{theoremBAintrinsic}:

\begin{thm}
\label{theoremkleinian1}
Let $\Phi = (\phi_e)_{e\in E}$ be an irreducible finite M\"obius IFS on $\partial\Hh^{d + 1}$ satisfying the open set condition, and let $\KK$ be its limit set. Let $\Gamma$ denote the group generated by $\Phi$, and let $\Lambda \subset G = \Isom(\Hh^{d + 1})$ be a lattice.
\begin{itemize}
\item[(i)] If $\Gamma$ is virtually contained in $\Lambda$, then $\KK \subset \UR_\Lambda$.
\item[(ii)] Suppose that there is no $g \in G$ for which $\Gamma$ is virtually contained in $g\Lambda g^{-1}$. Then $\mu_{\KK}(\UR_\Lambda ) = 0$, and more generally, if $\nu$ is a doubling measure on $\KK$ such that $\supp(\nu) = \KK$, then $\nu(\UR_\Lambda) = 0$.
\end{itemize}
\end{thm}

In this theorem, $\Hh^{d + 1}$ can be interpreted as either the Poincar\'e ball model of hyperbolic geometry (in which case $\partial\Hh^{d + 1} = S^d$), or as the upper half-space model (in which case $\partial\Hh^{d + 1} = \wbar{\R^d}$). Either way, the group of M\"obius transformations on $\partial\Hh^{d + 1}$ is isomorphic to $\Isom(\Hh^{d + 1})$, which explains how the M\"obius transformations $(\phi_e)_{e\in E}$ can be identified with elements of $G$. In what follows we will not distinguish between a M\"obius transformation and its corresponding isometry of $\Hh^{d + 1}$, but it should be observed that the M\"obius transformation is not itself an isometry of the space $\partial\Hh^{d + 1}$, but only a conformal map. If we interpret $\Hh^{d + 1}$ as the upper half-space model, then we should assume that $\infty\notin \KK$, so that $\KK$ inherits a metric from $\R^d$ with respect to which the notion of a doubling measure can be interpreted. Theorem \ref{theoremkleinian1} will be proven in Section \ref{sectiondoubling}.

We can relax the assumptions that $\Phi$ is finite, contracting on
some set $\FF$, and satisfies the open set condition if we consider a
more restricted class of measures, namely the class of Bernoulli
measures. This restriction will also allow us to improve the
conclusion of Theorem \ref{theoremkleinian1}(ii), and to bypass the
obstruction that occurs when $\Gamma$ is virtually contained in some
$g\Lambda g^{-1} \neq \Lambda$ (the obstruction that occurs when
$\Gamma$ is virtually contained in $\Lambda$ remains). We define a
\emph{compact M\"obius IFS} on $\partial\Hh^{d + 1}$ to be a
continuously varying family of M\"obius transformations $\Phi =
(\phi_e \in \Isom(\Hh^{d + 1}))_{e\in E}$, where $E$ is a compact
set. Note that in this definition, we do not assume that the family
$\Phi$ is contracting in any sense. We call $\Phi$ \emph{irreducible}
if it does not preserve any generalized sphere $\LL \subsetneqq S^d$,
nor any point of $\Hh^{d + 1}$. Given an irreducible compact M\"obius
IFS $\Phi$ and a measure $\mu\in \Prob(E)$ such that $\supp(\mu) = E$,
for $\beta$-a.e. $b\in \B $, the limit 
\begin{equation}
\label{codingmap2}
\pi(b) = \lim_{n\to\infty} \phi_{b^1_n}(\zero)
\end{equation}
exists in $\partial\Hh^{d + 1}$, where $\zero\in \Hh^{d + 1}$ is a distinguished point and $\phi_{b^1_n}$ is as in \eqref{codingmap1} (see \cite{MaherTiozzo}). Thus we can define the measure $\pi_*\beta$ on $\partial\Hh^{d + 1}$.

\begin{thm}
\label{theoremkleinian2}
Let $\Phi = (\phi_e)_{e\in E}$ be an irreducible compact M\"obius IFS
on $\partial\Hh^{d + 1}$. Let $\Gamma$ be the group generated by
$\Phi$, and let $\Lambda \subset G = \Isom(\Hh^{d + 1})$ be a
lattice. Suppose that $\Gamma$ is not virtually contained in
$\Lambda$. Then for all $\mu\in\Prob(E)$ such that $\supp(\mu) = E$,
we have $\pi_*\beta(\UR_\Lambda) = 0$, where $\beta =
\mu^{\otimes\N}$. Moreover, for $\beta$-a.e. $b\in \B$, any geodesic
ray ending at $\pi(b)$ is equidistributed in the unit tangent bundle
$T^1\Hh^{d + 1}/\Lambda \cong K \backslash G / \Lambda$ (where $K$ is
the maximal compact subgroup of $G$ fixing a distinguished tangent
vector at $\zero$). 
\end{thm}

\medskip

{\bf Summary.}
The theorems of \S \ref{sectionsim} and Theorem \ref{thm: illustrative application} all reduce to three theorems: \ref{theoremdoublingmatrix}, \ref{theoremCFequi}, and \ref{theoremgtequi}. The theorems of this section all reduce to two theorems: \ref{theoremkleinian1} and \ref{theoremkleinian2}. We will then prove these theorems in Sections \ref{sectiondoubling}-\ref{sectionCFequi}.

\section{Relation to the random walk setup}
\label{sectionprelim}
In this section we restate the results we will use from Part 1 of this paper, and from \cite{BenoistQuint7}. We use the following notation for all of the theorems below:

\begin{itemize}
\item $G$ is a semisimple real algebraic group with no compact factors, $\Lambda$ is a lattice in $G$, $X = G/\Lambda$, and $m_X$ is the $G$-invariant probability measure on $X$ obtained from Haar measure on $G$ (in some cases below $G$ and $\Lambda$ will be made more specific). The point $x_0 \in X$ corresponds to the coset $\Lambda$.
\item $E$ is a compact set, $e \mapsto g_e $ is a continuous map from $E$ to $G$, and $\mu\in\Prob(E)$ is a measure such that $\supp(\mu) = E$.
\item $\Gamma^+$ (resp. $\Gamma$) is the semigroup (resp. group) generated by $\{g_e : e\in E\}$.
\item For $b = (e_1, e_2, \ldots) \in \B$, and $n \in \N$, $g_{b_1^n}$ denotes the product $g_{e_n} \cdots g_{e_1}$.
\end{itemize}

By combining  Theorems \ref{thm: new main} and \ref{thm: assumptions satisfied} of Part 1, we immediately obtain the following:

\begin{thm}
\label{theorempart1}
Let $\pdim, \qdim$ be positive integers, let $\dimsum = \pdim + \qdim$, and let $G = \PGL_\dimsum(\R), \, \Lambda = \PGL_\dimsum(\Z), \, X = G/\Lambda$. Let $\mu$ be a probability measure with compact support $E \subset G$ which is in $(\pdim, \qdim)$-upper block form (see Definition \ref{def: block form}). Then for all $x\in X$,
\begin{itemize}
\item[(i)] $\Gamma^+ x$ is dense in $X$.
\item[(ii)] For $\beta$-a.e. $b\in \B$, the random walk trajectory
\begin{equation}
\label{randomwalk}
\big(g_{b_1^n} x\big)_{n\in\N}
\end{equation}
is equidistributed in $X$ with respect to $m_X$.
\end{itemize}
\end{thm}


We will also use:
\begin{thm}[Benoist-Quint, see {\cite[Theorems~1.1 and 1.3]{BenoistQuint7}}]
\name{theoremBQ}
Suppose that $\Gamma^+$ is Zariski dense in $G$. Then for all $x\in X$, there exist a closed group $\HBQ \subset G$ containing $\Gamma^+$ and an $\HBQ$-invariant probability measure $\nu_x$ such that $\supp(\nu_x) = \HBQ x$ and:
\begin{itemize}
\item[(i)] $\Gamma^+ x$ is dense in $\HBQ x$.
\item[(ii)] For $\beta$-a.e. $b\in \B$, the random walk trajectory \eqref{randomwalk} is equidistributed in $\HBQ x$ with respect to $\nu_x$.
\end{itemize}
\end{thm}

\begin{remark}
\label{remarkBQ}
If the identity component of $G$ is simple in Theorem \ref{theoremBQ}, then the group $\HBQ$ is either discrete or of finite index in $G$. This is because the adjoint action of $\Gamma^+$ on $\Lie(G)$ normalizes $\Lie(\HBQ)$, so since $\Gamma^+$ is Zariski dense, the adjoint action of $G$ normalizes $\Lie(\HBQ)$ as well, and thus either $\Lie(\HBQ) = \{0\}$ or $\Lie(\HBQ) = \Lie(G)$.

If $\HBQ$ is discrete, then $\nu_x$ is atomic and gives the same measure to every atom, and thus $\HBQ x$ is finite. In this case $\HBQ$ acts by permutations on $\HBQ x$, so a finite index subgroup of $\HBQ$ is contained in $\Stab_G(x) = g\Lambda g^{-1}$, where $x$ is the coset $g\Lambda$.

If $\HBQ$ is of finite index, then $\nu_x$ is the (renormalized) restriction of the natural measure $m_X$ on $X$ to one or more connected components of $X$. In particular, if $X$ is connected (which is true in the examples we consider), then $\nu_x=m_X$.
\end{remark}

The following is an immediate consequence of Theorem \ref{thm: fiber bundle
extension}, also proven in Part \ref{part1}.

\begin{thm}
\label{theoremequibootstrap}
Fix $x\in X$, and suppose that for $\beta$-a.e. $b\in \B$, the random walk trajectory $(g_{b_1^n}x)_{n \in \N}$ is equidistributed in $X$ with respect to $m_X$. Let $K$ be a compact group, let $\kappa:\Gamma\to K$ be a homomorphism, and for each $e\in E$ let $k_e = \kappa(g_e)$. Let $\bar K$ denote the closure of $\kappa(\Gamma)$ and let $m_{\bar K}$ denote Haar measure on $\bar K$, and assume that $\Gamma$ acts ergodically on $(X \times \bar K, m_X \otimes m_{\bar K})$. Finally, let $\bar \B = E^{\Z}$, $\bar \beta = \mu^{\otimes\Z}$, let $Y$ be a locally compact topological space, and let $f:\bar{\B} \to Y$ be a measurable transformation. Then for $\bar \beta$-a.e. $b\in \bar{\B}$, the sequence
\begin{equation}
\label{randomwalkandmore}
\big(g_{b_1^n} x,k_{b_1^n} ,f(T^n b)\big)_{n\in\N}
\end{equation}
is equidistributed in $X\times \bar K\times Y$ with respect to $m_X\otimes m_{\bar K}\otimes f_*\bar\beta$.
\end{thm}

\subsection{Relation to the setups considered in Sections \ref{sectionsim} and \ref{sectionconf}}\name{subsec: fundamental}
Now we show that the hypotheses of the above theorems are satisfied in the setups considered in \S\ref{sectionsim}-\S\ref{sectionconf}, which we summarize as follows:
\begin{itemize}
\item[Setup 1.] In \S\ref{sectionsim}, the fundamental objects are an irreducible compact algebraic similarity IFS $\Phi = (\phi_e)_{e\in E}$ on the space $\MM$ of $\pdim \times \qdim$ matrices, a contracting-on-average measure $\mu\in \Prob(E)$ such that $\supp(\mu) = E$, the groups $G = \PGL_\dimsum(\R), \, \Lambda = \PGL_\dimsum(\Z)$, and the homogeneous space $X = G/\Lambda$.
\item[Setup 2.] In \S\ref{sectionconf}, the fundamental objects are an irreducible compact M\"obius IFS $\Phi = (\phi_e)_{e\in E}$ on $\partial\Hh^{d + 1}$, a measure $\mu\in\Prob(E)$ such that $\supp(\mu) = E$, and a lattice $\Lambda \subset G = \Isom(\Hh^{d + 1})$.
\end{itemize}
We will explain how to connect Setups 1 and 2 with the homogeneous space random walks setup introduced in this section. In both setups the objects $G$, $\Lambda$, $E$, and $\mu$ are already defined, so it remains to define the family $(g_e)_{e\in E}$. In Setup 2 we notice that the M\"obius transformations $\phi_e \; (e\in E)$ are already members of $G$, so they define a family $(g_e)_{e\in E}$ via the formula $g_e = \phi_e^{-1}$. Note that taking the inverse in this definition ensures that the expressions $g_{b_1^n}$ and $\phi_{b_n^1}$ appearing respectively in the definitions of the random walk and the coding map (see \eqref{randomwalk} and \eqref{codingmap}) are related by the formula $g_{b_1^n} = (\phi_{b_n^1})^{-1} \; (b\in \B, \, n\in\N)$.

\ignore{
The reason for taking inverses is essentially that the ordering
$g_{b_n}\cdots g_{b_1}$ appearing in \eqref{randomwalk} and the
ordering $\phi_{b_1}\cdots\phi_{b_n}$ appearing in \eqref{codingmap}
are exactly opposite to each other. Another way to think of it is that
the family $(\phi_e)_{e\in E}$ can be thought of as acting on the
space $\Hh^{d + 1} = G/K$, where $K =\GO(d + 1)$ is a maximal compact
subgroup of $G$, while the elements $g_e$ are thought of as acting on
$X = G/\Lambda$. These two actions are in some sense ``inconsistent''
with each other, but can be made consistent by turning one of them
into a right action: $K \backslash G / \Lambda$. Thus $(\phi_e)_{e\in
  E}$ can be thought of as acting ``on the right'' while $(g_e)_{e\in
  E}$ can be thought of as acting ``on the left''. } 

In Setup 1, we will also define the family $(g_e)_{e\in E}$ via the
formula $g_e = \phi_e^{-1}$, but it takes a little more work to
describe how to view the algebraic similarities $\phi_e \; (e\in E)$
as elements of $G = \PGL_{\dimsum}(\R)$. We recall that in \S
\ref{subsec: main example} we defined subgroups $A,K,U \subset G$ by: 
\begin{align}
\label{KAUH}
A &= \{a_t : t\in\R\},&
K &= \GO_\pdim\oplus \GO_\qdim,&
U &= \{u_\bfalpha : \bfalpha\in \MM\}
\end{align}
(where as before matrices are identified with their images in $G$), and we let
\begin{equation}\label{eq: and set}
\para=AKU.
\end{equation}
Note that $A$ and $K$ commute with each other and normalize $U$, and thus the natural projections
\[\pi_A:\para\to A \text{ and } \pi_K:\para\to K\]
are homomorphisms. Let $\iota:\MM\to \para/AK$ be defined by the
formula $\iota(\bfalpha) = u_\bfalpha AK$. Then $\iota$ is a
homeomorphism, and $\iota(\mathbf 0)$ is the identity
coset $AK\in \para/AK$. Now consider the action $\rho$ of $\para$ on
$\MM$ that results from conjugating the action of $\para$ on
$\para/AK$ by left multiplication by the isomorphism $\iota$. It is
readily checked that $\rho(u_\bfalpha)(\bfbeta) = \bfbeta + \bfalpha$,
$\rho(a_t)(\bfbeta) = e^{t/\pdim+t/\qdim} \bfbeta$, and
$\rho(O_1\oplus O_2)(\bfbeta) = O_1 \bfbeta O_2^{-1}$. In particular
$\rho$ is faithful (since $\para\subset \PGL_{\dimsum}(\R)$ and thus
multiplication by $-1$ is considered trivial), and $\rho(\para)$ is
the group of algebraic similarities of $\MM$. So $\rho$ is an
isomorphism between $\para$ and the group of algebraic similarities of
$\MM$. By identifying each element of $\para$ with its image under
$\rho$, we can think of the algebraic similarities $\phi_e \; (e\in
E)$ as elements of $\para \subset G$, and from there define the family
$(g_e)_{e\in E}$ by the formula $g_e = \phi_e^{-1}$. Note that this
paragraph is the reason we needed to consider algebraic similarities,
rather than all similarities, in Theorems
\ref{theoremBAmatrix}---\ref{theoremgtequi}. 

We now show that we can apply Theorems \ref{theorempart1} and
\ref{theoremequibootstrap} in Setup 1, and Theorem \ref{theoremBQ}
in Setup 2.

\begin{itemize}
\item Let $\Phi = (\phi_e)_{e\in E}$ be an irreducible compact algebraic similarity IFS, where $E$ is a compact indexing set, and let $\mu \in \Prob(E)$ be a contracting-on-average measure such that $\supp(\mu) = E$. By replacing $E$ and $\mu$ with their images under the map $e \mapsto g_e = \phi_e^{-1}$, we can without loss of generality assume that $E$ is a subset of $G$ and that $g_e = e$ for all $e\in E$. We want to apply Theorem \ref{theorempart1} to show that for any $x \in X$, for $\beta$-a.e. $b\in\B$, the associated random walk trajectory \eqref{randomwalk} is equidistributed in $X$.

Note that replacing $\mu$ by its pushforward under a conjugation in $G$ does not affect the validity of this conclusion; indeed, if \eqref{randomwalk} is equidistributed then so is $(g_0 g_{b_1^n} x)_{n \in \N} = (g_0 g_{b_1^n} g_0^{-1} g_0 x)_{n \in \N}$, which is the random walk corresponding to the pushforward of $\mu$ under conjugation by $g_0$ and the initial point $g_0 x$. Taking an element of the semigroup generated by $\Phi$ which acts on $\MM$ as a contraction and translating the fixed point to the origin, we can assume with no loss of generality that $\supp(\mu)$ contains an element $h_0 \in AK$ with $\pi_A(h_0) = a_t, \; t > 0$. After this conjugation, let us show that the measure $\mu$ satisfies conditions (i)--(iii) of Definition \ref{def: block form}, where
\begin{align*}
a_g &= \pi_A(g),&
k_g &= \pi_K(g),&
u_g &= k_g^{-1} a_g^{-1} g.
\end{align*}
Clearly, these elements are of the form described in Definition \ref{def: block form}, and the growth assumption in (ii) follows from the contraction-on-average assumption. We will use the irreducibility assumption to verify (iii). Let $\Hmu \subset \para$ be the Zariski closure of $\Gamma$, and we will show that $\Lie(\Hmu) \supset \Lie(U)$. Let $Q$ be the identity component of $\Hmu \cap U$. Clearly, $\Hmu$ normalizes $Q$, and by Lemma \ref{lem: verify iv}, for all $g\in \Hmu$ we have $\log(u_g) \in \Lie(Q)$ and thus $u_g\in Q$. Now let $\LL = \{\bfalpha\in \MM : \log(u_\bfalpha) \in \Lie(\Hmu)\} = \{\bfalpha \in \MM : u_\bfalpha \in Q\}$. We claim that $\LL$ is invariant under the action of $\Gamma$ on $\MM$. Indeed, if $g\in \Gamma$ and $\bfalpha\in\LL$, then $g \cdot \iota(\bfalpha) = a_g k_g u_g u_\bfalpha/AK = g (u_\bfalpha u_g) g^{-1}/AK \in Q/AK$ and thus $\rho(g)(\bfalpha) \in \LL$. Thus by the irreducibility assumption, $\LL = \MM$ and thus $\Lie(\Hmu) \supset \Lie(U)$, as required.

\item In Setup 1 we will also need to know that the assumptions of Theorem \ref{theoremequibootstrap} are satisfied for the map $\kappa = \pi_K$. That is, we need to show that $\Gamma$ acts ergodically on $(X \times \bar{K}, m_X \otimes m_{\bar{K}})$, where $\bar{K}$ is the closure of $\pi_K(\Gamma)$ and $m_{\bar{K}}$ is Haar measure on $\bar{K}$. To see this, note that the ``contracting on average'' assumption on $\mu$ implies that $\Gamma$ is an unbounded subgroup of $G$. Thus by the Howe--Moore theorem (see e.g. \cite{Zimmer}), the action of $\Gamma$ on $X$ is mixing, and hence also weakly mixing. Moreover, the action of $\Gamma$ on $(\bar{K}, m_{\bar{K}})$ (via $\kappa$) is ergodic since $\kappa(\Gamma)$ is dense in $\bar{K}$. This implies (see \cite[Proposition~2.2]{Schmidt6}) that the product action of $\Gamma$ on $X \times \bar{K}$ is ergodic.

\item In Setup 2, we need to show that $\Gamma^+$ is Zariski dense, naturally using the assumption that the IFS $\Phi$ is irreducible. First of all, by \cite[Lemma~5.15]{BenoistQuint_book}, the Zariski closure of $\Gamma^+$, which we denote by $\Hmu$, is a group. It is clear that the limit set of $\Hmu$ in the sense of Kleinian groups contains the limit set of $\Phi$ in the sense of \S\ref{sectionconf}, which by assumption is not contained in any generalized sphere $\LL \subsetneqq \Hh^{d + 1}$ (or else the smallest such sphere would be invariant under $\Phi$). Thus $\Hmu$ is a Lie subgroup of $\Isom(\Hh^{d + 1})$ with no global fixed point whose limit set (in the sense of Kleinian groups) is not contained in any nonempty generalized sphere which is properly contained in $\partial\Hh^{d + 1}$. So by \cite[Proposition~16]{Greenberg}, either $\Hmu$ is discrete or $\Hmu = \Isom(\Hh^{d + 1})$. The former case is ruled out because Zariski closed discrete sets are finite, and $\Hmu$ is infinite (e.g. because its limit set is nonempty). Thus $\Gamma^+$ is Zariski dense.
\end{itemize}

\section{Doubling measures}
\label{sectiondoubling}
In this section, we prove Theorems \ref{theoremdoublingmatrix} and
Theorem \ref{theoremkleinian1}, using results from
Part \ref{part1} and \cite{BenoistQuint7} respectively. The proofs
are very similar. They rely on the notion of a porous set:

\begin{dfn}
\label{definitionporous}
Let $Z$ be a metric space. A subset $S\subset Z$ is called
\emph{porous} if there exists $c > 0$ such that for all $0 < r \leq
1$ and for all $z\in Z$, there exists $w\in Z$ such that $B(w,c
r)\subset B(z,r)\sm S$.
\end{dfn}
\begin{lem}[{\cite[Proposition~3.4]{JJKRRS}}]
\name{lem: porous measure}
If $S\subset Z$ is porous, then $S$ has measure zero with respect to
any doubling measure $\nu$ such that $\supp(\nu) = Z$.
\end{lem}

Before beginning the proofs of Theorems \ref{theoremdoublingmatrix} and \ref{theoremkleinian1}, we will provide equivalent characterizations of when a point is badly approximable (resp. uniformly radial) in the context of Theorem \ref{theoremdoublingmatrix} (resp. Theorem \ref{theoremkleinian1}).

\begin{lem}
\label{lemmaBAequivalent}
Let the notation be as in Setup 1, and assume that $\Phi$ is strictly contracting (i.e. that $\sup_{e\in E} |\phi_e'| < 1$). Then for each $b\in \B$, we have $\pi(b)\in\BA$ if and only if the sequence $(g_{b_1^n} x_0)_{n \in \N}$ is bounded in $X$.
\end{lem}
\begin{proof}
By the Dani correspondence principle, we have $\pi(b)\in\BA$ if and only if the orbit
\[
\big(a_t u_{\pi(b)} x_0\big)_{t\geq 0}
\]
is bounded in $X$ \cite[Theorem~2.20]{Dani4}. Write $g_n = g_{b_1^n} = a_{t_n} k_n u_{\bfalpha_n}$ for some $t_n\in\R$, $k_n\in K$, and $\bfalpha_n\in\MM$. Also write $\bfbeta_n = \pi(T^n b) \in \KK$, where $T: \B \to \B$ is the shift map, and let $h_n = u_{-\bfbeta_n} a_{t_n} k_n u_{\pi(b)}$. Obviously $h_n$ and $g_n$ agree in their projections to $AK$, and on the other hand, letting them act on $\mathcal{M}$ via the isomorphism $\iota:\mathcal{M} \to \para/AK$ (and recalling the minus sign in \equ{atualpha}), we have
\[
h_n^{-1}(\bfbeta_n) = u_{\pi(b)}^{-1} k_n^{-1} a_{t_n}^{-1}(0) =
u_{\pi(b)}^{-1}(0)
= \pi(b) = \phi_{b^1_n}(\bfbeta_n) = g_n^{-1}(\bfbeta_n).
\]
So $h_n = g_n$, and thus $h_n x_0 = g_n x_0$. Since $\Phi$ is strictly contracting, the limit set $\KK$ is compact, so the sequence $(\bfbeta_n)_{n\in\N}$ is bounded. Since $K$ is also compact, this shows that the distance from $h_n x_0$ to $a_{t_n} u_{\pi(b)} x_0$ is bounded by a number independent of $n$. So since the sequence $(a_{t_n})_{n\in\N}$ has bounded gaps in $\left(a_t\right)_{t \geq 0}$, we have
\[
\begin{split}
\text{$\big(g_n x_0\big)_{n\in\N}$ is bounded} &
\;\;\Leftrightarrow\;\; \text{$\big(a_{t_n} u_{\pi(b)} x_0 \big)_{n\in
\N}$ is bounded} \\ &
\;\;\Leftrightarrow\;\; \text{$\big(a_t u_{\pi(b)} x_0 \big)_{t\geq 0}$ \hspace{0.04 in} is bounded}.
\qedhere
\end{split}
\]
\end{proof}

\begin{lem}
\label{lemmaURequivalent}
Let the notation be as in Setup 2, and assume that $\Phi$ is strictly contracting on some compact set $\FF \subset \partial\Hh^{d + 1}$. Given $b\in \B$, we have $\pi(b)\in \UR_\Lambda$ if and only if the sequence $(g_{b_1^n} x_0)_{n \in N}$ is bounded in $X$.
\end{lem}
\begin{proof}
Let $K$ be the subgroup of $G$ fixing a distinguished tangent vector at the basepoint $\zero$, so that $T^1\HHH^{d+1} \cong K \backslash G/\Lambda$. Since $K$ is compact,
\begin{align*}
&\text{$\big(g_{b_1^n}x_0\big)_{n\in\N}$ is bounded in
$X$}\\
\Leftrightarrow\;\; & \text{the image of $\big(g_{b_1^n}
\big)_{n\in\N}$ is bounded in
$K\backslash G / \Lambda$}\\
\Leftrightarrow\;\; & \text{the image of $\big(\phi_{b_n^1}\big)_{n\in\N}$ is bounded in
$\Lambda \backslash G / K$}\\
\Leftrightarrow\;\; & \text{$\big(\phi_{b_n^1}(\zero) \big)_{n\in\N}$
remains within a bounded distance of
$\Lambda(\zero)$}.
\end{align*}
So to complete the proof, we need to show that the Hausdorff distance between the sequence $\big(\phi_{b_n^1}(\zero) \big)_{n\in\N}$ and the geodesic ray $[\zero,\pi(b)]$ from $\zero$ to $\pi(b)$ is finite. Since the sequence of successive distances $\big(\dist(\phi_{b_n^1}(\zero),\phi_{b_{n + 1}^1}(\zero))\big)_{n\in\N}$ is bounded, it suffices to show that the sequence of distances $\big(\dist(\phi_{b_n^1}(\zero),[\zero, \pi(b)])\big)_{n\in\N}$ is uniformly bounded. Now for each $n$,
\begin{align*}
\dist\big(\phi_{b_n^1}(\zero),[\zero,\pi(b) ]\big)
&=
\dist\big(\zero,[\phi_{b_n^1}^{-1}(\zero),\phi_{b_n^1}^{-1}(\pi(b))]\big)
= \dist\big(\zero,[\phi_{b_n^1}^{-1}(\zero), \pi(T^n b)]\big),
\end{align*}
so we just need to show that, after taking any subsequence along which both limits exist, we have
\begin{equation}
\label{limits}
\lim_{n\to\infty} \phi_{b_n^1}^{-1}(\zero) \neq \lim_{n\to\infty}
\pi(T^n b).
\end{equation}
But the left-hand side of \eqref{limits} belongs to $\partial\Hh^{d + 1}\sm V$, where $V \subset \Hh^{d + 1}\cup \partial\Hh^{d + 1}$ is a neighborhood of $\FF$ small and regular enough so that $\zero\notin V$ and $\phi_e(V) \subset V$ for all $e\in E$. On the other hand, since $\Phi$ is strictly contracting on $\FF$, the right-hand side of \eqref{limits} is a member of $\FF$. So the two cannot be equal, which completes the proof.
\end{proof}

We are now ready to prove Theorems \ref{theoremdoublingmatrix} and
\ref{theoremkleinian1}.

\begin{proof}[Proof of Theorem \ref{theoremkleinian1}(i)]
By Lemma \ref{lemmaURequivalent}, it suffices to show that for all $b\in \B$, the sequence $(g_{b_1^n} x_0)_{n \in \N}$ is bounded in $X = G/\Lambda$. But this sequence is contained in the orbit $\Gamma^+ x_0$, which by hypothesis is finite.
\end{proof}

\begin{proof}[Proof of Theorems \ref{theoremdoublingmatrix} and
\ref{theoremkleinian1}(ii)]
Let $K_j\nearrow X$ be an exhaustion of $X$ by compact sets, and for each $j$ let
\[
S_j = \{ b\in \B : (g_{b_1^n} x_0)_{n \in \N} \subset K_j
\}.
\]
Then by Lemma \ref{lemmaBAequivalent} (resp. Lemma \ref{lemmaURequivalent}), the set of badly approximable points (resp. uniformly radial points) can be written as $\bigcup_{j\in\N} \pi(S_j)$. By Lemma \ref{lem: porous measure}, in order to complete the proof, it suffices to show that for all $j$, the set $\pi(S_j)$ is porous in $\KK$.

By contradiction, suppose that there exists $j$ such that $\pi(S_j)$ is not porous in $\KK$. Then for all $m\in\N$, there exist $z_m\in \KK$ and $r_m \in (0,1)$ such that for all $w\in \KK$ such that $B(w,r_m/m) \subset B(z_m,r_m)$, we have $B(w,r_m/m)\cap \pi(S_j) \neq \emptyset$. Write $z_m = \pi(b)$ for some $b\in \B$. Let $n$ be the smallest integer such that $\phi_{b_n^1}(\KK) \subset B(z_m,r_m/2)$. Now since $\Phi$ satisfies the open set condition, by \cite{PRSS} it also satisfies the strong open set condition, i.e. there exists an open set $U$ such that $(\phi_e(U))_{e\in E}$ is a disjoint collection of subsets of $U$, and $U\cap \KK \neq \emptyset$. Fix $z_0\in U\cap \KK$, and let
\[
\lambda = \min_{e\in E} \inf |\phi_e'| > 0.
\]
We claim that there exists $c > 0$ such that for all $k\in\N$ and $\altb\in E^k$, we have
\begin{equation}\label{eq: can ensure that}
B(\phi_{\altb_k^1 b_n^1}
(z_0),c\lambda^k r_m)
\subset \phi_{\altb_k^1 b_n^1}
(U), \ \ \text{ where }
\phi_{\altb_k^1 b_n^1} = \phi_{b_n^1} \circ \phi_{\altb_k^1}.
\end{equation}
Indeed, an easy induction argument shows that
\[
B(\phi_{\altb_k^1}(z_0), \lambda^k \dist(z_0,\partial U)) \subset \phi_{\altb_k^1}(U),
\]
and the choice of $n$ ensures that the contraction rate of the map
$\phi_{b^1_n}$ is on the order of $r_m$. Combining these facts with
the bounded distortion property demonstrates \eqref{eq: can ensure
  that}. 

It follows that if $c\lambda^{k} \geq 1/m$, then for all $\altb \in
E^k$, we have $\phi_{\altb_k^1 b_n^1}(U)\cap \pi(S_j) \neq
\emptyset$. Thus there exists $b'\in S_j$ such that $\pi(b') \in
\phi_{\altb_k^1 b_n^1}(U)$. The defining property of $U$ implies that
$b_1^n \altb$ is an initial segment of $b'$, i.e. that $b' = b_1^n
\altb \altb' $ for some $\altb'\in \B$. In particular, we have 
\begin{equation}
\label{foralltau}
g_{\altb_1^i} x_m \in K_j \text{ for all $\altb \in E^k$ and $i = 0,\ldots,k$},
\end{equation}
where $x_m = g_{b_1^n} x_0$. In particular $x_m\in K_j$ for all $m$, so we can pass to a subsequence along which we have $x_m \dashrightarrow y \in K_j$. Taking the limit of \eqref{foralltau} along this subsequence shows that for all $\altb\in E^*$, we have $g_\altb y \in K_j$. In particular, the orbit $\Gamma^+ y$ is bounded. In Setup 1 this gives a contradiction to Theorem \ref{theorempart1}(i). In Setup 2, in view of Theorem \ref{theoremBQ}(i) and Remark \ref{remarkBQ}, it follows that the set $\Gamma y$ is finite. But then the finite index subgroup $\Stab_\Gamma(y) \leq \Gamma$ is entirely contained in $g \Lambda g^{-1}$, where $y = g x_0$. This contradicts the hypothesis of Theorem \ref{theoremkleinian1}(ii).
\end{proof}

\section{Bernoulli measures}
\label{sectionbernoulli}
In this section we prove Theorems \ref{theoremgtequi} and
\ref{theoremkleinian2}, using Theorems \ref{theorempart1},
\ref{theoremBQ}, respectively, as well as Theorem \ref{theoremequibootstrap}.

\begin{proof}[Proof of Theorem \ref{theoremgtequi}]
Recall that $\bar{\B} = E^\Z$, and define $\pi_+: \bar{\B} \to \MM$ by $\pi_+(b) =\pi(b_1^\infty)$. By the definition of a general algebraic self-similar measure, it suffices to show that for $\bar{\beta}$-a.e. $b \in \bar{\B}$, the trajectory $\{a_t u_{\pi_+(b)}x_0 : t\geq 0\}$ is equidistributed in $X$ with respect to $m_X$. By Theorem \ref{theorempart1}(ii), for $\beta$-a.e. $b\in \B$ the orbit $(g_{b_1^n} x_0)_{n \in \N}$ is equidistributed. We will apply Theorem \ref{theoremequibootstrap}. Let $\kappa = \pi_K, \, k_e = \kappa(e)$ be as in \S \ref{subsec: fundamental}, let $Y = E\times \MM$, and define $f: \bar{\B} \to Y$ by $f(b) = (b_0,\pi_+(b))$. Then for $\bar{\beta}$-a.e. $b \in \bar{\B}$, the sequence
\eq{eq: rwandmore}{
\big(
g_{b_1^n} x_0,k_{b_1^n} ,f(T^n b)\big)_{n\in\N}
}
is equidistributed with respect to the measure $m_X\otimes m_{\bar K}\otimes f_*\bar \beta$, where $m_{\bar K}$ is the Haar measure on $\bar K$, the closure of $\kappa(\Gamma)$. Note that $f_*\bar{\beta} = \mu\otimes\nu$, where $\nu = \pi_*\beta$. Now consider the map $f_2:X\times K\times Y \to X\times E$ defined by the formula
\[
f_2(x,k,(e,\bfalpha)) = (k^{-1} u_\bfalpha x,e).
\]
Since $f_2$ is continuous, the image of \eqref{eq: rwandmore} under $f_2$, i.e. the sequence
\begin{equation}
\label{imagef}
(x_n,b_n)_{n \in \N}, \ \ \ \text{ where } \ x_n= k_{b_1^n}^{-1} u_{\pi_+(T^n b)} g_{b_1^n} x_0,
\end{equation}
is equidistributed in $X\times E$ with respect to the measure $(f_2)_*[m_X\otimes m_{\bar K}\otimes f_*\bar \beta] = m_X\otimes\mu$.

Write $g_n = g_{b_1^n} = k_na_{t_n}u_{\bfalpha_n}$. As in the proof of Lemma \ref{lemmaBAequivalent}, we find that $g_n = u_{-\pi_+(T^n b)} a_{t_n} k_n u_{\pi_+(b)}$ and thus
\begin{equation}\label{eq: and thus}
x_n = k_n^{-1} u_{\pi_+(T^n b)} g_n x_0 = a_{t_n} u_{\pi_+(b)} x_0
\end{equation}
for all $n\in\N$.

For each $e\in E$, let $t_e\in\R$ be chosen so that $\pi_A(g_e) =
a_{t_e}$. Since $\pi_A$ is a homomorphism, we have $t_n = t_{n - 1} +
t_{b_n}$ for all $n\in\N$. Now let $F:X\to\R$ be a bounded continuous
function. Then the function $F':X\times E\to\R$ defined by the formula
\[
F'(x,e) = \int_{-t_e}^0 F(a_t x) \;\dee t
\]
is also a bounded continuous function. Here we use the convention that
if $b < a$, then $\int_a^b F(a_t x) \;\dee t = -\int_b^a F(a_t
x)\;\dee t$. Since \eqref{imagef} is equidistributed, plugging in
\eqref{eq: and thus} we find that
\begin{align*}
\int F'\;\dee(m_X\otimes\mu) &= \lim_{n\to\infty} \frac{1}{n}
\sum_{i = 1}^n F'(a_{t_i}
u_{\pi_+(b)} x_0, b_i)\\
&= \lim_{n\to\infty} \frac{1}{n} \sum_{i = 1}^n \int_{t_{i - 1}}^{t_i}
F(a_t u_{\pi_+(b)} x_0) \;\dee t\\
&= \left(\int t_e\;\dee \mu(e)\right) \lim_{n\to\infty} \frac{1}{t_n}
\int_0^{t_n} F(a_t u_{\pi_+(b)} x_0) \;\dee t
\end{align*}
(where in passing to the last line we used the special case of the
first two lines where $F \equiv 1$ and $F'(x,e) = t_e$). On the other hand,
\[
\int F'\;\dee(m_X\otimes\mu) = \int \left(t_e \int
F\;\dee m_X\right)\;\dee \mu(e) = \left(\int
t_e\;\dee \mu(e)\right)\left(\int F\;\dee m_X\right).
\]
Since $t_n \to \infty$ and the gaps $t_{n+1}-t_n \; (n\in\N)$ are bounded, it
follows that $\frac{1}{T}
\int_0^T F(a_t u_{\pi_+(b)} x_0) \;\dee t \to \int F\;\dee m_X$,
i.e. that $(a_t u_{\pi_+(b)} x_0)_{t \geq 0}$ is equidistributed with respect
to $m_X$.
\end{proof}

\begin{proof}[Proof of Theorem \ref{theoremkleinian2}]
Let $x = x_0$, and let $\HBQ \subset G$ and $\nu_x$ be as in Theorem \ref{theoremBQ}. Since by assumption $\Gamma$ is not virtually contained in $\Lambda = \Stab_G(x_0)$, Remark \ref{remarkBQ} shows that $\nu_x = m_X$. So by Theorem \ref{theoremBQ}(ii), for $\beta$-a.e. $b\in \B$ the orbit \eqref{randomwalk} is equidistributed. As in the previous proof, we want to apply Theorem \ref{theoremequibootstrap}. Let $\pi_+,\pi_-: \bar{\B} \to \partial\Hh^{d + 1}$ be defined by the formulas
\begin{align*}
\pi_+(b) &= \lim_{n\to\infty} \phi_{b_n^1}(\zero)\\
\pi_-(b) &= \lim_{n\to -\infty} \phi_{b_n^1}(\zero),
\end{align*}
with the convention that $\phi_{b_n^1} = \phi_{b_0^{n + 1}}^{-1}$ whenever $n \leq 0$.

Let $b\in \bar\B$ be a random variable with distribution $\bar\beta$. Then $\pi_+(b)$ and $\pi_-(b)$ are independent random variables with atom-free distributions, and thus $\pi_+(b) \neq \pi_-(b)$ almost surely. Let $\gamma(b)$ denote the bi-infinite geodesic from $\pi_-(b)$ to $\pi_+(b)$, and for each $n\in\Z$ let $v_n(b) \in T^1 \Hh^{d + 1} \cong K \backslash G$ be the unit tangent vector whose basepoint is the projection of $\phi_{b_n^1}(\zero)$ to $\gamma(b)$ and which is parallel to $\gamma(b)$, pointing in the direction of $\pi_+(b)$. Note that $v_n(b) = \phi_{b_1}(v_{n - 1}(T b))$. Equivalently, $v_n(b) = v_{n - 1}(T b) g_{b_1}$, where now we are thinking of $v_n(b)$ and $v_{n - 1}(T b)$ as elements of $K \backslash G$. Let $\kappa:G\to K = \{e\}$ be the trivial homomorphism, let $Y = T^1 \Hh^{d + 1} \times T^1 \Hh^{d + 1}$, and let $f(b) = (v_0(b),v_1(b))$. Then by Theorem \ref{theoremequibootstrap}, the sequence
\eq{eq: rwandmore1}{
\big(g_{b_1^n}x, v_0(T^n b), v_1(T^n b) \big)_{n
\in \N},
}
is almost surely equidistributed with respect to $m_X \otimes f_* \bar{\beta}$. Let $F:K\backslash G / \Lambda \to \R$ be a bounded continuous function, and let $T^+\gamma(b)$ be the space of unit vectors tangent to $\gamma(b)$ and pointing in the direction of $\pi_+(b)$. We need to show that
\begin{equation}
\label{NTSkleinian2}
\frac{1}{v_1 - v_0} \int_{v_0}^{v_1} F(w x_0) \;\dee w \tendsto{T^+\gamma(b) \ni v_1 \to \pi_+(b)} \int F \;\dee m_X \;\text{ for all $v_0\in T^+\gamma(b)$},
\end{equation}
where the left-hand integral is taken over all $w\in T^+\gamma(b)$ between $v_0$ and $v_1$, with respect to the pushfoward of Lebesgue measure on $\R$ under the differential of any unit speed parameterization of $\gamma(b)$. The expression $v_1 - v_0$ is interpreted as the distance between the basepoints of $v_1$ and $v_0$. In what follows, it may happen that $v_1 < v_0$ in the sense that the basepoint of $v_0$ is closer to $\pi_+(b)$ than $v_1$ is, in which case we think of $v_1 - v_0$ as a negative number and we use the convention $\int_{v_0}^{v_1} h(w) \;\dee w \df -\int_{v_1}^{v_0} h(w) \;\dee w$ for any function $h$.

To demonstrate \eqref{NTSkleinian2}, first observe that
\begin{align*}
& \frac{1}{n}\int_{v_0(b)}^{v_n(b)} F(w x_0) \;\dee w \\
= & \frac{1}{n}\sum_{i = 0}^{n - 1} \int_{v_0(T^i b)
g_{b_1^i}}^{v_1(T^i b) g_{b_1^i}} F(w x_0) \;\dee
w\\
= & \frac{1}{n}\sum_{i = 0}^{n - 1} \int_{v_0(T^i
b)}^{v_1(T^i b)} F(w g_{b_1^i} x_0) \;\dee w\\
\tendsto{n \to \infty} & \iint \int_{v_0}^{v_1} F(w x) \;\dee w \;\dee m_X(x)
\;\dee f_*\bar{\beta}(v_0,v_1)\\
= & \left(\int (v_1 - v_0) \;\dee f_*\bar{\beta}(v_0,v_1)\right)
\left(\int F \;\dee m_X \right).
\end{align*}
Note that the last two lines make sense because for all $(v_0,v_1)\in \supp(f_*\bar{\beta})$, the tangent vectors $v_0$ and $v_1$ span the same geodesic. To summarize, we have
\begin{equation}
\label{summarywithF}
\frac{1}{n}\int_{v_0(b)}^{v_n(b)} F(w x_0) \;\dee w \tendsto{n\to\infty} c \int F \;\dee m_X,
\end{equation}
where $c \in \R$ is a constant independent of $F$.

By \cite[Theorems~1.2 and 1.3]{MaherTiozzo}, if $F \equiv 1$ then the left-hand side of \eqref{summarywithF} converges to a positive number almost surely. This implies that $c > 0$ and thus we can divide \eqref{summarywithF} by its special case that occurs when $F \equiv 1$, yielding the limit
\[
\frac{1}{v_n(b) - v_0(b)}\int_{v_0(b)}^{v_n(b)}
F(w x_0) \;\dee w \tendsto{n\to\infty} \int F \;\dee
m_X.
\]
Since $v_n(b) \to \pi_+(b)$ and $(v_{n+1}(b)-v_n(b))_{n\in\N}$ is
bounded, this implies that \eqref{NTSkleinian2} holds, i.e. that the
directed segment $[v_0(b),\pi_+(b)]$ of the bi-infinite geodesic
$\gamma(b)$ is equidistributed in $K \backslash G / \Lambda$. Since
any two geodesic rays ending at the same point have the same
equidistribution properties, this completes the proof. 
\end{proof}

\section{Equidistribution under the Gauss map}
\label{sectionCFequi}
In this section we prove the following result. The result may be
well-known but we were unable to find a suitable reference. Combining
it with 
Theorem \ref{theoremgtequi} yields Theorem \ref{theoremCFequi} as an
immediate corollary.

\begin{thm}
\label{theoremgtCFrelation}
Fix $\alpha\in (0,1)$, and suppose that the orbit $(a_t u_\alpha
x_0)_{t \geq 0}$ is equidistributed in $X = G/\Lambda =
\PGL_2(\R)/\PGL_2(\Z)$ with respect to Haar measure. Then the orbit
$(\mathcal{G}^n \alpha)_{n\in\N}$ is equidistributed with respect to
Gauss measure, where $\mathcal{G}$ is the Gauss map. Equivalently, if $b=(b_1, b_2,
\ldots)$ is the sequence of continued fraction coefficients of $\alpha
= [0;b_1,b_2,\ldots]$, then the sequence $(T^n b)_{n\in\N}$ is equidistributed in
$\N^\N$ with respect to Gauss measure, where $T$ is the shift map. 
\end{thm}

The converse to Theorem \ref{theoremgtCFrelation} is not true:

\begin{example}
Let $b\in\N^\N$ be chosen so that the sequence $(T^n b)_{n\in\N}$ is
equidistributed with respect to Gauss measure, and let $S \subset \N$
be an infinite set of density zero. Then if $\altb\in\N^\N$ is chosen
so that $\altb_n = b_n$ for all $n\in\N\sm S$, then the sequence $(T^n
b)_{n\in\N}$ is also equidistributed with respect to Gauss
measure. However, by choosing the integers $\altb_n$ ($n\in S$) large
enough, it is possible to guarantee an arbitrary degree of
approximability for the encoded point $\alpha =
[0;\altb_1,\altb_2,\ldots]$. In particular, $\altb$ may be chosen so
that $\alpha$ is very well approximable, in which case it is not hard
to show that the orbit $(a_t u_\alpha x_0)_{t \geq 0}$ cannot be
equidistributed in $X$ with respect to any measure (due to escape of
mass). 
\end{example}

The idea of the proof of Theorem \ref{theoremgtCFrelation} is to
define a map $f:X\to\N^\N$ which is continuous outside a set of
measure zero, such that the image of the orbit $(a_t u_\alpha x_0)_{t
  \geq 0}$ is the orbit $(T^n b)_{n\in\N}$. To define this set, we use
the fact that elements of $X$ can be interpreted as lattices in $\R^2$
via the map $g x_0 \mapsto g(\Z^2)$. In what follows we let $L_x$
denote the lattice corresponding to a point $x\in X$. 

We define a \emph{best approximation} in a lattice $L\subset \R^2$ to
be a point $(\xi_1,\xi_2)\in L \sm\{0\}$ with the following property:
there is no point $(\gamma_1,\gamma_2)\in L\sm\{0,\pm (\xi_1,\xi_2)\}$
such that $|\gamma_1| \leq |\xi_1|$ and $|\gamma_2| \leq |\xi_2|$. It
is well-known that if $\alpha\in\R$, then the set of best
approximations $(\xi_1,\xi_2)$ in the lattice $u_\alpha \Z^2$ that
satisfy $\xi_2>1$ is precisely the set $\{u_\alpha(p_n,q_n) :
n\in\N\}$, where $(p_n/q_n)_{n\in\N}$ is the sequence of convergents
of $\alpha$ \cite[Theorems~16 and 17]{Khinchin_book}. Also, it is easy
to see using Minkowski's convex body theorem that the set of best
approximations in $L$ with second coordinate $\geq 1$  is
infinite unless $L$ has a nontrivial 
intersection with $\{0\} \times \R$. Accordingly we let $X'$ denote
the set of  points $x\in X$ such that $L_x \cap (\{0\} \times \R) =
\{0\}$. Let $Y$ denote the set of increasing sequences in
$[1,\infty)$ which begin with 1 and have no finite accumulation
points, equipped with the Tychonoff topology. Define a function $f_1:X' \to Y$ by
letting $f_1(x)$ denote the sequence of
numbers consisting of the elements of the set 
\[
\{\xi_2 \geq 1 : (\xi_1,\xi_2)\in L_x \text{ is a best approximation}\}
\]
listed in ascending order and rescaled by a homothety so that they
begin with 1. Using continued fractions
(see e.g. \cite[Chapter 10]{Karpenkov}),
it is not hard to show that for each $x\in X'$, the sequence $f_1(x) = (y_1,y_2,\ldots)$ satisfies a recursive equation of the form $y_{n+1} = a_n
y_n + y_{n-1}$ with $a_n \in \N$. Note
that $X'$ is an $\{a_t\}$-invariant 
set of full $m_X$-measure, and for all $t \geq 0$ and $x\in X$, there
exists $n\geq 0$ such that $f_1(a_t x) = T^n \circ f_1(x)$, where
$T:Y\to Y$ is the shift map. (More precisely, $n$ is the smallest
number such that the $n$th coordinate of $f_1(x)$ is at least
$e^t$.) Also note that the set of discontinuities of $f_1$ is 
contained in the set $\{x\in X': L_x\cap (\R\times\{0,1\}) \neq
\{0\}\}$, which is a set of $m_X$-measure zero.

\begin{lem}
\label{lemmaexistsmeasure}
For all $x\in X'$ such that the trajectory $(a_t x)_{t\geq 0}$ is
equidistributed in $X$ with respect to the measure $m_X$, the orbit 
\begin{equation}
\label{trajectory3}
\big(T^n f_1(x)\big)_{n\in\N}
\end{equation}
is equidistributed in $Y$, with respect to some probability measure
$\mu$ which is independent of $x$. 
\end{lem}
\begin{proof}
Indeed, let $F:Y\to\R$ be a bounded continuous function, and define
$F':Y\to\R$ and $h : X' \to \R$ by the formulas 
\[
F'(y_1,y_2,\ldots) = \sum_{\substack{i\in\N \\ 1\leq y_i < e}} F(y_i,y_{i +
  1},\ldots), \ \ h = F' \circ f_1.
\] 
(Here $\log(e) = 1$.) When $(y_1,y_2,\ldots) \in F_1(X')$, the
recursive equation $y_{n+1} = a_n 
y_n + y_{n-1}$ ($a_n \geq 1$) guarantees that the number of summands
in this series is uniformly bounded (in fact $\leq 3$), and therefore
$h$ is bounded.  

Write $f_1(x) = (y_1,y_2,\ldots)$. Then for all $i\in\N$ and $t\geq
0$, $F\circ T^{i-1} f_1(x) = F(y_i,y_{i+1},\ldots)$ is a term in $F'\circ f_1(a_t x)$ if and
only if $\log(y_i) - 1 < t \leq \log(y_i)$. For all $n\geq 0$, we have 
\begin{align*}
\sum_{i = 1}^n F\circ T^{i-1} f_1(x)
&= \sum_{i = 1}^n \int_{\log(y_i) - 1}^{\log(y_i)} F\circ T^{i-1} f_1(x) \;\dee t\\
&= \int_0^{\log(y_n)} F'\circ f_1(a_t x) \;\dee t + O(1),
\end{align*}
so
\begin{equation}
\label{equalsT}
\lim_{n\to\infty} \frac{1}{\log(y_n)} \sum_{i = 1}^n F\circ T^{i-1} f_1(x)
= \lim_{T\to\infty} \frac{1}{T} \int_0^T F'\circ f_1(a_t x) \;\dee t
\end{equation}
assuming the right-hand side exists.

The set of discontinuities of $h$ is
contained in the set $\{x\in X': L_x\cap
(\{0,1,e\}\times\R)\neq\{0\}\}$, which is of $m_X$-measure zero. 
Thus by the Portmanteau theorem, if $\nu_n\to \nu$ with respect to the
weak-* topology, then $\int h\;\dee\nu_n \to \int
h\;\dee\nu$. Thus,
letting $\nu_n = \frac 1n\int_0^n \delta_{a_t x} \;\dee t$ in the
Portmanteau theorem and using the equidistribution assumption shows that the right-hand side of \eqref{equalsT}
converges to $\int F'\circ f_1 \;\dee m_X$. Rearranging yields  
\begin{equation}
\label{nearlyequidistributed}
\begin{split}
\lim_{n\to\infty} \frac{1}{n}&\sum_{i = 1}^n F\circ T^{i-1} f_1(x)
= \left(\lim_{n\to\infty} \frac{\log(y_n)}{n}\right)
\left(\int F'\circ f_1 \;\dee m_X\right)\\
&\text{for all $x$ such that $(a_t x)_{t\geq 0}$ is equidistributed.}
\end{split}
\end{equation}
As of yet, we do not claim that the limits exist, but only that the left-hand limit exists if and only if the right-hand limit does.

Setting $F \equiv 1$ in \eqref{nearlyequidistributed}, we see that the limit $\lim_{n\to\infty} \frac{\log(y_n)}{n}$ exists and is independent of $x$. Write $\lim_{n\to\infty} \frac{\log(y_n)}{n} = c$ for some constant $c > 0$. Then we have
\[
\lim_{n\to\infty} \frac{1}{n}\sum_{i = 1}^n F\circ T^{i-1} f_1(x)
= \int F \;\dee \mu \df c \int F'\circ f_1 \;\dee m_X
\]
for all $x$ such that $(a_t x)_{t\geq 0}$ is equidistributed. This shows that the sequence $(T^n f_1(x))_{n\geq 0}$ is equidistribuited with respect to $\mu$, completing the proof.
\end{proof}

\begin{proof}[Proof of Theorem \ref{theoremgtCFrelation}]
Define $f_2:Y\to\N^\N$ by letting
\[
f_2(y_1,y_2,\ldots) = \left(\lfloor y_{n + 1}/y_n\rfloor\right)_{n\in\N}
\]
Then the set of discontinuities of $f_2$ is contained in the set $\{(y_1,y_2,\ldots) : y_{n + 1}/y_n\in\N \text{ for some $n$}\}$, which is of measure zero with respect to the probability measure $\mu$ defined in Lemma \ref{lemmaexistsmeasure}. Thus by the Portmanteau theorem, the image of every equidistributed sequence in $Y$ under $f_2$ is equidistributed in $\N^\N$ with respect to the measure $\nu = (f_2)_* \mu$. On the other hand, if $\alpha\in (0,1)$, then the sequence $f_2\circ f_1(u_\alpha x_0)$ is precisely the sequence of partial quotients of the continued fraction expansion of $\alpha$, except that the first partial quotient is omitted. Thus
\begin{equation}
\label{equinu}
\begin{gathered}
\text{the sequence $(T^n(b))_{n\in\N}$ is equidistributed with respect to $\nu$}\\
\text{for all $\alpha = [0;b_1,b_2,\ldots]$ such that $(a_t u_\alpha x_0)_{t\geq 0}$ is equidistributed}\\
\text{with respect to $m_X$}.
\end{gathered}
\end{equation}
A standard computation shows that whenever $x_1,x_2\in X$ satisfy $x_2
= g x_1$ for some lower triangular matrix $g\in G$, then the
trajectory $(a_t x_1)_{t \geq 0}$ is equidistributed with respect to
$m_X$ if and only if $(a_t x_2)_{t \geq 0}$ is equidistributed with
respect to $m_X$. Now if $S \subset \R$ is any set of positive
Lebesgue measure, then the set $\{g u_\alpha x_0 : \alpha \in S, g
\text{ lower triangular} \}$ has positive $m_X$-measure. Thus, for
Lebesgue-a.e. $\alpha\in\R$, the trajectory $(a_t u_\alpha x_0)_{t
  \geq 0}$ is equidistributed with respect to $m_X$. On the other
hand, for Lebesgue-a.e. $\alpha = [0;b_1,b_2,\ldots]\in\R$, the orbit
$(T^n(b))_{n\in\N}$ is equidistributed with respect to the Gauss
measure. Thus \eqref{equinu} implies that $\nu$ is equal to Gauss
measure. Plugging this equality into \eqref{equinu} completes the
proof of Theorem \ref{theoremgtCFrelation}. 
\end{proof}

\bibliographystyle{amsplain}

\bibliography{bibliography}

\end{document}